\documentclass[12pt]{article}
\usepackage[pdftex, bookmarksopen=true, bookmarks=true, unicode, setpagesize]{hyperref}
\hypersetup{colorlinks=true, linkcolor=black, citecolor=black}

\usepackage{amssymb, amsmath, amsthm}

\numberwithin{equation}{section}
\usepackage{xcolor}

\allowdisplaybreaks
\usepackage{cite}

\newcommand\R{\mathbb{R}}

\newtheorem{theorem}{Theorem}[section]
\newtheorem{corollary}[theorem]{Corollary}
\newtheorem{lemma}[theorem]{Lemma}
\newtheorem{proposition}[theorem]{Proposition}

\theoremstyle{remark}

\theoremstyle{remark}

\theoremstyle{remark}
\newtheorem{remark}[theorem]{Remark}

\newcommand{\indlim}{\operatornamewithlimits{ind\, lim}}
\newcommand{\prlim}{\operatornamewithlimits{proj\, lim}}

\begin{document}

\vspace{-20mm}
\begin{center}{\Large \bf 
Lie structures of the group of Sheffer operators}
\end{center}

{\large Dmitri Finkelshtein}\\ Department of Mathematics, Swansea University, Bay Campus, Swansea SA1 8EN, U.K.;
e-mail: \texttt{d.l.finkelshtein@swansea.ac.uk}\vspace{2mm}

{\large Eugene Lytvynov}\\ Department of Mathematics, Swansea University, Bay Campus, Swansea SA1 8EN, U.K.;
e-mail: \texttt{e.lytvynov@swansea.ac.uk}\vspace{2mm}

{\large Maria Jo\~{a}o Oliveira}\\ DCeT, Universidade Aberta, 
 1269-001 Lisbon, Portugal; Center for Mathematical Studies, University of Lisbon, 1749-016 Lisbon, Portugal;\\
e-mail: \texttt{mjoliveira@ciencias.ulisboa.pt}\vspace{2mm}

{\small
\begin{center}
{\bf Abstract}
 \end{center}
 
\noindent
 Let $\Phi$ be an  (LB)-space over $\mathbb F=\mathbb R$ or $\mathbb C$, and let $\Phi'$ be the dual space of~$\Phi$. We study the set $\mathbb S(\Phi)$ of Sheffer operators acting in polynomials on $\Phi'$. We prove that $\mathbb S(\Phi)$ is a group for the usual product of operators. We equip $\mathbb S(\Phi)$ with a natural topology which makes $\mathbb S(\Phi)$ into an infinite-dimensional manifold with a global parametrization. We show that $\mathbb S(\Phi)$ is an infinite-dimensional, regular Lie group, and provide an explicit description of the Lie algebra of $\mathbb S(\Phi)$, including an explicit form of the Lie bracket on it. Our main results are new even in the one-dimensional case, $\Phi=\mathbb{F}$. Furthermore, our results lead to improved understanding of the Lie algebra of the Riordan group, cf.\ Cheon, Luz\'on, Mor\'on, Prieto-Martinez, {\it Adv. Math.} 319 (2017) 522--566.

 } \vspace{2mm}

 \vspace{2mm}


\section{Introduction}
Let $(p_k(z))_{k\in\mathbb N_0}$ be a monic polynomial sequence on $\mathbb F=\R$ or $\mathbb C$. One says that $(p_k(z))_{k\in\mathbb N_0}$ is a Sheffer sequence if its (exponential) generating function has the form $\sum_{k=0}^\infty \frac{\xi^k}{k!}\, p_k(z)=\exp[z B(\xi)]A(\xi)$, where $A(\xi)=1+\sum_{k=1}^\infty a_k\xi^k$ and $B(\xi)=\xi+\sum_{k=2}^\infty b_k\xi^k$ are formal power series in $\xi\in\mathbb F$. The class of Sheffer sequences includes sequences of binomial type, or just binomial sequences (for which $A(\xi)=1$), and Appell sequences (for which $B(\xi)=\xi$). The classification of all orthogonal Sheffer sequences was given by Meixner \cite{Meixner}.

 The theory studying Sheffer sequences is called (modern) umbral calculus. After a long period when umbral calculus was used for formal calculations, the theory became rigorous in the 1970s due to the seminal 
 works of G.-C.~Rota, S. Roman and their co-authors, e.g.\ \cite{RomanRota, RotaKahanerOdlyzko}, see also the monographs \cite{C, Roman}. Umbral calculus found applications in combinatorics, theory of special functions, approximation theory, probability and statistics, topology and physics, see e.g.\ the survey paper \cite{DiBL} and the references therein.

 Denote by $\mathbb S(\mathbb F)$, $\mathbb A(\mathbb F)$ and $\mathbb B(\mathbb F)$ the sets of Sheffer, Appell and binomial sequences, respectively. A central object of study of umbral calculus is the umbral composition, which 
equips $\mathbb S(\mathbb F)$, $\mathbb A(\mathbb F)$ and $\mathbb B(\mathbb F)$ with a group structure. Let us briefly explain how the umbral composition is defined. Denote by $\mathcal P(\mathbb F)$ the vector space of polynomials on~$\mathbb F$. Any monic polynomial sequence $(p_k(z))_{k\in\mathbb N_0}$ can be identified with a linear operator $P$ in $\mathcal P(\mathbb F)$ that satisfies $Pz^k=p_k(z)=z^{k}+\sum_{i=0}^{k-1} p_{ik}\, z^i$ for all $k\in\mathbb N_0$. Thus, in terms of the basis $(z^k)_{k\in\mathbb N_0}$ in $\mathcal P(\mathbb F)$, the operator $P$ can be represented as an infinite upper unitriangular matrix $P=[p_{ik}]_{i, k\in\mathbb N_0}$ (with $p_{kk}=1$ and $p_{ik}=0$ for all $i>k\ge0$). Let $(q_k(z))_{k\in\mathbb N_0}$ be another monic polynomial sequence whose corresponding matrix is $Q=[q_{ik}]_{i, k\in\mathbb N_0}$. Then the matrix of the product of the linear operators~$P$ and~$Q$, denoted by $R$, is equal to the product of the matrices $P$ and $Q$, and the operator $R$ determines a monic polynomial sequence $(r_k(z))_{k\in\mathbb N_0}$. By definition, the umbral composition of polynomial sequences $(p_k(z))_{k\in\mathbb N_0}$ and $(q_k(z))_{k\in\mathbb N_0}$ is $(r_k(z))_{k\in\mathbb N_0}$. Thus, 
$$r_k(z)=\sum_{i=0}^k q_{ik}p_i(z)=z^k+\sum_{i=0}^{k-1}\big(\sum_{j=i}^kp_{ij}q_{jk}\big)z^i.$$
Then $\mathbb S(\mathbb F)$, equipped with the umbral composition, is a group, $\mathbb A(\mathbb F)$ is an abelian normal subgroup of $\mathbb S(\mathbb F)$, $\mathbb B(\mathbb F)$ is a subgroup of $\mathbb S(\mathbb F)$, and $\mathbb S(\mathbb F)$ is the semidirect product of $\mathbb A(\mathbb F)$ and $\mathbb B(\mathbb F)$, denoted by $\mathbb S(\mathbb F)=\mathbb A(\mathbb F)\rtimes\mathbb B(\mathbb F)$, see e.g.\ \cite[Chapter~III, Sections~4, 5]{Roman}.

For a Sheffer sequence $(p_k(z))_{k\in\mathbb N_0}$, the corresponding operator $P$ is called a Sheffer operator. In the special case where $(p_k(z))_{k=0}^\infty$ is an Appell sequence or a binomial sequence, $P$ is called an Appell operator or an umbral operator, respectively. Below, with an abuse of notation, we will denote by $\mathbb S(\mathbb F)$, $\mathbb A(\mathbb F)$ and $\mathbb B(\mathbb F)$ also the groups of Sheffer, Appell and umbral operators, respectively. 

The umbral composition of Sheffer sequences (equivalently product of Sheffer operators) can be explicitly described with the help of their generating functions. Denote by $\mathcal F_0(\mathbb F)$ and $\mathcal F_1(\mathbb F)$ the sets of formal power series of the form $A(\xi)=1+\sum_{k=1}^\infty a_k\xi^k$ and $B(\xi)=\xi+\sum_{k=2}^\infty b_k\xi^k$, respectively. The $\mathcal F_0(\mathbb F)$ is an abelian 
group for multiplication of formal power series. The~$\mathcal F_1(\mathbb F)$ is a group for substitution of one formal power series into the other, i.e., for $B^{(1)}(\xi), B^{(2)}(\xi)\in \mathcal F_1(\mathbb F)$ their group product is given by $B^{(1)}(\xi)\circ B^{(2)}(\xi)=B^{(1)}(B^{(2)}(\xi))$. The group $\mathcal F_1(\mathbb F)$ was introduced by Jennings \cite{Jennings} in the 1950s and has been actively studied since that. We refer to the paper \cite{Babenko} which reviews numerous algebraic, geometric and topological problems related to this group. 
 
Next, consider the group $(\mathcal S(\mathbb F), \ast)$, where $\mathcal S(\mathbb F)=\mathcal F_0(\mathbb F)\times \mathcal F_1(\mathbb F)$ and the group product is defined by 
$$(A^{(1)}(\xi), B^{(1)}(\xi))\ast(A^{(2)}(\xi), B^{(2)}(\xi))=(A^{(1)}(B^{(2)}(\xi))A^{(2)}(\xi), B^{(1)}(B^{(2)}(\xi))).$$
The $\mathcal S(\mathbb F)$ is called the group of (unipotent) substitutions with
pre-function, see e.g.\ \cite{Poinsot} and the references therein. Then $\mathcal F_0(\mathbb F)$ and
$\mathcal F_1(\mathbb F)$ can be naturally identified as subgroups of $\mathcal S(\mathbb F)$ and furthermore $\mathcal S(\mathbb F)=\mathcal F_0(\mathbb F)\rtimes\mathcal F_1(\mathbb F)$. Define a map 
$\mathcal I:\mathcal S(\mathbb F)\to \mathbb S(\mathbb F)$ by $\mathcal I(A(\xi), B(\xi))=P$, where $P\in\mathbb S(\mathbb F)$ has the generating function $\exp[z B(\xi)]A(\xi)$. Then $\mathcal I$ is a group isomorphism, and furthermore $\mathcal I(\mathcal F_0(\mathbb F))=\mathbb A(\mathbb F)$ and $\mathcal I(\mathcal F_1(\mathbb F))=\mathbb B(\mathbb F)$, see e.g.\ \cite[Chapter~III, Section~5]{Roman}.

Motivated by problems of enumerative combinatorics, Shapiro et al.\ \cite{Shapiro} (see also \cite{ShapiroBook}) introduced the concept of Riordan matrices (arrays), which form a group under the matrix multiplication. A Riordan matrix is an infinite lower unitriangular matrix $R=[r_{ik}]_{i, k\in\mathbb N_0}$ ($r_{ik}\in\mathbb F$) satisfying $\sum_{i=k}^\infty r_{ik}=B(\xi)^kA(\xi)$ for all $k\in\mathbb N_0$, where $A(\xi)\in\mathcal F_0(\mathbb F)$ and $B(\xi)\in \mathcal F_1(\mathbb F)$. We denote by $\mathfrak R(\mathbb F)$ the Riordan group consisting of such matrices $R$. He et al.\ \cite{HeHsuShiue} (see also \cite{WangZhang}) proved that the Sheffer group and the Riordan group are isomorphic.

Note that, for each $R\in\mathfrak R(\mathbb F)$, the transposed matrix $R^T$ determines a monic polynomial sequence $(r_k(z))_{k\in\mathbb N_0}$ whose (ordinary) generating function is of the form\linebreak $\sum_{k=0}^\infty \xi^k r_k(z)=(1-z B(\xi))^{-1}A(\xi)$. Here the formal power series $A(\xi)$ and $B(\xi)$ are as in the definition of the Riordan matrix $R$. Orthogonal polynomials with such a generating function (even in several non-commuting variables) were discussed by Anshelevich \cite{ansh} (see also the references therein). Such polynomial sequences play an important role in free probability, see e.g.\ \cite{ansh2}. 

Cheon et al.\ \cite{Cheonatal} equipped the Riordan group with a structure of an infinite-di\-men\-sion\-al manifold, proved that $\mathfrak R(\mathbb F)$ is an infinite-dimensional Lie group, in the sense of Milnor \cite{Milnor}, and explicitly described the set of infinite matrices which form the Lie algebra of $\mathfrak R(\mathbb F)$. It should be noted that the paper \cite{Cheonatal} dealt with the larger group of Riordan lower triangular matrices that have non-zero entries on the diagonal. The paper \cite{Cheonatal} was preceded by Bacher's preprint \cite{Bacher}. See also the preceding paper~\cite{Poinsot} for a discussion of one-parameter subgroups of $\mathbb S(\mathbb F)$. 

A lot of research has been done to extend umbral calculus to the multivariate case, see Section~4 in \cite{DiBL}. However, this research had a significant drawback of being basis-dependent.
The paper \cite{DBLR} was a pioneering work in which some elements of such a theory were developed on a separable Hilbert space. However, no general definition of a binomial sequence or a Sheffer sequence on a Hilbert space was given therein. See also Section 4 of the follow-up paper \cite{DL}.

The paper~\cite{FKLO} developed foundations of infinite-dimensional, basis-independent umbral calculus.
In fact, examples of Sheffer sequences had previously appeared in infinite dimensional analysis on numerous occasions. Some of these polynomial sequences are orthogonal with respect to a probability measure, while others are related to point processes through their correlation functions or analytic structures on infinite dimensional spaces (see e.g.\ the references in~\cite{FKLO}).

Let us briefly recall the definition of a Sheffer sequence from \cite{FKLO} (in a slightly modified but essentially equivalent form, see also \cite{FKLO2}). 
Let $X$ be a locally compact topological vector space. Let $\Phi=C_\mathrm c(X;\mathbb F)$ be the space of continuous $\mathbb F$-valued functions on $X$ with compact support, equipped with its usual locally convex topology. Then the dual of $\Phi$ is $\Phi'=M(X;\mathbb F)$, the space of $\mathbb F$-valued Radon measures on $X$. For $k\ge 2$, denote by $\Phi^{\hat\odot k}$ the subspace of $C_\mathrm c(X^k;\mathbb F)$ consisting of symmetric functions of $k$ variables from $X$. Let also $\Phi^{\hat\odot 1}=\Phi$. For $\omega\in\Phi'$ and $f^{(k)}\in \Phi^{\hat\odot k}$ ($k\ge1$), denote $\langle\omega^{\otimes k}, f^{(k)}\rangle=\int_{X^k}f^{(k)}\, d\omega^{\otimes k}$. 
A polynomial on $\Phi'$ is a function of the form 
$\Phi'\ni\omega\mapsto p(\omega)=f^{(0)}+\sum_{k=1}^K\langle\omega^{\otimes k}, f^{(k)}\rangle$, where 
$f^{(0)}\in\mathbb F$, $f^{(k)}\in\Phi^{\hat\odot k}$ for $k=1, \dots, K$ and $K\in\mathbb N$. 
 Denote by $\mathcal P(\Phi')$ the vector space of polynomials on~$\Phi'$, and equip 
$\mathcal P(\Phi')$ with the topology of the locally convex direct sum of the spaces $\Phi^{\hat\odot k}$, $k\in\mathbb N_0$. (Here $\Phi^{\hat\odot 0}=\mathbb F$.) 
Consider a continuous linear operator $P$ in $\mathcal P(\Phi')$ that acts as follows:
\begin{equation} \label{vcxtew5u}
\big(P\langle\cdot^{\otimes k}, f^{(k)}\rangle\big)(\omega)=P_{0k}f^{(k)}+\sum_{i=1}^k \langle \omega^{\otimes i}, P_{ik}f^{(k)}\rangle, 
\end{equation}
where $P_{ik}\in\mathcal L(\Phi^{\hat\odot k}, \Phi^{\hat\odot i})$ is a continuous linear operator for $0\le i<k$ and $P_{kk}$ is the identity operator in~$\Phi^{\hat\odot k}$. Thus, $P$ can be written in the form of a block-matrix $[P_{ik}]_{i, k\in\mathbb N_0}$ with $P_{ik}=0$ for $i>k$. Denote by $P^{(k)}(\omega, f^{(k)})$ the polynomial on the right-hand side of \eqref{vcxtew5u}. Then $(P^{(k)}(\omega, f^{(k)}))_{f^{(k)}\in\Phi^{\hat\odot k}, \, k\in\mathbb N_0}$ is, by definition, a monic polynomial sequence on~$\Phi'$. We say that it is a Sheffer sequence on $\Phi'$ if its (exponential) generating function has the form $\sum_{k=0}^\infty \frac1{k!}\, P^{(k)}(\omega, \xi^{\otimes k})=\exp\big[\int_X B(\xi)\, d\omega\big]A(\xi)$ for $\omega\in\Phi'$ and $\xi\in\Phi$. Here
 $A(\xi)=1+\sum_{k=1}^\infty A_k\xi^{\otimes k}$, with $A_k\in\mathcal L(\Phi^{\hat\odot k}, \mathbb F)=(\Phi^{\hat\odot k})'$, and $B(\xi)=\xi+\sum_{k=2}^\infty B_k\xi^{\otimes k}$, with $B_k\in\mathcal L(\Phi^{\hat\odot k}, \Phi)$, are formal tensor power series in variable $\xi\in\Phi$, taking values in $\mathbb F$ and~$\Phi$, respectively. We similarly define an Appell sequence on $\Phi'$ (setting $B(\xi)=\xi$) and a binomial sequence on $\Phi'$ (setting $A(\xi)=1$). 
 
 For a Sheffer sequence, an Appell sequence and a binomial sequence on $\Phi'$, the corresponding operator~$P$ in~\eqref{vcxtew5u} 
 is called a Sheffer operator, an Appell operator and an umbral operator, respectively; see \cite{ShefHomeo}.
  We denote 
 the sets of these operators by $\mathbb S(\Phi)$, $\mathbb A(\Phi)$ and $\mathbb B(\Phi)$, respectively. 

Note that, if the underlying space $X$ contains only a finite number of points, then we recover from the above definition the multivariate Sheffer sequences. In particular, if $X$ has a single point, then this is just the classical definition of a monic Sheffer sequence on $\mathbb F$.

In this paper, we prove that the set of Sheffer operators, $\mathbb S(\Phi)$, is a group for the usual product of operators, and furthermore $\mathbb S(\Phi)=\mathbb A(\Phi)\rtimes\mathbb B(\Phi)$. 
We equip $\mathbb S(\Phi)$ with the topology of the topological product of the spaces $\mathcal L(\Phi^{\hat\odot k}, \Phi^{\hat\odot i})$ with $0\le i<k$, which makes~$\mathbb S(\Phi)$ into an infinite-dimensional manifold with a global parametrization. We show that $\mathbb S(\Phi)$ is an infinite-dimensional, regular Lie group, in the sense of Milnor \cite{Milnor}, and furthermore $\mathbb A(\Phi)$ and $\mathbb B(\Phi)$ are embedded Lie subgroups of $\mathbb S(\Phi)$ \cite{Gloeckner}. We provide an explicit description of the Lie algebra $\mathfrak s(\Phi)$ of $\mathbb S(\Phi)$. 

Let us briefly describe the latter result. Note that each functional $\alpha_k\in(\Phi^{\hat\odot k})'$ can be identified with an $\mathbb F$-valued Radon measure $\alpha_k(dx_1\dotsm dx_k)$ on $X^k$. Similarly, an operator $\beta_k\in\mathcal L(\Phi^{\hat\odot k}, \Phi)$ can be identified with an integral kernel: $(\beta_kf^{(k)})(x)=\int_{X^k} f^{(k)}(x_1, \dots, x_k)\, \beta_k(x, dx_1\dotsm dx_k)$. We prove that a continuous linear operator $Q$ in $\mathcal P(\Phi')$ belongs to $\mathfrak s(\Phi)$ if and only if there exist sequences $(\alpha_k)_{k=1}^\infty$, with $\alpha_k\in(\Phi^{\hat\odot k})'$, and $(\beta_k)_{k=2}^\infty$\, , with $\beta_k\in\mathcal L(\Phi^{\hat\odot k}, \Phi)$, such that, for each $p\in \mathcal P(\Phi')$, 
\begin{align}
(Qp)(\omega)&=\sum_{k=1}^\infty\int_{X^k} (D_{x_1}\dotsm D_{x_k}\, p)(\omega)\, \alpha_k(dx_1\dotsm dx_k)\notag\\
&\quad+\sum_{k=2}^\infty \int_X\int_{X^k} (D_{x_1}\dotsm D_{x_k}\, p)(\omega)\, \beta_k(x, dx_1\dotsm dx_k)\, \omega(dx).
\label{vftyre6i4bhft}\end{align}
Here, for $x\in X$, $D_x$ denotes the G\^ateaux derivative in direction $\delta_x$, the Dirac measure at $x$. As a corollary, we conclude that a continuous linear operator $P$ in $\mathcal P(\Phi')$ belongs to $\mathbb S(\Phi)$ if and only if $P=\exp(Q)$ where $Q$ is as in \eqref{vftyre6i4bhft}. The Lie algebra~$\mathfrak a(\Phi)$ of $\mathbb A(\Phi)$ is characterized by the requirement that $\beta_k=0$ for all $k\ge2$, and the Lie algebra $\mathfrak b(\Phi)$ of $\mathbb B(\Phi)$ is characterized by the requirement that $\alpha_k=0$ for all $k\ge1$.

Even in the one-dimensional case, this description of $\mathfrak s(\mathbb F)$ is a new result. In fact, in this case, formula \eqref{vftyre6i4bhft} takes the simple form
\begin{equation}\label{vcdrt4e}
(Qp)(z)=\sum_{k=1}^\infty\alpha_k (D^kp)(z)+z\sum_{k=2}^\infty \beta_k (D^kp)(z), 
\end{equation}
where $\alpha_k, \beta_k\in\mathbb F$ and $D$ denotes the derivative.  Note that, for this result, it is principal that we are dealing with the Sheffer group $\mathbb S(\mathbb F)$ rather than the Riordan group $\mathfrak R(\mathbb F)$, see Subsection~\ref{vgdfr6e64u} below. 

The Lie bracket on $\mathfrak s(\Phi)$ is the usual commutator of operators acting in $\mathcal P(\Phi')$. To write down the Lie bracket explicitly, we first note that formula \eqref{vftyre6i4bhft} provides a bijection between $\mathfrak s(\Phi)$ and the set $\mathcal W(\Phi)$ of all pairs of formal tensor power series $(\alpha(\xi), \beta(\xi))$ of variable $\xi\in\Phi$, where $\alpha(\xi)=\sum_{k=1}^\infty \alpha_k\xi^{\otimes k} $ and $\beta(\xi)=\sum_{k=2}^\infty\beta_k\xi^{\otimes k}$, with $\alpha_k$ and $\beta_k$ as in \eqref{vftyre6i4bhft}. Then the image of the Lie bracket on $\mathfrak s(\Phi)$ under this bijective map is the following Lie bracket on $\mathcal W(\Phi)$:
\begin{equation}\label{cxtewu65}
[(\alpha^{(1)}(\xi), \beta^{(1)}(\xi)), (\alpha^{(2)}(\xi), \beta^{(2)}(\xi)) ]=(\alpha(\xi), \beta(\xi)), 
\end{equation}
where
\begin{align}
\alpha(\xi)&= D_{\beta^{(2)}}\alpha^{(1)}(\xi)- D_{\beta^{(1)}}\alpha^{(2)}(\xi), \label{bvcyed6}\\
\beta(\xi)&= D_{\beta^{(2)}}\beta^{(1)}(\xi)- D_{\beta^{(1)}}\beta^{(2)}(\xi).\label{cfre6u4e}
\end{align}
Here $D_{\beta^{(i)}}$ denotes the derivative of a formal tensor power series in direction of the formal tensor power series $\beta^{(i)}(\xi)$, see Section~\ref{vgydi67r} below. 

Note that, in the one-dimensional case, the operator $D$ of differentiation and the operator $M$ of multiplication by the variable $z$ satisfy the commutation relation $[D, M]=1$, hence they generate a Weyl algebra. Therefore, the explicit form of the Lie bracket on $\mathfrak s(\mathbb F)$ can be thought of as a non-trivial consequence of formula~\eqref{vcdrt4e} and the commutation relation in the Weyl algebra. A similar observation holds in the general case, see Subsection~\ref{xra5aq354q} below.

We also briefly discuss an infinite-dimensional generalization $\mathfrak R(\Phi) $ of the Riordan group $\mathfrak R(\mathbb F)$. The elements of $\mathfrak R(\Phi) $ can be written in the form of a block-matrix $[(k!/i!)P_{ik}]_{i, k\in\mathbb N_0}$, provided that $[P_{ik}]_{i, k\in\mathbb N_0}$ belongs to $\mathbb S(\Phi)$. This representation leads to a group isomorphism between $\mathfrak R(\Phi) $ and $\mathbb S(\Phi)$, compare with  \cite{HeHsuShiue,WangZhang}.  We note that, in the case where the space $\Phi$ is finite-dimensional, i.e., $\Phi=\mathbb F^N$ ($N\in\mathbb N$, $N\ge2$), $\mathfrak R(\Phi)$ is (the subgroup of the unitriangular block-matrices from) the multivariate Riordan group as defined and studied in~\cite[Section~7.3]{ShapiroBook}, see also the references therein.

 We provide an explicit description for the matrix representation of the Lie algebra $\mathfrak r(\Phi) $ of $\mathfrak R(\Phi) $ and the Lie bracket on it. In the  one-dimensional case, $\Phi=\mathbb F$, this result leads to improved understanding of the Lie algebra $\mathfrak R(\mathbb F)$, compare with \cite{Cheonatal}.

Under a weak assumption on the underlying space $X$, the space $\Phi=C_\mathrm c(X;\mathbb F)$ is a (strict) (LB)-space. So everywhere below we will only assume that $\Phi$ is a (strict) (LB)-space and $\Phi'$ is its dual space. We will also assume below that all vector spaces under consideration are complex. Nevertheless, our results remain true in the real case, in particular, all the analytic maps appearing in this paper will be real analytic in the real case. 

The paper is organized as follows. In Section~\ref{vgdtrs5w4v}, we discuss preliminaries: (LB)-spaces, continuous linear operators acting between (LB)-spaces, injective tensor product of Banach spaces, and the special case of the (LB)-space $C_\mathrm c(X;\mathbb C)$. In Section~\ref{vgydi67r}, we discuss Lie structures of $\mathcal S(\Phi)$, an infinite-dimensional generalization of the Lie group~$\mathcal S(\mathbb C)$. In Section~\ref{tdr6e65e48}, we discuss Lie structures of the group of continuous linear operators in $\mathcal P(\Phi')$ that determine monic polynomial sequences on $\Phi'$. In Section~\ref{bvfyr64u}, we discuss the main results of the paper related to the Sheffer Lie group $\mathbb S(\Phi)$. Finally, in 
the Appendix~A, we present several auxiliary results concerning analyticity of functions acting between spaces of continuous linear operators in (LB)-spaces; and in the Appendix~B, we briefly recall some constructions of Milnor \cite{Milnor}, regarding infinite-dimensional Lie structures, in the special case of  a global parametrization.

\section{Preliminaries} \label{vgdtrs5w4v}

\subsection{(LB)-spaces and continuous linear operators in them}\label{vctes56u7}
 Let $(\Phi_n)_{n=1}^\infty$ be a sequence of complex Banach spaces such that, for each $n\in\mathbb N$, $\Phi_n$ is a subspace of $\Phi_{n+1}$, and the restriction of the norm in $\Phi_{n+1}$ to $\Phi_n$ coincides with the norm in $\Phi_n$. Denote $\Phi:=\bigcup_{n=1}^\infty \Phi_n$ and equip $\Phi$ with the inductive limit topology of the $\Phi_n$ spaces being embedded into $\Phi$, see e.g.\ \cite[Ch.~II, \S~6]{Schaefer}. 
 Then $\Phi$ is called an {\it (LB)-space}, see e.g.\ \cite[Ch.~II, \S~6.3]{Schaefer} (sometimes $\Phi$ is also called {\it a strict (LB)-space}\/). Obviously, an (LB)-space is a special case of a locally convex topological vector space (l.c.s.).

By \cite[Corollaire, p.~87]{DS} or \cite[Ch.~II, \S~6.6]{Schaefer}, the space $\Phi$ is complete. If each $\Phi_n$ is a proper subspace of $\Phi_{n+1}$, then, by \cite[Corollaire~2]{DS}, the space $\Phi$ is not metrizable. By \cite[Proposition~2]{DS}, the topology induced by $\Phi$ on each $\Phi_n$ coincides with the original topology on $\Phi_n$. By \cite[Proposition~4]{DS}, a subset $E$ of $\Phi$ is bounded, respectively compact if and only if $E\subset\Phi_n$ for some $n\in\mathbb N$ and $E$ is bounded, respectively compact in $\Phi_n$. The latter property means that $\Phi$ is compactly regular, see e.g.\ \cite[Definition~8.5.32]{PCB}.

For l.c.s.'s $V_1$ and $V_2$, we denote by $\mathcal L(V_1, V_2)$ the vector space of all continuous linear operators $A:V_1\to V_2$. If $V=V_1=V_2$, we denote $\mathcal L(V):=\mathcal L(V, V)$.

\begin{lemma}\label{vcyre6i}
Let $\Phi=\indlim_{n\to\infty}\Phi_n$ and $\Psi=\indlim_{m\to\infty}\Psi_m$ be (LB)-spaces. A linear operator $A:\Phi\to\Psi$ is continuous if and only if, for each $n\in\mathbb N$, there exists $m\in\mathbb N$ such that the restriction $A\restriction_{\Phi_n}\in\mathcal L(\Phi_n, \Psi_m)$.
\end{lemma}

\begin{proof} By \cite[Proposition~5]{DS} or \cite[Ch.~II, \S~6.1]{Schaefer}, a linear operator $A:\Phi\to\Psi$ is continuous if and only if $A\restriction_{\Phi_n}\in\mathcal L(\Phi_n, \Psi)$ for all $n\in\mathbb N$. 
By the definition of the inductive limit, we obviously have the inclusion
$\bigcup_{m\in\mathbb N}\mathcal L(\Phi_n, \Psi_m)\subset \mathcal L(\Phi_n, \Psi)$.
On the other hand, if $A\in \mathcal L(\Phi_n, \Psi)$, then the linear operator $A$ is bounded, i.e., it sends a bounded subset of $\Phi_n$ to a bounded subset of $\Psi$, see e.g.\ \cite[Section~4-4]{Wilansky}. Since a bounded subset of $\Psi$ is a bounded subset of some $\Psi_m$, this implies that the linear operator $A$ maps $\Phi_n$ into $\Psi_m$ and is bounded. Hence $A\in\mathcal L(\Phi_n, \Psi_m)$, and so 
\begin{equation}\bigcup_{m\in\mathbb N}\mathcal L(\Phi_n, \Psi_m)=\mathcal L(\Phi_n, \Psi).\label{utf7ir}
\end{equation}
\end{proof}

\begin{remark}\label{seiytr8o5w5} Let $\Phi=\indlim_{n\to\infty}\Phi_n$ be an (LB)-space, and let $\Phi'=\mathcal L(\Phi, \mathbb C)$ be its (continuous) dual space. By Lemma~\ref{vcyre6i} with $\Psi=\mathbb C$, we have
$\Phi'=\bigcap_{n\in\mathbb N} \Phi_n'$. 
\end{remark}

For $\Phi$ and $\Psi$ as in Lemma \ref{vcyre6i}, we define a topology on $\mathcal L(\Phi, \Psi)$ as follows. 
First, we fix $n\in\mathbb N$. Each $\mathcal L(\Phi_n, \Psi_m)$ is a Banach space, $\mathcal L(\Phi_n, \Psi_m)$ is a subspace of $\mathcal L(\Phi_n, \Psi_{m+1})$, and the restriction of the norm in $\mathcal L(\Phi_n, \Psi_{m+1})$ 
to $\mathcal L(\Phi_n, \Psi_{m})$ coincides with the norm in $\mathcal L(\Phi_n, \Psi_{m})$. Hence, in view of \eqref{utf7ir}, we construct the (LB)-space
\begin{equation*}
\mathcal L(\Phi_n, \Psi)=\indlim_{m\to\infty}\mathcal L(\Phi_n, \Psi_m).
\end{equation*}
Next, in view of Lemma~\ref{vcyre6i} and its proof, we may define, for each $n\in\mathbb N$, a linear map
\begin{equation*}
\mathcal L(\Phi, \Psi)\ni A\mapsto A\restriction_{\Phi_n}\in\mathcal L(\Phi_n, \Psi).
\end{equation*}
By using these maps, we equip 
$\mathcal L(\Phi, \Psi)=\bigcap_{n=1}^\infty \bigcup_{m=1}^\infty\mathcal L(\Phi_n, \Psi_m)$
 with the projective limit topology (see e.g.\ \cite[Ch.~II, \S~5]{Schaefer}):
\begin{equation*}
\mathcal L(\Phi, \Psi)=\prlim_{n\to\infty}\mathcal L(\Phi_n, \Psi).
\end{equation*}
By e.g.\ \cite[Ch.~II, \S~5.3]{Schaefer}, the l.c.s.\ $\mathcal L(\Phi, \Psi)$ is complete. In particular, the dual space $\Phi'=\mathcal L(\Phi, \mathbb C)$ is complete.

 For a l.c.s.\ $V$ and a closed interval $I\subset \mathbb{R}$, we consider the space $C(I;V)$ of all continuous $V$-valued curves (functions) on $I$. We equip $C(I;V)$ with its usual (locally convex) topology of uniform convergence. Thus, a local base for zero in $C(I;V)$ is given by the sets $\{f\in C(I;V): f(t)\in U\ \forall t\in I\}$, where $U\in\mathcal N_0$ and $\mathcal N_0$ is a local base for zero in $V$.
 
 We similarly define a topology on $C^k(I;V)$ ($k\in\mathbb N$), the space of all $k$-times  continuously differentiable $V$-valued curves on $I$. Furthermore, we equip $C^\infty(I;V)=\bigcap_{k=0}^\infty C^k(I;V)$ with the projective limit topology of the $C^k(I;V)$ spaces, which makes $C^\infty(I;V)$ an l.c.s. (Here $C^0(I;V):=C(I;V)$.) We will call functions from $C^\infty(I;V)$ \emph{smooth curves}.

\begin{lemma}\label{vcrswu5} Let $\Phi$ and $\Psi$ be as in Lemma \ref{vcyre6i}. We have
\begin{align}
C([0, 1];\mathcal L(\Phi_n, \Psi))&=\indlim_{m\to\infty}C([0, 1];\mathcal L(\Phi_n, \Psi_m)), \label{fxdste5swu5}\\
C([0, 1];\mathcal L(\Phi, \Psi))&=\prlim_{n\to\infty}C([0, 1];\mathcal L(\Phi_n, \Psi)), \label{fstsw5u}
\end{align}
where in formula \eqref{fstsw5u}, the projective limit is constructed via the maps
$$ C([0, 1];\mathcal L(\Phi, \Psi))\ni A(\cdot)\mapsto A(\cdot)\restriction_{\Phi_n}\in C([0, 1];\mathcal L(\Phi_n, \Psi)), \quad n\in\mathbb N.$$
\end{lemma}

\begin{proof} Since the (LB)-space is compactly regular, formula \eqref{fxdste5swu5} immediately follows, see e.g.\ \cite[p.~61]{Domanski}. By using the definition of the projective limit topology, we similarly conclude that formula \eqref{fstsw5u} holds. 
\end{proof}

Recall that, for a continuous function from $C(I;V)$ with a (sequentially) complete l.c.s.\ $V$, its Riemann integral is well defined, see e.g.\ \cite[Section~2]{Milnor}.

\begin{lemma}\label{dtstesrew5}

 Let $\Phi$ and $\Psi$ be as in Lemma~\ref{vcyre6i}.
 
 (i) Let $A(\cdot)\in C([0, 1];\mathcal L(\Phi, \Psi))$. Then, the function $[0, 1]\ni t\mapsto \int_0^t A(r)\, dr\in \mathcal L(\Phi, \Psi)$ is differentiable and has derivative $A(t)$. 

(ii)~Let $I\subset\R$ be a closed interval that contains $0$,  let $A(\cdot)\in C(I;\mathcal L(\Phi, \Psi))$ and $B_0\in \mathcal L(\Phi, \Psi)$. Then, there exists a unique curve $B(\cdot)\in C^1(I;\mathcal L(\Phi, \Psi))$ that solves the Cauchy problem 
$B'(t)=A(t)$ ($t\in I$), $B(0)=B_0$, and furthermore $B(t)=B_0+\int_0^t A(r)\, dr$. 
\end{lemma} 

\begin{proof}

 (i)~By~Lemma~\ref{vcrswu5}, for each $n\in\mathbb N$, there exists $m\in\mathbb N$ such that, for all $t\in[0, 1]$, the restriction $A(t)\restriction_{\Phi_n}$ belongs to $\mathcal L(\Phi_n, \Psi_m)$, and the curve $[0, 1]\ni t\mapsto A(t)\restriction_{\Phi_n}\in \mathcal L(\Phi_n, \Psi_m)$ is continuous. Therefore, $\int_0^t (A(r)\restriction_{\Phi_n})\, dr$ is the Riemann integral of a function with values in the Banach space $\mathcal L(\Phi_n, \Psi_m)$. This implies the statement, cf.\ \cite[1.11~Remark]{Schmeding}.

 (ii) This follows from part (i) and \cite[Assertion~3.2]{Milnor}.
\end{proof}

\subsection{Injective tensor product}\label{vcftstwe6}
For Banach spaces $B_1$ and $B_2$, we denote by $B_1\otimes B_2$ the (algebraic) tensor product of~$B_1$ and $B_2$, e.g.\ \cite[Section~1.1]{Ryan}. Let $B_1\hat\otimes B_2$ denote the completion of $B_1\otimes B_2$ in the injective norm, \footnote{The injective tensor product is usually denoted by $\hat\otimes_\varepsilon$, however we drop the lower index $\varepsilon$ since we will not use any other completions of $B_1\otimes B_2$ in this paper.} see e.g.\ \cite[Section~3.1]{Ryan}. Namely, for any $u=\sum_{i=1}^n b_{i1}\otimes b_{i2}\in B_1\otimes B_2$, the injective norm of $u$ is given by
\begin{align}
\|u\|_{B_1\hat\otimes B_2}&=\sup\bigg{\{}\bigg|\sum_{i=1}^n\omega_1(b_{i1})\omega_2(b_{i2})\bigg|: \omega_i\in B_i', \ \|\omega_i\|_{B_i'}\le1, \ i=1, 2\bigg\}\label{cxts5tws5w}\\
&=\sup\bigg{\{}\bigg{\|}\sum_{i=1}^nb_{i1}\, \omega_2(b_{i2})\bigg{\|}_{B_1}: \omega_2\in B_2', \ \|\omega_2\|_{B_2'}\le1\bigg{\}}\label{vyre7ki7tr8eid}\\
&=\sup\bigg{\{}\bigg{\|}\sum_{i=1}^n\omega_1(b_{i1})b_{i2}\bigg{\|}_{B_2}: \omega_1\in B_1', \ \|\omega_1\|_{B_1'}\le1\bigg{\}}.\label{hgcdyjr78}
\end{align}
It easily follows from \eqref{cxts5tws5w}--\eqref{hgcdyjr78} that the injective tensor product is associative, and so we can similarly consider the injective tensor product of several Banach spaces.

By e.g.\ \cite[Proposition~3.1]{Ryan}, we have, for any $b_1\in B_1$ and $b_2\in B_2$, 
$$\|b_1\otimes b_2\|_{B_1\hat\otimes B_2}=\|b_1\|_{B_1}\|b_2\|_{B_2}, $$
and for any $\omega_1\in B_1'$ and $\omega_2\in B_2'$, we have $\omega_1\otimes\omega_2\in B_1'\otimes B_2'\subset (B_1\hat\otimes B_2)'$ with
\begin{equation}\label{draze4tqa}
\|\omega_1\otimes\omega_2\|_{(B_1\hat\otimes B_2)'}=\|\omega_1\|_{B_1'}\|\omega_2\|_{B_2'}.\end{equation}
 
 By e.g.\ \cite[Proposition~3.2]{Ryan}, for Banach spaces $B_{ij}$ ($i, j\in\{1, 2\}$) and continuous linear operators $A_i\in\mathcal L(B_{i1}, B_{i2})$, there exists a unique continuous linear operator $A_1 \otimes A_2\in\mathcal L(B_{11}\hat\otimes B_{21}, B_{12}\hat\otimes B_{22})$ such that $(A_1\otimes A_2)(b_1\otimes b_2)=(A_1b_1)\otimes(A_2b_2)$ for all $b_1\in B_{11}$ and $b_2\in B_{21}$. Furthermore, we have
 \begin{equation*}
 \|A_1\otimes A_2\|_{\mathcal L(B_{11}\hat\otimes B_{21}, B_{12}\hat\otimes B_{22})}=\|A_1\|_{\mathcal L(B_{11}, B_{12})} \|A_2\|_{\mathcal L(B_{21}, B_{22})}.
 \end{equation*}

The injective tensor product respects subspaces, see e.g.\ \cite[p.~47]{Ryan}. More precisely, if $C_i$ is a closed subspace of the Banach space $B_i$ ($i=1, 2$), then $C_1\hat\otimes C_2$ is a closed subspace of $B_1\hat\otimes B_2$.

Let $B$ be a Banach space and $k\in \mathbb N$. We define a continuous linear operator $\operatorname{Sym}_k\in\mathcal L(B^{\hat\otimes k})$ by
$$\operatorname{Sym}_k(b_1\otimes b_2\otimes\dots\otimes b_k):=\frac1{k!}\sum_{\pi\in S_k}b_{\pi(1)}\otimes b_{\pi(2)}\otimes\dots\otimes b_{\pi(k)}, \quad b_1, b_2, \dots, b_k\in B, $$
where $S_k$ is the symmetric group of order $k$. Since 
the norm of this operator is obviously bounded by 1 and $\operatorname{Sym}_kb^{\otimes k}=b^{\otimes k}$ for all $b\in B$, the norm of $\operatorname{Sym}_k$ is equal to 1. We define $B^{\hat\odot k}$, the {\it $k$th symmetric injective tensor power of $B$}, as the range of $\operatorname{Sym}_k$. Since $\operatorname{Sym}_k^2=\operatorname{Sym}_k$ on $B^{\hat\otimes k}$, we easily see that
 $B^{\hat\odot k}$ is a closed subspace of $B^{\hat\otimes k}$. We denote
\begin{equation*}
b_1\odot b_2\odot\dots\odot b_k:=\operatorname{Sym}_k(b_1\otimes b_2\otimes\dots\otimes b_k).\end{equation*}
It follows from the polarization identity that $B^{\hat\odot k}$ is equal to the closure of the linear span of vectors of the form $b^{\otimes k}=b^{\odot k}$ with $b\in B$. We will also denote $B^{\hat\otimes0}=B^{\hat\odot 0}:=\mathbb C$ and $\operatorname{Sym}_0:=\mathbf 1$, the identity function in $\mathbb C$.

For any Banach spaces $B_1$ and $B_2$, any $i_1, i_2\in\mathbb N$, $j_1, j_2\in\mathbb N_0$, and $A_1\in\mathcal L(B_1^{\hat\odot i_1}, B_2^{\hat\odot j_1})$, $A_2\in\mathcal L(B_1^{\hat\odot i_2}, B_2^{\hat\odot j_2})$, the symmetric tensor product of $A_1$ and $A_2$ is defined as 
$$A_1\odot A_2\in\mathcal L(B_1^{\hat\odot(i_1+i_2)}, B_2^{\hat\odot(j_1+j_2)}), \quad A_1\odot A_2:=\operatorname{Sym}_{j_1+j_2}\bigg((A_1 \otimes A_2)\restriction_{B_1^{\hat\odot(i_1+i_2)}}\bigg).$$ 
As easily seen, $A_1\odot A_2=A_2\odot A_1$. Furthermore, this symmetric tensor product is obviously associative.

Let $\Phi=\indlim_{n\to\infty}\Phi_n$ be an (LB)-space. Then, for a fixed $k\in\mathbb N$, the Banach space $\Phi_n^{\hat\odot k}$ is a subspace of $\Phi_{n+1}^{\hat\odot k}$ and the restriction of the norm in $\Phi_{n+1}^{\hat\odot k}$ to $\Phi_n^{\hat\odot k}$ coincides with the norm in $\Phi_n^{\hat\odot k}$. Hence, we define the (LB)-space 
\begin{equation}\label{xreaWQQ}\Phi^{\hat\odot k}:=\indlim_{n\to\infty}\Phi_n^{\hat\odot k}.\end{equation}
We similarly define the (LB)-space $\Phi^{\hat\otimes k}$. We will denote $\Phi^{\hat\otimes0}=\Phi^{\hat\odot 0}:=\mathbb C$.

\begin{remark}\label{cfsthewdd}
Let $A\in\mathcal L(\Phi^{\hat\odot k}, \Phi^{\hat\odot l})$. It follows from Lemma~\ref{vcyre6i} and the definition of $\Phi^{\hat\odot k}$ that the operator $A$ is uniquely characterized by its values on the set $\{\xi^{\otimes k}\mid\xi\in\Phi\}$.
\end{remark}

Lemma~\ref{vcyre6i} allows us to define, for any $A_1\in\mathcal L(\Phi^{\hat\odot i_1}, \Phi^{\hat\odot j_1})$ and $A_2\in\mathcal L(\Phi^{\hat\odot i_2}, \Phi^{\hat\odot j_2})$, their symmetric tensor product $A_1\odot A_2\in\mathcal L(\Phi^{\hat\odot (i_1+i_2)}, \Phi^{\hat\odot (j_1+j_2)})$.

Recall Remark~\ref{seiytr8o5w5} and formula \eqref{draze4tqa}. For each $k\in\mathbb N$ and $\omega\in\Phi'$, we have 
$\omega^{\otimes k}\in(\Phi_n^{\hat\otimes k})'$ for all $n\in\mathbb N$, and so $\omega^{\otimes k}\in(\Phi_n^{\hat\odot k})'$ for all $n$. Hence, by \eqref{xreaWQQ} and e.g.\ \ \cite[Ch.~II, \S~6.1]{Schaefer}, $\omega^{\otimes k}\in (\Phi^{\hat\odot k})'$.

For $f^{(k)}$ from $\Phi^{\hat\odot k}$ (or $\Phi^{\hat\otimes k}$) and 
$\psi^{(k)}$ from $(\Phi^{\hat\odot k})'$ (or $(\Phi^{\hat\otimes k})'$), we will denote by either $\langle\psi^{(k)}, f^{(k)}\rangle$ or just $\psi^{(k)}f^{(k)}$
the dual pairing between $\psi^{(k)}$ and $f^{(k)}$. (Thus, $\langle\cdot, \cdot\rangle$ is a bilinear form on $(\Phi^{\hat\odot k})'\times \Phi^{\hat\odot k}$, or on $(\Phi^{\hat\otimes k})'\times \Phi^{\hat\otimes k}$.) In particular, for $\omega\in\Phi'$ and $f^{(k)}\in \Phi^{\hat\odot k}$, the dual pairing 
$\langle\omega^{\otimes k}, f^{(k)}\rangle$ is well-defined.

\begin{lemma}\label{fsrea45y}
Let $f^{(k)}\in \Phi^{\hat\odot k}$ be such that $\langle\omega^{\otimes k}, f^{(k)}\rangle=0$ for all $\omega\in\Phi'$. Then $f^{(k)}=0$.
\end{lemma}

\begin{proof}
 Recall \eqref{xreaWQQ}. Choose $n\in\mathbb N$ such that $f^{(k)}\in\Phi_n^{\hat\odot k}$. 
 It easily follows from \eqref{cxts5tws5w} (see also \cite[p.~45]{Ryan}) that the norm $\|f^{(k)}\|_{\Phi_n^{\hat \otimes k}}$ is equal to the norm of the $k$-linear form $b$ on $(\Phi_n')^k$ defined by
 $b(\omega_1, \dots, \omega_k):=\langle\omega_1\otimes\dots\otimes \omega_k, f^{(k)}\rangle$ for $\omega_1, \dots, \omega_k\in\Phi_n'$.
 Since $f^{(k)}$ belongs to the $k$th symmetric tensor power of $\Phi_n$, the form $b$ is symmetric. Hence, by the polarization identity, the form $b$ is completely determined by the values 
 $b(\omega, \dots, \omega)=\langle \omega^{\otimes k}, f^{(k)}\rangle$, $\omega\in \Phi_n'$.
But by our assumption, $b(\omega, \dots, \omega)=0$ for all $\omega\in\Phi_n'$, and so the form $b$ is identically equal to zero. This in turn implies that $\|f^{(k)}\|_{\Phi_n^{\hat\otimes k}}=0$. 
\end{proof}

\subsection{A class of examples}\label{vcts5ywiyu}

Let us now discuss a class of examples of the constructions in Subsections~\ref{vctes56u7} and~\ref{vcftstwe6}.

 Let $X$ be a locally compact topological space. Furthermore, we assume that the space $X$ is countable at infinity, i.e., $X$ is the union of a sequence of relatively compact open sets $U_n$ such that $\overline U_n\subset U_{n+1}$. 
 Let $\Phi_n$ denote the Banach space of all continuous functions $f:X\to\mathbb C$ that vanish outside $\overline U_n$, with $\|f\|_{\Phi_n}:=\sup_{x\in X}|f(x)|$. Let $C_\mathrm c(X;\mathbb C)$ denote the space of all continuous functions $f:X\to\mathbb C$ with compact support. We equip $C_\mathrm c(X;\mathbb C)$ with the coarsest locally convex topology such that $f_n\to f$ in $C_\mathrm c(X;\mathbb C)$ if and only if there exists a compact set $\mathcal K$ in $X$ such that all functions $f_n$ vanish outside~$\mathcal K$ and $f_n\to f$ uniformly on $X$. Then $C_\mathrm c(X;\mathbb C)$ is the inductive limit of the $\Phi_n$ spaces, see e.g.\ \cite[Ch.~III, \S1, Sect.~1]{Bourbaki}. Thus, $\Phi:=C_\mathrm c(X;\mathbb C)$ is an (LB)-space.

 Let $\mathcal B(X)$ denote the Borel $\sigma$-algebra on $X$ and let $\mathcal B_0(X)$ denote the collection of all precompact sets from $\mathcal B(X)$. Recall that a positive Radon measure $\sigma$ on $(X, \mathcal B(X))$ is a measure satisfying $\sigma(A)<\infty$ for all $A\in\mathcal B_0(X)$, and a complex-valued Radon measure $\omega$ on $X$ has the form $\omega=(\omega_1-\omega_2)+i(\omega_3-\omega_4)$, where $\omega_1, \omega_2, \omega_3, \omega_4$ are positive Radon measures. Note that, generally speaking, a complex-valued Radon measure $\omega$ is well-defined only on $\mathcal B_0(X)$. We denote by $M(X;\mathbb C)$ the vector space of all complex-valued Radon measures on $X$. This space can be identified with $\Phi'$, the dual space of $\Phi$. The dual pairing between 
 $\omega\in \Phi'=M(X;\mathbb C)$ and $f\in \Phi=C_\mathrm c(X;\mathbb C)$ is given by $\langle\omega, f\rangle=\int_X f\, d\omega$, see e.g.\ \cite[Ch.~III, \S1, Sects.~3 \&\ 5]{Bourbaki} and 
 \cite[\S 29]{Bauer}.

 Similarly to \cite[Section~3.2]{Ryan}, one can show that, for all $k\ge2$ and $n\in\mathbb N$, $\Phi_n^{\hat\otimes k}$ is the space of all continuous functions $f^{(k)}:X^k\to\mathbb C$ that vanish outside $(\overline U_n)^k$, with $\|f^{(k)}\|_{\Phi_n^{\hat\otimes k}}=\sup_{(x_1, \dots, x_k)\in X^k}|f^{(k)}(x_1, \dots, x_k)|$. Hence, $\Phi^{\hat\otimes k}=C_\mathrm c(X^k;\mathbb C)$. One can easily conclude that $\Phi^{\hat\odot k}$ is the subspace of $C_\mathrm c(X^k;\mathbb C)$ that consists of all symmetric functions $f^{(k)}$ from $C_\mathrm c(X^k;\mathbb C)$. This also implies that, for each $\omega\in\Phi'=M(X;\mathbb C)$ and $f^{(k)}\in\Phi^{\hat\odot k}$, we have 
 $\langle\omega^{\otimes k}, f^{(k)}\rangle=\int_{X^k}f^{(k)}\, d\omega^{\otimes k}$.

\section{Lie structures on formal tensor power series}\label{vgydi67r}
Everywhere below $\Phi=\indlim_{n\to\infty}\Phi_n$ is an (LB)-space.
\subsection{Formal tensor power series}

Recall Remark \ref{cfsthewdd}. 
Let $i\in\mathbb N_0$. A {\it formal tensor power series on $\Phi$ with values in $\Phi^{\hat\odot i}$} is a formal series of the form $C(\xi)=\sum_{k=i}^\infty C_{k}\xi^{\otimes k}$ for $\xi\in\Phi$, where $(C_{k})_{k\ge i}$ is an arbitrary sequence of operators $C_{k}\in\mathcal L(\Phi^{\hat\odot k}, \Phi^{\hat\odot i})$. We denote by
$\mathcal F(\Phi;\Phi^{\hat\odot i})$ the set of all formal tensor power series on $\Phi$ with values in $\Phi^{\hat\odot i}$. We equip $\mathcal F(\Phi;\Phi^{\hat\odot i})$ with the topology of the topological product of the spaces $\mathcal L(\Phi^{\hat\odot k}, \Phi^{\hat\odot i})$ with $k\ge i$. Hence, $\mathcal F(\Phi;\Phi^{\hat\odot i})$ is a complete l.c.s., see e.g.\ \cite[Ch.~II, \S~5.2 and 5.3]{Schaefer}.

\begin{remark}\label{bydtrsw5w}
 Note that for $s\in\mathbb C$, and $C(\xi)=\sum_{k=i}^\infty C_{k}\xi^{\otimes k}\in \mathcal F(\Phi;\Phi^{\hat\odot i})$, we have
$C(s\xi)=\sum_{k=i}^\infty C_{k}(s\xi)^{\otimes k}=\sum_{k=i}^\infty s^k\, C_{k}\xi^{\otimes k}$, 
where on the right-hand side we have a formal power series in $s$ with coefficients from $\Phi^{\hat\odot i}$. 
\end{remark}

\begin{remark}\label{cfta4a4zz} Let $C_k\in\mathcal L(\Phi^{\hat\odot k}, \Phi^{\hat\odot i})$. Then, for each $\omega\in\Phi'$, we obviously have $\langle\omega^{\otimes i}, C_k\, \cdot\rangle\in(\Phi^{\hat\odot k})'$. Therefore, for $C(\xi)=\sum_{k=i}^\infty C_{k}\xi^{\otimes k}\in \mathcal F(\Phi;\Phi^{\hat\odot i})$ and $\omega\in\Phi'$, we may define $\langle \omega^{\otimes i}, C(\xi)\rangle:= \sum_{k=i}^\infty \langle\omega^{\otimes i}, C_k\xi^{\otimes k}\rangle\in\mathcal F(\Phi;\mathbb C)$. It follows from Lemma~\ref{fsrea45y} that the formal tensor power series $\big(\langle \omega^{\otimes i}, C(\xi)\rangle\big)_{\omega\in\Phi'}$ uniquely characterize $C(\xi)$.
\end{remark}

For formal tensor power series $A^{(j)}(\xi)=\sum_{k=0}^\infty A^{(j)}_k\xi^{\otimes k}\in\mathcal F(\Phi;\mathbb C)$ ($j=1, 2$), we define their {\it (multiplicative) product}
\begin{equation}\label{cfrtsw5yw}
A^{(1)}(\xi)A^{(2)}(\xi)
:=\sum_{k=0}^\infty\sum_{l=0}^\infty \big(A^{(1)}_k\odot A^{(2)}_l\big)\xi^{\otimes(k+l)}=\sum_{k=0}^\infty A_k\xi^{\otimes k}\in\mathcal F(\Phi;\mathbb C), \end{equation}
 where 
 \begin{equation}\label{bftyde6w}
A_0=A_0^{(1)}A_0^{(2)}, \quad A_k=\sum_{l=0}^k A^{(1)}_l\odot A^{(2)}_{k-l}\quad\text{for }k\in\mathbb N.
\end{equation}

More generally, for $i, j\in\mathbb N_0$ and formal tensor power series $B(\xi)=\sum_{k=i}^\infty B_k\xi^{\otimes k}\in\mathcal F(\Phi, \Phi^{\hat\odot i})$ and $C(\xi)=\sum_{l=j}^\infty C_l\xi^{\otimes l}\in\mathcal F(\Phi, \Phi^{\hat\odot j})$, we define their {\it symmetric tensor product}
\begin{equation}\label{cxdtsrea5}
B(\xi)\odot C(\xi)=\sum_{k=i+j}^\infty D_k\xi^{\otimes k}\in\mathcal F(\Phi;\Phi^{\hat\odot (i+j)}), 
\end{equation} where 
\begin{equation}\label{vcftse5a5ywu}D_k:=\sum_{l=i}^{k-j} B_l\odot C_{k-l}.
\end{equation}
If $i=j=0$, then the symmetric tensor product is just the multiplicative product. Furthermore, in the case where only one of $i$ and $j$ is equal to zero, say $i=0$, the formal tensor power series $B(\xi)$ becomes scalar-valued, hence we will still write $B(\xi)C(\xi)$ instead of $B(\xi)\odot C(\xi)$.

For $C(\xi)=\sum_{k=i}^\infty C_{k}\xi^{\otimes k}\in \mathcal F(\Phi;\Phi^{\hat\odot i})$ ($i\in\mathbb N_0$) and $B(\xi)=\sum_{k=1}^\infty B_k\xi^{\otimes k}\in\mathcal F(\Phi, \Phi)$, we define the {\it composition (substitution)}
\begin{equation}
(C\circ B)(\xi)=C(B(\xi)):=\sum_{k=i}^\infty C_k \bigg(\sum_{l=1}^\infty B_l\xi^{\otimes l}\bigg)^{\otimes k}
=\sum_{k=i}^\infty D_k\xi^{\otimes k}\in \mathcal F(\Phi;\Phi^{\hat\odot i}), \label{vcxdeay543q}
\end{equation}
where $D_0=C_0$ if $i=0$, and for $k\ge\max\{1, i\}$, we have
\begin{equation}\label{crtsdtrsew}
D_k=\sum_{l=1}^k C_l\sum_{\substack{i_1, \dots, i_l\in\mathbb N\\i_1+\dots+i_l=k}}B_{i_1}\odot B_{i_2}\odot\cdots\odot B_{i_l}.\end{equation}
Sometimes it will be convenient to write the composition in \eqref{vcxdeay543q} as $C(\xi)\circ B(\xi)$.

Similarly, for each (formal) power series of a complex argument, $q(z)=\sum_{k=0}^\infty q_kz^k$ ($q_k\in\mathbb C$), and $A(\xi)=\sum_{k=1}^\infty A_k\xi^{\otimes k}\in\mathcal F(\Phi;\mathbb C)$, we define the {\it composition} 
\begin{equation}\label{cxrsa4wq4}
q(A(\xi))=q_0+\sum_{k=1}^\infty q_k\bigg(\sum_{l=1}^\infty A_l\xi^{\otimes l}\bigg)^k
=q_0+\sum_{k=1}^\infty B_k\xi^{\otimes k}\in\mathcal F(\Phi;\mathbb C), 
\end{equation}
where 
\begin{equation}\label{xra4waq}
B_k=\sum_{l=1}^kq_l\sum_{\substack{i_1, \dots, i_l\in\mathbb N\\i_1+\dots+i_l=k}}A_{i_1}\odot\cdots\odot A_{i_l}.
\end{equation}

Let $C_k\in\mathcal L(\Phi^{\hat\odot k}, \Phi^{\hat\odot i})$ ($k\ge i$) and consider the function $\Phi\ni\xi\mapsto C_k\xi^{\otimes k}\in\Phi^{\hat\odot i}$. 
For a fixed $\zeta\in\Phi$, the G\^ateaux derivative of this function in direction $\zeta$ at $\xi$ is given by
$$ D_\zeta(C_k\xi^{\otimes k}):=\lim_{t\to0}\frac1t\big(C_k(\xi+t\zeta)^{\otimes k}-C_k\xi^{\otimes k}\big)=kC_k(\xi^{\odot(k-1)}\odot\zeta).$$
Note that $C_k(\cdot\odot\zeta)\in\mathcal L(\Phi^{\hat\odot(k-1)}, \Phi^{\hat\odot i})$. In view of this, for $\zeta\in\Phi$ and $C(\xi)=\sum_{k=i}^\infty C_k\xi^{\otimes k}\in\mathcal F(\Phi;\Phi^{\hat\odot i})$, we define the {\it $ D_\zeta$\,-derivative of $C(\xi)$} as\footnote{Note that the summation on the right-hand side of formula \eqref{cxtrswu5e4i6} starts with $i-1$. Hence, for $i\ge1$, $ D_\zeta C(\xi)$ belongs to $\mathcal F(\Phi;\Phi^{\hat\odot i})$ only if $C_i=0$.}
\begin{equation}\label{cxtrswu5e4i6}
 D_\zeta C(\xi):=\sum_{k=i}^\infty kC_{k}(\xi^{\odot (k-1)}\odot\zeta)=\sum_{k=i-1}^\infty(k+1)C_{k+1}(\xi^{\odot k}\odot\zeta).\end{equation}

Let $B(\xi)=\sum_{k=1}^\infty B_k\xi^{\otimes k}\in\mathcal F(\Phi;\Phi)$ and $C(\xi)=\sum_{k=i}^\infty C_k\xi^{\otimes k}\in\mathcal F(\Phi;\Phi^{\hat\odot i})$. 
Substituting in formula~\eqref{cxtrswu5e4i6} the formal tensor power series $B(\xi)$ for $\zeta$, 
 we define $D_{B}C(\xi)$, the {\it derivative of $C(\xi)$ in direction $B(\xi)$}, as an element of $\mathcal F(\Phi;\Phi^{\hat\odot i})$. Indeed, similarly to formulas~\eqref{vcxdeay543q}, \eqref{crtsdtrsew}, we calculate $D_{B}C(\xi)$ as follows:
\begin{align}
 D_{B}C(\xi)&:=\sum_{k=i}^\infty kC_{k}(\xi^{\otimes (k-1)}\odot B(\xi))=\sum_{k=i}^\infty\sum_{l=1}^\infty kC_k\big(\xi^{\otimes(k-1)}\odot (B_l\xi^{\otimes l})\big)
 \notag\\&=\sum_{k=i}^\infty\sum_{l=1}^\infty kC_k\big(\mathbf 1_{k-1}\odot B_l\big)\xi^{\otimes(k-1+l)}=\sum_{k=i}^\infty D_k\xi^{\otimes k}, \label{cfdst}
 \end{align}
where
\begin{equation*}
D_k=\sum_{l=i}^{k}l\, C_{l}(\mathbf 1_{l-1}\odot B_{k-l+1}), 
\end{equation*}
and $\mathbf 1_l$ is the identity operator in $\Phi^{\hat\odot l}$.

We denote by $\mathcal F_0(\Phi)$ the set of all formal tensor power series 
$\sum_{k=0}^\infty A_k\xi^{\otimes k}\in \mathcal F(\Phi;\mathbb C)$, with $A_0=1$. We denote by $\mathcal F_1(\Phi)$ the set of all formal tensor power series 
$\sum_{k=1}^\infty B_k\xi^{\otimes k}\in \mathcal F(\Phi;\Phi)$, 
with $B_1=\mathbf 1$, the identity operator in $\Phi$. 

\subsection{The $\mathcal F_0(\Phi)$ as a Lie group for multiplication}

We define the l.c.s.\ $\mathcal W_0(\Phi)$ as the topological product of all the spaces $(\Phi^{\hat\odot k})'$ with $k\in\mathbb N$. Since each space $(\Phi^{\hat\odot k})'$ is complete, $\mathcal W_0(\Phi)$ is complete, see e.g.\ \cite[Ch.~II, \S~5.2 and 5.3]{Schaefer}.

 With an abuse of notation, we will identify $\mathcal W_0(\Phi)$ with the subspace of $\mathcal F(\Phi;\mathbb C)$ that consists of all formal tensor power series $\alpha(\xi)=\sum_{k=1}^\infty \alpha_k\xi^{\otimes k}$ ($\alpha_k\in(\Phi^{\hat\odot k})'$).
The bijective map
\begin{equation}\label{biject0}
\mathcal W_0(\Phi)\ni \alpha(\xi)\mapsto 1+\alpha(\xi)\in \mathcal F_0(\Phi)
\end{equation}
determines a topology on $\mathcal F_0(\Phi)$
which coincides with the topology 
induced by that on $\mathcal F(\Phi;\mathbb C)$.

For the definition of a Lie group modeled on a complete l.c.s.\ and its Lie algebra, see Subsections~\ref{app:LieGroup}--\ref{app:LieAlgebra} in the Appendix.

\begin{proposition} \label{ccfgrag543w}
The $\mathcal F_0(\Phi)$ equipped with the multiplicative product of formal tensor power series is a commutative Lie group modeled on $\mathcal W_0(\Phi)$ through the global parametrization   \eqref{biject0}. The corresponding Lie algebra, hence, can be identified with $\mathcal W_0(\Phi)$.
\end{proposition}

\begin{proof} To show that $\mathcal F_0(\Phi)$ is a group, we only need to check that, for each $A(\xi)=\sum_{k=0}^\infty A_k\xi^{\otimes k}\in\mathcal F_0(\Phi)$, there exists $\tilde A(\xi)=\sum_{k=0}^\infty \tilde A_k\xi^{\otimes k}\in \mathcal F_0(\Phi)$ such that $A(\xi)\tilde A(\xi)=1$. (Below we will denote $\tilde A(\xi)$ by $A^{-1}(\xi)$.) But \eqref{bftyde6w} implies that $\tilde A_k$ can be found from the recurrence formula 
\begin{equation}\label{vcter4y56}
\tilde A_0=1, \quad \tilde A_1=-A_1, \quad \tilde A_k=-A_k-\sum_{l=1}^{k-1}A_{k-l}\odot\tilde A_l\quad\text{for }k\ge2. 
\end{equation}

Next, we have to check that the maps
\begin{gather}
\mathcal F_0(\Phi)^2\ni(A^{(1)}(\xi), A^{(2)}(\xi))\mapsto A^{(1)}(\xi)A^{(2)}(\xi)\in \mathcal F_0(\Phi), \label{bvste6e6}\\
 \mathcal F_0(\Phi)\ni A(\xi)\mapsto A^{-1}(\xi)\in \mathcal F_0(\Phi)\label{vgtfdy6ew}
 \end{gather}
 are analytic. It follows from formulas~\eqref{cfrtsw5yw}, \eqref{bftyde6w} and 
Lemma~\ref{tfy7e6e} in the Appendix that, in order to prove the analyticity of the map in \eqref{bvste6e6}, it is sufficient to prove that, for each $k\in\mathbb N$, the following map is analytic:
$$\bigg(\prod_{l=1}^k (\Phi^{\hat\odot l})'\bigg)^2\ni(A_1^{(1)}, \dots, A_k^{(1)}, A_1^{(2)}, \dots, A_k^{(2)})\mapsto A_k^{(1)}+A_k^{(2)}+\sum_{l=1}^{k-1}A_l^{(1)}\odot A_{k-l}^{(2)}\in(\Phi^{\hat\odot k})'.$$
But this analyticity is a consequence of Lemma~\ref{cgtsrta} (i). The analyticity of the map in \eqref{vgtfdy6ew} follows similarly from \eqref{vcter4y56} if one additionally takes into account that the composition of analytic maps is analytic. Indeed, an induction argument shows that each 
$\tilde A_k$ ($k\ge 1$) in formula \eqref{vcter4y56} is an analytic function of $A_1, \dots, A_k$.
 \end{proof}

 For the definition of the exponential map for a Lie group, see Subsection~\ref{app:ExpMap} in the Appendix, and for the definition of an equivalent global parametrization of a Lie group, see Subsection~\ref{app:LieAlgebra} in the Appendix.


\begin{proposition}\label{cxtstew5}
The exponential map for $\mathcal F_0(\Phi)$ exists and is given by
\begin{equation}\label{crq4yxs}
\mathcal W_0(\Phi)\ni \alpha(\xi)\mapsto \exp(\alpha(\xi))=1+\sum_{n=1}^\infty\frac1{n!}\alpha(\xi)^n\in\mathcal F_0(\Phi).\end{equation}
This is a bijective map and its inverse is given by
\begin{equation}\label{xweqa4q42q}
\mathcal F_0(\Phi)\ni A(\xi)\mapsto\log(A(\xi))=\sum_{n=1}^\infty\frac{(-1)^{n+1}}n\, (A(\xi)-1)^n\in\mathcal W_0(\Phi).\end{equation}
The maps in formulas \eqref{crq4yxs} and \eqref{xweqa4q42q} are analytic; hence, the exponential map \eqref{crq4yxs} provides an equivalent global parametrization of $\mathcal F_0(\Phi)$.
\end{proposition}

\begin{proof} Fix $\alpha(\xi)=\sum_{k=1}^\infty\alpha_k\xi^{\otimes k}\in\mathcal W_0(\Phi)$. For each $t\in\R$, we define $A(t, \xi):=\exp(t\alpha(\xi))\in\mathcal F_0(\Phi)$. By \eqref{cxrsa4wq4} and \eqref{xra4waq}, we have $A(t, \xi)=1+\sum_{k=1}^\infty A_k(t)\xi^{\otimes k}$, where
\begin{equation}\label{vcrtw5y}
A_k(t)=\sum_{l=1}^k\frac{t^l}{l!}\sum_{\substack{i_1, \dots, i_l\in\mathbb N\\i_1+\dots+i_l=k}}\alpha_{i_1}\odot\cdots\odot \alpha_{i_l}, \quad k\in\mathbb N.\end{equation}
Obviously, for each $k\in\mathbb N$, the map $\R\ni t\mapsto A_k(t)\in(\Phi^{\hat\odot k})'$ is smooth. 
Hence, the map $ \R\ni t\mapsto A(t, \xi)\in\mathcal F_0(\Phi)$ is smooth. For any $s, t\in\R$, we obviously have $A(s, \xi)A(t, \xi)=A(s+t, \xi)$. Hence, $(A(t, \xi))_{t\in\R}$ is a one-parameter subgroup of~$\mathcal F_0(\Phi)$. Formula \eqref{vcrtw5y} implies that $A_k'(0)=\alpha_k$, and so $A'(0, \xi)=\alpha(\xi)$. Therefore, the exponential of $\alpha(\xi)$ is $A(1, \xi)=\exp(\alpha(\xi))$.

By Lemma~\ref{tfy7e6e} and \eqref{vcrtw5y}, to show that the map in \eqref{crq4yxs} is analytic, it is sufficient to prove that, for each $k\in\mathbb N$, the map
$$\prod_{l=1}^k (\Phi^{\hat\odot l})'\ni(\alpha_1, \dots, \alpha_k)\mapsto \sum_{l=1}^k\frac{1}{l!}\sum_{\substack{i_1, \dots, i_l\in\mathbb N\\i_1+\dots+i_l=k}}\alpha_{i_1}\odot\cdots\odot \alpha_{i_l}\in(\Phi^{\hat\odot k})'$$
is analytic. But this follows from Lemma~\ref{cgtsrta} (i) and Remark~\ref{cxfet678i9}. 

Next, it is straightforward to see that, for all $\alpha(\xi)\in\mathcal W_0(\Phi)$ and $A(\xi)\in\mathcal F_0(\Phi)$, we have $\log(\exp(\alpha(\xi)))=\alpha(\xi)$ and $\exp(\log(A(\xi)))=A(\xi)$. The analyticity of the map in~\eqref{xweqa4q42q} can be shown similarly to the analyticity of the exponential map. 
That the exponential map  \eqref{crq4yxs} provides an equivalent global parametrization
follows immediately from Lemma~\ref{analytic_reparametrization} in the Appendix.
\end{proof}

For the definition of a regular Lie group, see Subsection~\ref{app:RegLieGroup} in the Appendix.


\begin{proposition}\label{vytde6e6}
The Lie group $\mathcal F_0(\Phi)$ is regular.
\end{proposition}

\begin{proof} Recall that the map in formula \eqref{biject0} provides a global parametrization of $\mathcal F_0(\Phi)$. 
  Let $A(\xi)\in\mathcal F_0(\Phi)$ and $\alpha(\xi)\in\mathcal W_0(\Phi)$. We find:
 \begin{equation}\label{cfxdtrsts5tsw}
 \frac d{ds}\Big|_{s=0}\big(A(\xi)(1+s\alpha(\xi))-1\big)=A(\xi)\alpha(\xi)\in \mathcal W_0(\Phi). 
 \end{equation}

Let $[0, 1]\ni t\mapsto\alpha(t, \xi)=\sum_{k=1}^\infty\alpha_k(t)\xi^{\otimes k}\in\mathcal W_0(\Phi)$ be a smooth curve. In view of \eqref{cfxdtrsts5tsw}, in order to prove that the Lie group $\mathcal F_0(\Phi)$ is regular, we first need to find a smooth curve $[0, 1]\ni t\mapsto A(t, \xi)=1+\sum_{k=1}^\infty A_k(t)\xi^{\otimes k}\in\mathcal F_0(\Phi)$ that satisfies
\begin{equation}\label{drtwu6}
A'(t, \xi)=A(t, \xi)\alpha(t, \xi), \quad t\in[0, 1].
\end{equation}
An easy calculation shows that equation~\eqref{drtwu6} is equivalent to the equations
$$ A_1'(t)=\alpha_1(t), \quad A'_k(t)=\alpha_k(t)+\sum_{l=1}^{k-1}A_l(t)\odot\alpha_{k-l}(t)\quad \text{for }k\ge 2.$$
By Lemma~\ref{dtstesrew5} (i), a solution to this equation is given by the recurrence formula
\begin{equation}\label{vfdr6e6}
 A_1(t)=\int_0^t \alpha_1(r)\, dr, \quad A_k(t)=\int_0^t\bigg(\alpha_k(r)+\sum_{l=1}^{k-1}A_l(r)\odot\alpha_{k-l}(r)\bigg)\, dr\quad \text{for }k\ge 2.\end{equation}
Obviously, the resulting curve $[0,1]\ni t\mapsto A(t,\xi)$ is smooth.

Next, we have to prove that the map 
$$C^\infty([0, 1];\mathcal W_0(\Phi))\ni\alpha(\cdot, \xi)\mapsto A(1, \xi)\in\mathcal F_0(\Phi)$$
is analytic.

As easily seen, for each continuous curve $\alpha(\cdot, \xi)=\sum_{k=1}^\infty\alpha_k(\cdot)\xi^{\otimes k}\in C([0, 1];\mathcal W_0(\Phi))$, formula~\eqref{vfdr6e6} with $t=1$ still defines $ A(1, \xi)=1+\sum_{k=1}^\infty A_k(1)\xi^{\otimes k}$ as an element of $\mathcal F_0(\Phi)$. Since the space $C^\infty([0, 1];\mathcal W_0(\Phi))$ is continuously embedded into $C([0, 1];\mathcal W_0(\Phi))$, it is therefore sufficient to prove that the map 
$$C([0, 1];\mathcal W_0(\Phi))\ni\alpha(\cdot, \xi)\mapsto A(1, \xi)\in\mathcal F_0(\Phi)$$
is analytic.
 But this easily follows from Lemmas~\ref{tfy7e6e}, \ref{cgtsrta}~(ii) (see also Remark~\ref{cxfet678i9}), and~\ref{ctrsw6u4}. 
\end{proof}

\subsection{The $\mathcal F_1(\Phi)$ as a Lie group for composition}

Let the complete l.c.s.\ $\mathcal W_1(\Phi)$ be defined as the topological product of all the spaces $ \mathcal L(\Phi^{\hat\odot k}, \Phi)$ with $k\ge2$. With an abuse of notation, we will identify $\mathcal W_1(\Phi)$ with the subspace of $\mathcal F(\Phi;\Phi)$ that consists of all formal tensor power series $\beta(\xi)=\sum_{k=2}^\infty \beta_k\xi^{\otimes k}$.

The bijective map
\begin{equation}\label{vcfts5w5}
\mathcal W_1(\Phi)\ni\beta(\xi)\mapsto \xi+\beta(\xi)\in \mathcal F_1(\Phi)
\end{equation}
 determines a topology on $\mathcal F_1(\Phi)$, which 
 coincides with the topology 
induced by that on $\mathcal F(\Phi;\Phi)$.

\begin{proposition} \label{ccdyfytsfystfc}
The $\mathcal F_1(\Phi)$ equipped with the composition of formal tensor power series is a Lie group modeled on $\mathcal W_1(\Phi)$ through the global parametrization \eqref{vcfts5w5}.
The corresponding Lie algebra, hence, can be identified with  $\mathcal W_1(\Phi)$. Under this identification, the  Lie bracket on $\mathcal W_1(\Phi)$ is given by
\begin{equation}\label{cxteaaq}
[\beta^{(1)}(\xi), \beta^{(2)}(\xi)]= D_{\beta^{(2)}}\beta^{(1)}(\xi)- D_{\beta^{(1)}}\beta^{(2)}(\xi), \quad \beta^{(1)}, \beta^{(2)}\in\mathcal W_1(\Phi).
\end{equation}
\end{proposition}

\begin{proof}
 Let $B^{(j)}(\xi)=\sum_{k=1}^\infty B_k^{(j)}\xi^{\otimes k}\in\mathcal F_1(\Phi)$ ($j=1, 2$). By \eqref{vcxdeay543q} and 
 \eqref{crtsdtrsew}, we have 
\begin{equation}\label{vcdse5qcc}
B^{(1)}(B^{(2)}(\xi))=\sum_{k=1}^\infty B_k\xi^{\otimes k}\in\mathcal F_1(\Phi), 
\end{equation}
where $B_1=\mathbf 1$ and 
\begin{equation}\label{vhdtuwe}
B_k=\sum_{l=1}^k B_l^{(1)}\sum_{\substack{i_1, \dots, i_l\in\mathbb N\\i_1+\dots+i_l=k}}B_{i_1}^{(2)}\odot B_{i_2}^{(2)}\odot\cdots\odot B_{i_l}^{(2)}, \quad k\ge 2.\end{equation}

The composition is obviously an associative operation. Hence, to prove that $\mathcal F_1(\Phi)$ is a group, we have to show that, for each $B(\xi)=\sum_{k=1}^\infty B_k\xi^{\otimes k}\in\mathcal F_1(\Phi)$, there exists $B^{\langle-1\rangle}(\xi)\in\mathcal F_1(\Phi)$ such that $B(B^{\langle-1\rangle}(\xi))=B^{\langle-1\rangle}(B(\xi))=\xi$. 
First, we find $B^{(1)}(\xi)=\sum_{k=1}^\infty B^{(1)}_k\xi^{\otimes k} \in \mathcal F_1(\Phi)$ such that $B(B^{(1)}(\xi))=\xi$. Formulas~\eqref{vcdse5qcc}, \eqref{vhdtuwe} imply the recurrence formula
\begin{equation*} B_2^{(1)}=-B_2, \quad 
B^{(1)}_k= -\sum_{l=2}^{k} B_l\sum_{\substack{i_1, \dots, i_l\in\mathbb N\\i_1+\dots+i_l=k}}B_{i_1}^{(1)}\odot B_{i_2}^{(1)}\odot\cdots\odot B_{i_l}^{(1)}\quad \text{for }k\ge3, \end{equation*}
with $B^{(1)}_1=\mathbf 1$. Similarly, we find $B^{(2)}(\xi)=\sum_{k=1}^\infty B^{(2)}_k\xi^{\otimes k} \in \mathcal F_1(\Phi)$ such that $B^{(2)}(B(\xi))=\xi$ from the recurrence formula 
\begin{equation}\label{bvcyrd}
B_2^{(2)}=-B_2, \quad B^{(2)}_k= -\sum_{l=1}^{k-1} B_l^{(2)}\sum_{\substack{i_1, \dots, i_l\in\mathbb N\\i_1+\dots+i_l=k}}B_{i_1}\odot B_{i_2}\odot\cdots\odot B_{i_l}\quad\text{for } k\ge3, 
\end{equation}
with $B^{(2)}_1=\mathbf 1$. The associativity of the composition implies $B^{(1)}=B^{(2)}=B^{\langle-1\rangle}$.

Similarly to the proof of Proposition~\ref{ccfgrag543w}, we now show that $\mathcal F_1(\Phi)$ is a Lie group. Formulas \eqref{vcdse5qcc} and \eqref{vhdtuwe}, Lemmas~\ref{tfy7e6e}, \ref{gctessswe5} (i), \ref{cgtsrta} (i) and Remark~\ref{cxfet678i9} imply that the map $\mathcal F_1(\Phi)^2\ni(B^{(1)}(\xi), B^{(2)}(\xi))\mapsto B^{(1)}(B^{(2)}(\xi))\in\mathcal F_1(\Phi)$ is analytic. Analogously, formula \eqref{bvcyrd} implies that the map $\mathcal F_1(\Phi)\ni B(\xi)\mapsto B^{\langle-1\rangle}(\xi)\in\mathcal F_1(\Phi)$ is analytic. Indeed, an induction argument shows that each 
$B_k^{(2)}$ ($k\ge 2$) in formula \eqref{bvcyrd} is an analytic function of $B_2, \dots, B_k$.

To calculate the Lie bracket on $\mathcal W_1(\Phi)$, we use Subsection~\ref{app:LieAlgebra} in the Appendix.
 Let $\beta^{(j)}(\xi)=\sum_{k=2}^\infty \beta_k^{(j)}\xi^{\otimes k}\in\mathcal W_1(\Phi)$ ($j=1, 2$) and $s_1, s_2\in\mathbb C$. We have, by using \eqref{cfdst}, \eqref{vcdse5qcc} and \eqref{vhdtuwe}, 
\begin{align}
&(\xi+s_1\beta^{(1)}(\xi))\circ(\xi+s_2\beta^{(2)}(\xi))-\xi=s_2\beta^{(2)}(\xi)+s_1\sum_{k=2}^\infty \beta_k^{(1)}(\xi+s_2\beta^{(2)}(\xi))^{\otimes k}\notag\\
&\quad=s_2\beta^{(2)}(\xi)+s_1\beta^{(1)}(\xi)+s_1s_2\sum_{k=2}^\infty k\beta_k^{(1)}(\xi^{\otimes(k-1)}\odot \beta^{(2)}(\xi))+R(s_1, s_2, \xi)\notag\\
&\quad=s_1\beta^{(1)}(\xi)+s_2\beta^{(2)}(\xi)+s_1s_2\, D_{\beta^{(2)}}\beta^{(1)}(\xi)+R(s_1, s_2, \xi), \label{xear4aq435}
\end{align}
where $R(s_1, s_2, \xi)$ is the sum of the higher order terms in $s_1$ and $s_2$. Formula \eqref{xear4aq435} implies~\eqref{cxteaaq}. 
\end{proof}

\begin{proposition}\label{gfxsetser5}
The exponential map 
\begin{equation}\label{EXP1}
\mathcal W_1 (\Phi)\ni \beta(\xi)\mapsto\operatorname{Exp}(\beta(\xi))\in\mathcal F_1(\Phi)
\end{equation}
exists for $\mathcal F_1(\Phi)$. This map is bijective and we denote its inverse by
\begin{equation*}
\mathcal F_1(\Phi) \ni B(\xi)\mapsto\operatorname{Log}(B(\xi))\in\mathcal W_1(\Phi).
\end{equation*}
  Both maps $\operatorname{Exp}$ and $\operatorname{Log}$ are analytic;  hence, the exponential map $\operatorname{Exp}$  provides an equivalent global parametrization of $\mathcal F_1(\Phi)$. In particular, under the global parametrization $\operatorname{Exp}$, the Lie bracket on $\mathcal W_1(\Phi)$ is still given by  \eqref{cxteaaq}.
\end{proposition}

\begin{proof} We divide the proof into several steps.

{\it Step 1}. Let $\beta(\xi)=\sum_{k=2}^\infty \beta_k\xi^{\otimes k}\in\mathcal W_1(\Phi)$ and $\gamma(\xi)=\xi+\sum_{k=2}^\infty \gamma_k\xi^{\otimes k}\in\mathcal F_1(\Phi)$. We state that there exists a unique smooth curve
$\R\ni t \mapsto B(t, \xi)=\xi+\sum_{k=2}^\infty B_k(t)\xi^{\otimes k}\in\mathcal F_1(\Phi)$ such that 
\begin{align}
B'(t, \xi)&=\beta(B(t, \xi)), \quad t\in\R, \label{rewq}\\
B(0, \xi)&=\gamma(\xi).\label{darwa4w}
\end{align}
 Indeed, in view of formulas \eqref{vcxdeay543q}, \eqref{crtsdtrsew}, the differential equation \eqref{rewq} means
\[
B_2'(t)=\beta_2, \quad B'_k(t)=\beta_k+\sum_{l=2}^{k-1}\beta_l \sum_{\substack{i_1, \dots, i_l\in\mathbb N\\i_1+\dots+i_l=k}}B
_{i_1}(t)\odot B_{i_2}(t)\odot\cdots\odot B_{i_l}(t)\quad\text{for }k\ge3.
\]
Hence, by Lemma~\ref{dtstesrew5} (ii), the unique solution of the Cauchy problem \eqref{rewq}, \eqref{darwa4w} is given by the recurrence formula
\begin{gather}
B_2(t)=\gamma_2+t\beta_2, \notag\\
B_k(t)=\gamma_k+ t\beta_k+\sum_{l=2}^{k-1} \sum_{\substack{i_1, \dots, i_l\in\mathbb N\\i_1+\dots+i_l=k}}\beta_l\int_0^t B
_{i_1}(r)\odot B_{i_2}(r)\odot\cdots\odot B_{i_l}(r)\, dr\quad\text{for } k\ge3.\label{cxreay4}
\end{gather}
 It is easy to see by induction that 
\begin{equation}\label{cxera4y5}
B_k(t)=\gamma_k+t\beta_k+\sum_{m=1}^{k-1}t^m\rho_{k, m}(\gamma_2, \dots, \gamma_{k-1}, \beta_2, \dots, \beta_{k-1}), \quad k\ge3, 
\end{equation}
where each $\rho_{k, m}(\gamma_2, \dots, \gamma_{k-1}, \beta_2, \dots, \beta_{k-1})\in\mathcal L(\Phi^{\hat\odot k}, \Phi)$ is completely determined by $\gamma_2, \dots, \gamma_{k-1}, \beta_2, \dots, \beta_{k-1}$.
In particular, $\mathbb R\ni t\mapsto B(t, \xi)$ is a smooth curve in $\mathcal F_1(\Phi)$.

{\it Step 2}. Let $B(t, \xi)$ be the solution to the Cauchy problem \eqref{rewq}, \eqref{darwa4w} in which $\gamma(\xi)=\xi$. A standard argument shows that $\big(B(t, \xi)\big)_{t\in\R}$ is a one-parameter subgroup of $\mathcal F_1(\Phi)$. Indeed, fix $s\in\R$ and define $\theta(t, \xi):=B(s+t, \xi)$ for $t\in\R$.
Then, 
\begin{equation}\label{cxra4aq4}
\theta'(t, \xi)=\frac d{dr}\Big|_{r=s+t}B(r, \xi)=\beta(B(s+t, \xi))=\beta(\theta(t, \xi)), \quad \theta(0, \xi)=B(s, \xi).
\end{equation}
Next, define the curve $\psi(t, \xi):=B(t, B(s, \xi))$. Then
\begin{equation}\label{a4bvdr}
\psi'(t, \xi)=B'(t, B(s, \xi))=\beta(B(t, B(s, \xi)))=\beta(\psi(t, \xi)), \quad \psi(0, \xi)=B(s, \xi).
\end{equation}
Formulas \eqref{cxra4aq4} and \eqref{a4bvdr} imply that $\theta(t, \xi)=\psi(t, \xi)$ for all $t\in\R$, i.e., $B(s+t, \xi)=B(t, B(s, \xi))$. By \eqref{cxera4y5}, we have $B'(0, \xi)=\beta(\xi)$, and so $B(t, \xi)=\operatorname{Exp}(t\beta(\xi))$. 

Denoting $B(\xi)=\xi+\sum_{k=2}^\infty B_k\xi^{\otimes k}:=B(1, \xi)=\operatorname{Exp}(\beta(\xi))\in\mathcal F_1(\Phi)$, we conclude from formula~\eqref{cxera4y5} that
\begin{equation}\label{vcxtaq}
B_2=\beta_2, \quad B_k=\beta_k+\sum_{m=1}^{k-1}\eta_{k, m}(\beta_2, \dots, \beta_{k-1})\quad \text{for } k\ge3, \end{equation}
where 
\begin{equation}\label{fss5ws55}\eta_{k, m}(\beta_2, \dots, \beta_{k-1}):=\rho_{k, m}(0, \dots, 0, \beta_2, \dots, \beta_{k-1}).
\end{equation}

{\it Step 3}. We state that, for each $B(\xi)=\xi+\sum_{k=2}^\infty B_k\xi^{\otimes k}\in\mathcal F_1(\Phi)$, there exists a unique $\beta(\xi)=\sum_{k=2}^\infty \beta_k\xi^{\otimes k}\in\mathcal W_1(\Phi)$ such that $\operatorname{Exp}(\beta(\xi))=B(\xi)$.
Indeed, formula \eqref{vcxtaq} implies the recurrence formula: 
\begin{equation}\label{fxdraq5y}
\beta_2=B_2, \quad \beta_k=B_k-\sum_{m=1}^{k-1}\eta_{k, m}(\beta_2, \dots, \beta_{k-1})\quad\text{for }k\ge3.
\end{equation}
 Thus, the $\operatorname{Log}$ map exists and $\beta(\xi)=\operatorname{Log}(B(\xi))$.

{\it Step 4}. Similarly to the proof of Proposition~\ref{ccdyfytsfystfc}, we easily conclude by induction from formulas~\eqref{cxreay4}, \eqref{cxera4y5} and \eqref{fss5ws55} that each map $(\beta_2, \dots, \beta_{k-1})\mapsto \eta_{k, m}(\beta_2, \dots, \beta_{k-1})$ is analytic. Hence, formulas \eqref{vcxtaq} and \eqref{fxdraq5y} imply that both maps $\operatorname{Exp}$ and $\operatorname{Log}$ are analytic. 
The fact that the exponential map $\operatorname{Exp}$ provides an equivalent global parametrization of $\mathcal F_1(\Phi)$ follows from Lemma~\ref{analytic_reparametrization} in the Appendix. 
\end{proof}



\begin{proposition}\label{vcte6u}
The Lie group $\mathcal F_1(\Phi)$ is regular.
\end{proposition}

\begin{proof}  Recall that the map in formula \eqref{vcfts5w5} provides a global parametrization of $\mathcal F_1(\Phi)$.
  Let $B(\xi)=\xi+\sum_{k=2}^\infty B_k\xi^{\otimes k}\in\mathcal F_1(\Phi)$ and $\beta(\xi)=\sum_{k=2}^\infty\beta_k\xi^{\otimes k}\in\mathcal W_1(\Phi)$.  For $s\in\mathbb C$, we have
\begin{equation}\label{cxrear4aqaa}
B(\xi+s\beta(\xi))=B(\xi)+s\bigg(\beta(\xi)+\sum_{k=2}^\infty kB_k(\xi^{\otimes(k-1)}\odot\beta(\xi))\bigg)+R(s, \xi), 
\end{equation}
where $R(s, \xi)$ is the sum of the higher order terms in $s$. Then, by \eqref{cfdst} and \eqref{cxrear4aqaa}, 
\begin{equation}\label{cxdreaw4t}
\frac{d}{ds}\Big|_{s=0}\big(B(\xi+s\beta(\xi))-\xi\big)=\beta(\xi)+\sum_{k=2}^\infty kB_k(\xi^{\otimes(k-1)}\odot\beta(\xi))= D_{\beta}B(\xi)\in \mathcal W_1(\Phi).\end{equation}

Let now $[0, 1]\ni t\mapsto \beta(t, \xi)=\sum_{k=2}^\infty\beta_k(t)\xi^{\otimes k}\in\mathcal W_1(\Phi)$ be a smooth curve. In view of~\eqref{cxdreaw4t}, in order to prove that the Lie group $\mathcal F_1(\Phi)$ is regular, we first need to find a smooth curve  $[0, 1]\ni t\mapsto B(t, \xi)=\xi+\sum_{k=2}^\infty B_k(t)\xi^{\otimes k}\in\mathcal F_1(\Phi)$ that satisfies
\begin{align}
\sum_{k=2}^\infty B'_k(t)\xi^{\otimes k}&= \beta(t, \xi)+\sum_{k=2}^\infty k B_k(t)(\xi^{\otimes(k-1)}\odot\beta(t, \xi))\label{yufr7r7}\\
&= \beta(t, \xi)+\sum_{k=3}^\infty \sum_{j=2}^{k-1}j B_j(t)\big(\mathbf 1_{j-1}\odot\beta_{k-j+1}(t)\big)\xi^{\otimes k}.\label{vcfts5uw}
\end{align}
Equation~\eqref{vcfts5uw} is equivalent to the system of equations
\begin{equation}
 B_2'(t)=\beta_2(t), \quad B'_k(t)=\beta_k(t)+\sum_{j=2}^{k-1}j B_j(t)\big(\mathbf 1_{j-1}\odot\beta_{k-j+1}(t)\big), \quad k\ge 3.\label{vydy6e}
\end{equation}
A solution to \eqref{vydy6e} is given by the recurrence formula
\begin{align}
 B_2(t)&=\int_0^t\beta_2(r)\, dr, \notag\\
 B_k(t)&=\int_0^t\bigg(\beta_k(r)+\sum_{j=2}^{k-1}j B_j(r)\big(\mathbf 1_{j-1}\odot\beta_{k-j+1}(r)\big)\bigg)dr, \quad k\ge 3.\label{dstere5w5w}
\end{align}
Clearly, the resulting curve $[0,1]\ni t\mapsto B(t,\xi)$ is smooth. 

Next, with the help of Lemmas~\ref{tfy7e6e}, \ref{gctessswe5} (ii), \ref{cgtsrta} (ii), and \ref{ctrsw6u4}, the analyticity of the map
$ C^\infty([0, 1];\mathcal W_1(\Phi))\ni\beta(\cdot, \xi)\mapsto B(1, \xi)\in\mathcal F_1(\Phi)$
easily follows by analogy with the proof of Proposition~\ref{vytde6e6}. 
\end{proof}

\subsection{The semidirect product of $\mathcal F_0(\Phi)$ and $\mathcal F_1(\Phi)$}

Denote $\mathcal S(\Phi):=\mathcal F_0(\Phi) \times\mathcal F_1(\Phi)$. We define a product $\ast$ in $\mathcal S(\Phi)$ by 
\begin{equation}\label{vcyrtsw5u}
(A^{(1)}(\xi), B^{(1)}(\xi))\ast (A^{(2)}(\xi), B^{(2)}(\xi))=(A(\xi), B(\xi)), 
\end{equation}
where
\begin{equation}\label{tsay4w}
 A(\xi)=A^{(1)}(B^{(2)}(\xi)) A^{(2)}(\xi), \quad B(\xi)=B^{(1)}(B^{(2)}(\xi)).
 \end{equation}
 As easily seen, the $\mathcal S(\Phi)$ equipped with product $\ast$ is a group. In particular, for $(A(\xi), B(\xi))\in\mathcal S(\Phi)$, its inverse is given by
\begin{equation}\label{vfyrds5ws5a}
(A(\xi), B(\xi))^{-1}=(A^{-1}(B^{\langle -1\rangle}(\xi)), \, B^{\langle -1\rangle}(\xi)).\end{equation}

 By identifying $A(\xi)\in\mathcal F_0(\Phi)$ and $B(\xi)\in\mathcal F_1(\Phi)$ with $(A(\xi), \xi)$ and $(1, B(\xi))$, respectively, we easily see that
 $\mathcal F_0(\Phi)$ is an abelian normal subgroup of $\mathcal S(\Phi)$, and $\mathcal F_1(\Phi)$ is a subgroup of $\mathcal S(\Phi)$. Furthermore, each $(A(\xi), B(\xi))\in \mathcal S(\Phi)$ admits a unique representation as a product of elements from $\mathcal F_0(\Phi)$ and $\mathcal F_1(\Phi)$:
 $$(A(\xi), B(\xi))=(A(B^{\langle-1\rangle}(\xi)), \xi)\ast(1, B(\xi)).$$
 Hence, $\mathcal S(\Phi)$ is the semidirect product of $\mathcal F_0(\Phi)$ and $\mathcal F_1(\Phi)$, i.e., $\mathcal S(\Phi):=\mathcal F_0(\Phi) \rtimes\mathcal F_1(\Phi)$.

Denote $\mathcal W(\Phi):=\mathcal W_0(\Phi)\times\mathcal W_1(\Phi)$, which is a complete l.c.s. The bijective map
\begin{equation}\label{vcstsreaea}
\mathcal T:\mathcal W(\Phi)\to \mathcal S(\Phi), \quad \mathcal T(\alpha(\xi), \beta(\xi))= (1+\alpha(\xi), \xi+\beta(\xi))
\end{equation}
determines a topology on $\mathcal S(\Phi)$.

\begin{proposition} \label{vtdfrdserse54s45sw}
The $\mathcal S(\Phi)$ is a Lie group modeled on $\mathcal W(\Phi)$ through the global parametrization $\mathcal T$. The corresponding Lie algebra, hence, can be identified with  $\mathcal W(\Phi)$. Under this identification, the  Lie bracket on $\mathcal W(\Phi)$ is given by formulas \eqref{cxtewu65}--\eqref{cfre6u4e}.
\end{proposition}

\begin{proof} In view of Propositions~\ref{ccfgrag543w} and~\ref{ccdyfytsfystfc} and formulas \eqref{vcyrtsw5u}--\eqref{vcstsreaea}, to show that $\mathcal S(\Phi)$ is a Lie group, we only need to prove that the map
$\mathcal S(\Phi)\ni(A(\xi), B(\xi))\mapsto A(B(\xi))\in\mathcal F_0(\Phi)$
is analytic. But this can be easily done by analogy with the proof of the fact that the composition of formal tensor power series from $\mathcal F_1(\Phi)$ is analytic. 

In view of the definition of the product in $\mathcal S(\Phi)$, the fact that $\beta(\xi)$ in formula~\eqref{cxtewu65} is given by \eqref{cfre6u4e} follows immediately from \eqref{cxteaaq}. The proof of formula \eqref{bvcyed6} is analogous to that of formula~\eqref{cxteaaq}. 
\end{proof}

\begin{proposition}\label{cydsawa}
The exponential map
\begin{equation}\label{cfdt6ue}
\mathcal W(\Phi)\ni(\alpha(\xi), \beta(\xi))\mapsto\operatorname{EXP}(\alpha(\xi), \beta(\xi))\in\mathcal S(\Phi)
\end{equation}
exists and is given by
\begin{equation}\label{vtyrswu53w}
\operatorname{EXP}(\alpha(\xi), \beta(\xi))=\bigg(\exp\bigg[\int_0^1\alpha\big(\operatorname{Exp}(r\beta(\xi))\big)\, dr\bigg], \, \operatorname{Exp}(\beta(\xi))\bigg).
\end{equation}
The exponential map in \eqref{cfdt6ue} is bijective and we denote its inverse by
\begin{equation*}
\mathcal S(\Phi)\ni(A(\xi), B(\xi))\mapsto \operatorname{LOG}(A(\xi), B(\xi))\in\mathcal W(\Phi).
\end{equation*}
Both maps $\operatorname{EXP}$ and $\operatorname{LOG}$ are analytic; hence, the exponential map $\operatorname{EXP}$  provides an equivalent global parametrization of $\mathcal S(\Phi)$. In particular, under the global parametrization $\operatorname{EXP}$, the Lie bracket on  $\mathcal W(\Phi)$ is still given by formulas \eqref{cxtewu65}--\eqref{cfre6u4e}.
\end{proposition}

\begin{proof} Let $(\alpha(\xi), \beta(\xi))\in\mathcal W(\Phi)$. For $t\in\R$, define $\theta(t, \xi):=\operatorname{Exp}(t\beta(\xi))\in\mathcal F_1(\Phi)$ and
\begin{equation}\label{ctru}
\psi(t, \xi):=\exp\bigg[\int_0^t \alpha(\theta(r, \xi))\, dr\bigg]\in\mathcal F_0(\Phi).
\end{equation}
We state that $\big(\psi(t, \xi), \theta(t, \xi)\big)_{t\in\R}$ is a one-parameter subgroup of $\mathcal S(\Phi)$. In view of~\eqref{vcyrtsw5u} and~\eqref{tsay4w}, we have, for $s, t\in\R$, 
$$(\psi(s, \xi), \theta(s, \xi))\ast (\psi(t, \xi), \theta(t, \xi))=( A(s, t, \xi), \theta(s+t, \xi)), $$
where
\begin{align*}
 A(s, t, \xi)&=\exp\bigg[\int_0^{s} \alpha(\theta(r, \theta(t, \xi)))\, dr+\int_0^{t} \alpha(\theta(r, \xi))\, dr\bigg]\\
&=\exp\bigg[\int_0^{s} \alpha(\theta(r+t, \xi))\, dr+\int_0^{t} \alpha(\theta(r, \xi))\, dr\bigg]\\
&=\exp\bigg[\int_{t}^{s+t} \alpha(\theta(r, \xi))\, dr+\int_0^{t} \alpha(\theta(r, \xi))\, dr\bigg]=\psi(s+t, \xi).
\end{align*}
Since $\R\ni t\mapsto \theta(t, \xi)\in \mathcal F_1(\Phi)$ is a smooth curve, formula \eqref{ctru} implies that $\R\ni t \mapsto \psi(t, \xi)\in \mathcal F_0(\Phi)$ is a smooth curve. Thus, $\big(\psi(t, \xi), \theta(t, \xi)\big)_{t\in\R}$ is indeed a one-parameter subgroup of $\mathcal S(\Phi)$. By construction, $\frac{d}{dt}\big|_{t=0}\theta(t, \xi)=\beta(\xi)$. 
Furthermore, $\frac{d}{dt}\psi(t, \xi)=\psi(t, \xi)\alpha(\theta(t, \xi))$, 
which implies $\frac{d}{dt}\big|_{t=0}\psi(t, \xi)=\alpha(\xi)$. 
Hence, formula \eqref{vtyrswu53w} is proven.

Next, we will show that, for an arbitrary $(A(\xi), B(\xi))\in\mathcal S(\Phi)$, there exists a unique $(\alpha(\xi), \beta(\xi))\in\mathcal W(\Phi)$ such that $\operatorname{EXP}(\alpha(\xi), \beta(\xi))=(A(\xi), B(\xi))$, i.e., $(\alpha(\xi), \beta(\xi))=\operatorname{LOG}(A(\xi), B(\xi))$. 

By \eqref{vcyrtsw5u}, \eqref{tsay4w} and Proposition~\ref{gfxsetser5}, we must have $\beta(\xi)=\operatorname{Log}(B(\xi))$. Let $\theta(t, \xi)=\sum_{k=1}^\infty \theta_k(t)\xi^{\otimes k}:=\operatorname{Exp}(t\beta(\xi))$ ($t\in\R$) be the corresponding one-parameter subgroup of $\mathcal F_1(\Phi)$. Define $\tilde \alpha(\xi)=\sum_{k=1}^\infty\tilde \alpha_k\xi^{\otimes k}:=\log(A(\xi))\in\mathcal W_0(\Phi)$. In view of~\eqref{vtyrswu53w}, it is sufficient to prove that there exists a unique $\alpha(\xi)=\sum_{k=1}^\infty \alpha_k\xi^{\otimes k}\in \mathcal W_0(\Phi)$ such that 
$
\tilde \alpha(\xi)=\int_0^1 \alpha(\theta(r, \xi))\, dr$. But the latter equation is equivalent to
$$\tilde\alpha_k=\sum_{l=1}^k \alpha_l \sum_{\substack{i_1, \dots, i_l\in\mathbb N\\i_1+\dots+i_l=k}}\int_0^1 \theta
_{i_1}(r)\odot \theta_{i_2}(r)\odot\cdots\odot \theta_{i_l}(r)\, dr, \quad k\in\mathbb N.$$
The unique solution to this system of equations is given by the recurrence formula
$$\alpha_1=\tilde \alpha_1, \quad\alpha_k=\tilde \alpha_k-\sum_{l=1}^{k-1} \alpha_l \sum_{\substack{i_1, \dots, i_l\in\mathbb N\\i_1+\dots+i_l=k}}\int_0^1 \theta
_{i_1}(r)\odot \theta_{i_2}(r)\odot\dots\odot \theta_{i_l}(r)\, dr\quad \text{for }k\ge2.$$

That both maps $\operatorname{EXP}$ and $\operatorname{LOG}$ are analytic easily follows from Propositions~\ref{cxtstew5} and~\ref{gfxsetser5} and their proofs; see, in particular, formula \eqref{cxera4y5} with $\gamma_k=0$. Finally, the fact that the exponential map $\operatorname{EXP}$ provides an equivalent global parametrization of $\mathcal S(\Phi)$ follows from Lemma~\ref{analytic_reparametrization} in the Appendix. 
\end{proof}


\begin{proposition}\label{acufytFTCD}
The Lie group $\mathcal S(\Phi)$ is regular.
\end{proposition}

\begin{proof} 
 Recall that the map $\mathcal T$ provides a global parametrization of $\mathcal S(\Phi)$.
  Let $(A(\xi), B(\xi))\in\mathcal S(\Phi)$ and $(\alpha(\xi), \beta(\xi))\in\mathcal W(\Phi)$. Similarly to~\eqref{cxrear4aqaa}, we have, for $s\in\mathbb C$, 
\begin{align*}
&(A(\xi), B(\xi))\ast(1+s\alpha(\xi), \xi+s\beta(\xi))=\big(A(\xi+s\beta(\xi))(1+s\alpha(\xi)), \, B(\xi+s\beta(\xi))\big)\\
&\quad =\big(A(\xi)+s D_{\beta}A(\xi)+sA(\xi)\alpha(\xi)+R_1(s, \xi), \, B(\xi)+s D_{\beta}B(\xi)+R_2(s, \xi)\big), 
\end{align*}
where $R_1(s, \xi)$ and $R_2(s, \xi)$ are the sums of the higher order terms in $s$.  Hence, 
\begin{equation}\label{vctesese}
\frac{d}{ds}\Big|_{s=0}\mathcal T^{-1}\big((A(\xi), B(\xi))\ast(1+s\alpha(\xi), \xi+s\beta(\xi))\big)=\big( D_{\beta}A(\xi)+A(\xi)\alpha(\xi), \, D_{\beta}B(\xi)\big). 
\end{equation}

Let now 
$$[0, 1]\ni t\mapsto (\alpha(t, \xi), \beta(t, \xi))=\bigg(\sum_{k=1}^\infty \alpha_k(t)\xi^{\otimes k}, \sum_{k=2}^\infty\beta_k(t)\xi^{\otimes k}\bigg)\in\mathcal W(\Phi)$$
 be a smooth curve. In view of \eqref{vctesese}, we have to find a smooth curve 
 $$[0, 1]\ni t\mapsto (A(t, \xi), B(t, \xi))=\bigg(1+\sum_{k=1}^\infty A_k(t)\xi^{\otimes k}, \, \xi+\sum_{k=2}^\infty B_k(t)\xi^{\otimes k}\bigg)\in\mathcal S(\Phi)$$
 that satisfies
 \begin{align}
 A'(t, \xi)&= \big(D_{\beta(t, \cdot)}A(t, \cdot)\big)(\xi)+A(t, \xi)\alpha(t, \xi), \label{teww5w}\\
 B'(t, \xi)&= \big(D_{\beta(t, \cdot)}B(t, \cdot)\big)(\xi).\label{vtyde6i4}
 \end{align}
Let $B(t, \xi)$ be the solution to equation \eqref{vtyde6i4} that is given by \eqref{dstere5w5w}. 
Similarly to~\eqref{yufr7r7}, equation \eqref{teww5w} can be written in the form
\begin{equation}
\sum_{k=1}^\infty A_k'(t)\xi^{\otimes k}= A_1(t)\beta(t, \xi)+\sum_{k=2}^\infty kA_k(t)\big(\xi^{\otimes(k-1)}\odot\beta(t, \xi)\big)+A(t, \xi)\alpha(t, \xi).\label{bvgutr7ir4}
\end{equation}
Similarly to \eqref{vydy6e}, \eqref{dstere5w5w}, a solution to equation \eqref{bvgutr7ir4} is given by the recurrence formula
\begin{align*}
A_1(t)&=\int_0^t \alpha_1(r)\, dr, \\
A_k(t)&= \int_0^t\bigg(\sum_{i=1}^{k-1}iA_i(r)(\mathbf 1_{i-1}\odot \beta_{k-i+1}(r))+\alpha_k(r)+\sum_{i=1}^{k-1} A_i(r)\odot\alpha_{k-i}(r)\bigg)dr, \quad k\ge2.
\end{align*}
The rest of the proof is analogous to that of Propositions~\ref{vytde6e6} and \ref{vcte6u}.
\end{proof}

\section{The Lie group of operators defining monic\\ polynomial sequences}\label{tdr6e65e48}

\subsection{Monic polynomial sequences}
 For any finite sequence $(f^{(k)})_{k=0}^K$ with $f^{(k)}\in\Phi^{\hat\odot k}$ and $K\in\mathbb N_0$, the function
\begin{equation}\label{vgdyeswse}
p(\omega)=\sum_{k=0}^K\langle \omega^{\otimes k}, f^{(k)}\rangle, \quad\omega\in\Phi', 
\end{equation}
is called a {\it polynomial on $\Phi'$}. (In \eqref{vgdyeswse}, we denote $\langle \omega^{\otimes 0}, f^{(0)}\rangle{:= }f^{(0)}$.) If in \eqref{vgdyeswse} $f^{(K)}\ne0$, then we say that the {\it degree of the polynomial $p$} is $K$. We denote by $\mathcal P(\Phi')$ and $\mathcal P^{(n)}(\Phi')$ the vector space of all polynomials on $\Phi'$ and the vector space of all polynomials on $\Phi'$ of degree $\le n$, respectively.

Denote by $\mathcal F_{\mathrm{fin}}(\Phi)$ the vector space of all sequences $f=(f^{(k)})_{k=0}^\infty$ such that $f^{(k)}\in\Phi^{\hat\odot k}$ and, for some $K\in\mathbb N_0$ (depending on $f$), we have $f^{(k)}=0$ for all $k>K$. We equip $\mathcal F_{\mathrm{fin}}(\Phi)$ with the topology of the locally convex direct sum of the spaces $\Phi^{\hat\odot k}$, see e.g.\ \cite[p.~55]{Schaefer}. Since each (LB)-space $\Phi^{\hat\odot k}$ is complete, by e.g.\ \cite[Ch.II, \S~6.2]{Schaefer}, $\mathcal F_{\mathrm{fin}}(\Phi)$ is complete.

Define a linear map $I:\mathcal F_{\mathrm{fin}}(\Phi)\to \mathcal P(\Phi')$ by
$(If)(\omega):=\sum_{k=0}^\infty \langle \omega^{\otimes k}, f^{(k)}\rangle$
for $f=(f^{(k)})_{k=0}^\infty\in \mathcal F_{\mathrm{fin}}(\Phi)$. (The series in the definition of $I$ is in fact finite.) It easily follows from Lemma~\ref{fsrea45y} that $I$ is a bijection. The map $I$ determines a complete locally convex topology on $\mathcal P(\Phi')$.

We denote by $\mathbb U(\Phi)$ the vector space of all continuous linear operators $P\in\mathcal L(\mathcal P(\Phi'))$ that satisfy the condition $P\big(\mathcal P^{(n)}(\Phi')\big)\subset \mathcal P^{(n)}(\Phi')$ for all $n\in\mathbb N_0$;
equivalently each $P\in\mathbb U(\Phi)$ acts as follows:
\begin{equation}\label{dr6e6}
\big(P\langle\cdot^{\otimes k}, f^{(k)}\rangle\big)(\omega)=\sum_{i=0}^k \langle\omega^{\otimes i}, P_{ik}f^{(k)}\rangle, \quad f^{(k)}\in\Phi^{\hat\odot k}, \ k\in\mathbb N_0, \end{equation}
where $P_{ik}\in \mathcal L(\Phi^{\hat\odot k}, \Phi^{\hat\odot i})$. By an abuse of notation, we will also think of $P\in\mathbb U(\Phi)$ as the infinite upper triangular block matrix $[P_{ik}]_{i, k\in\mathbb N_0}$ in which $P_{ik}\in\mathcal L(\Phi^{\hat\odot k}, \Phi^{\hat\odot i})$. We equip~$\mathbb U(\Phi)$ with the complete locally convex topology of the product of the spaces $\mathcal L(\Phi^{\hat\odot k}, \Phi^{\hat\odot i})$ with $0\le i\le k$. 

We denote by $\mathbb M(\Phi)$ the subset of $\mathbb U(\Phi)$ that consists of all $P=[P_{ik}]_{i, k\in\mathbb N_0}\in\mathbb U(\Phi)$ with $P_{kk}=\mathbf 1_k$ for all $k\in\mathbb N_0$. For $P\in \mathbb M(\Phi)$ and $f^{(k)}\in\Phi^{\hat\odot k}$ ($k\in\mathbb N_0$), we denote by $P^{(k)}(\omega, f^{(k)})$ the polynomial on the right-hand side of formula~\eqref{dr6e6}. Then we call $(P^{(k)}(\omega, f^{(k)}))_{f^{(k)}\in\Phi^{\hat\odot k}, \, k\in\mathbb N_0}$ a {\it monic polynomial sequence on~$\Phi'$ (corresponding to the operator $P$)}.

\begin{remark}\label{ghcydydsszza} It follows from Remark~\ref{cfsthewdd} that an operator $P\in\mathbb M(\Phi)$ is uniquely cha\-ra\-cterized by the polynomials $P^{(k)}(\omega, \xi^{\otimes k})$ with $\xi\in\Phi$ and $k\in\mathbb N_0$. \end{remark}

Let $\mathbb V(\Phi)$ denote the closed vector subspace of $\mathbb U(\Phi)$ that consists of all $V=[V_{ik}]_{i, k\in\mathbb N_0}\in\mathbb U(\Phi)$ such that $V_{kk}=\mathbf 0$
for all $k\in\mathbb N_0$. Thus, $\mathbb V(\Phi)$ is the space of all continuous linear operators $V\in\mathcal L(\mathcal P(\Phi'))$ that satisfy the condition
\begin{equation}\label{cxzraw4t}
 V\big(\mathcal P^{(n)}(\Phi')\big)\subset \mathcal P^{(n-1)}(\Phi'), \quad n\in\mathbb N_0, 
 \end{equation} 
where $\mathcal P^{(-1)}(\Phi'):=\{0\}$. Note that $\mathbb V(\Phi)$ is equipped with the topology of the topological product of the spaces $\mathcal L(\Phi^{\hat\odot k}, \Phi^{\hat\odot i})$ with $0\le i<k$.

Consider the bijective map 
\begin{equation}\label{cxdzdeszesaease}\mathbb V(\Phi)\ni V\mapsto \mathbf 1+V\in \mathbb M(\Phi),\end{equation}  where $\mathbf 1$ is the identity operator in $\mathcal P(\Phi')$. This map determines a topology on~$\mathbb M(\Phi)$. Thus, $\mathbb M(\Phi)$ is a manifold modeled on the complete l.c.s.\ $\mathbb V(\Phi)$ through the global parametrization \eqref{cxdzdeszesaease}. Note that the constructed topology on $\mathbb M(\Phi)$ coincides with the topology on $\mathbb M(\Phi)$ induced by that on $\mathbb U(\Phi)$.

\subsection{$\mathbb M(\Phi)$ as a Lie group}

Let $P^{(l)}=[P^{(l)}_{ik}]_{i, k\in\mathbb N_0}\in\mathbb U(\Phi)$ ($l=1, 2$), and let $P:=P^{(1)}P^{(2)}$ be the usual product (composition) of the continuous linear operators $P^{(1)}$ and $P^{(2)}$. As easily seen, $P\in\mathbb U(\Phi)$ and $[P_{ik}]_{i, k\in\mathbb N_0}$, the block matrix of~$P$, is equal to the product of the block matrices of $P^{(1)}$ and $P^{(2)}$, i.e., 
\begin{equation}\label{rtse5wu656}
P_{ik}=\sum_{j=i}^kP^{(1)}_{ij}P^{(2)}_{jk}, \quad i, k\in\mathbb N_0.\end{equation}

\begin{proposition}\label{fgeraq}
The $\mathbb M(\Phi)$ is a Lie group modeled on $\mathbb V(\Phi)$ through the global parametrization   \eqref{cxdzdeszesaease}. The corresponding Lie algebra, hence, can be identified with  $\mathbb V(\Phi)$. Under this identification, the  Lie bracket on $\mathbb V(\Phi)$ is given by
\begin{equation}\label{xra}
[V^{(1)}, V^{(2)}]=V^{(1)}V^{(2)}-V^{(2)}V^{(1)}, \quad V^{(1)}, V^{(2)}\in\mathbb V(\Phi).
\end{equation}
\end{proposition}

\begin{proof} Formula \eqref{rtse5wu656} implies that $\mathbb M(\Phi)$ is closed under multiplication of operators. 
The identity operator $\mathbf 1$ in $\mathcal P(\Phi')$ is obviously an element of $\mathbb M(\Phi)$. Thus, to prove that $\mathbb M(\Phi)$ is a group, we only need to show that each $P=[P_{ik}]_{i, k\in\mathbb N_0}\in\mathbb M(\Phi)$ has inverse $P^{-1}\in\mathbb M(\Phi)$. 

We state that there exists $Q=[Q_{ik}]_{i, k\in\mathbb N_0}\in
\mathbb M(\Phi)$ such that $QP=\mathbf 1$. Indeed, formula \eqref{rtse5wu656} implies that $Q_{ii}=\mathbf 1_i$ for all $i\in\mathbb N_0$, and for each fixed $i\in\mathbb N_0$, we have the recurrence formula 
\begin{equation}\label{s4aq43q5}
Q_{i\, i+1}=-P_{i\, i+1}, \quad Q_{ik}=-P_{ik}-\sum_{j=i+1}^{k-1}Q_{ij}P_{jk}, \quad k\ge i+2.
\end{equation}
 Similarly, we find $R=[R_{ik}]_{i, k\in\mathbb N_0}\in \mathbb M(\Phi)$ such that $PR=\mathbf 1$. In this case, $R_{kk}=\mathbf 1_k$ for all $k\in\mathbb N_0$, and for each fixed $k\in\mathbb N$, we have the recurrence formula
 $$ R_{k-1\, k}=-P_{k-1\, k}, \quad R_{ik}=-P_{ik}-\sum_{j=i+1}^{k-1}P_{ij}R_{jk}, \quad i=k-2, k-3, \dots, 0.$$
Finally, the associativity of the product of linear operators implies $Q=R=P^{-1}$.

 Next, to show that $\mathbb M(\Phi)$ is a Lie group, we have to prove that the maps 
\begin{equation*}
\mathbb M(\Phi)^2\ni(P^{(1)}, P^{(2)})\mapsto P^{(1)}P^{(2)}\in\mathbb M(\Phi), \quad \mathbb M(\Phi)\ni P\mapsto P^{-1}\in\mathbb M(\Phi)
\end{equation*}
are analytic. For the former map, this follows from formula \eqref{rtse5wu656} and Lemmas~\ref{tfy7e6e} and~\ref{gctessswe5} (i) (see also Remark~\ref{gcyd6ue6}). For the latter map, we similarly use formula \eqref{s4aq43q5}, which implies, by induction, that 
 $Q_{ik}$ is an analytic function of $P_{jl}$ with $i\le j<l\le k$.

Thus, $\mathbb M(\Phi)$ is a  Lie group and the Lie algebra of $\mathbb M(\Phi)$ can be identified with~$\mathbb V(\Phi)$. We now calculate the Lie bracket on $\mathbb V(\Phi)$. For any $s_1, s_2\in\mathbb C$ and $V^{(1)}, V^{(2)}\in\mathbb V(\Phi)$, we have 
$$(\mathbf 1+s_1V^{(1)})(\mathbf 1+s_2V^{(2)})-\mathbf 1= s_1V^{(1)}+s_2V^{(2)}+s_1s_2V^{(1)}V^{(2)}, $$
which implies \eqref{xra}.
\end{proof}

\begin{remark}\label{gcdtsw5qy}
It follows from the recurrence relation \eqref{s4aq43q5} that if $P=[P_{ik}]_{i, k\in\mathbb N_0}\in\mathbb M(\Phi)$ and $Q=[Q_{ik}]_{i, k\in\mathbb N_0}=P^{-1}$, then each $Q_{ik}$ is a linear combination of $P_{ik}$ and operators of the form $P_{ij_1}P_{j_1j_2}\dotsm P_{j_lk}$ with $i<j_1<j_2<\dots<j_l<k$.
\end{remark}

\begin{proposition}\label{gfstew5}
The exponential map for $\mathbb M(\Phi)$ exists and is given by
\begin{equation}\label{vctsaq45678}
\mathbb V(\Phi)\ni V\mapsto\exp(V)=\mathbf 1+\sum_{n=1}^\infty\frac1{n!}V^n\in\mathbb M(\Phi).\end{equation}
This is a bijective map and its inverse is given by
\begin{equation}\label{vctre5y67}
\mathbb M(\Phi)\ni P\mapsto\log(P)=\sum_{n=1}^\infty\frac{(-1)^{n+1}}n\, (P-\mathbf 1)^n\in \mathbb V(\Phi).
\end{equation}
  The maps in formulas \eqref{vctsaq45678} and \eqref{vctre5y67} are analytic; hence, the exponential map in~\eqref{vctsaq45678} provides an equivalent global parametrization of $\mathbb M(\Phi)$. In particular,  under the global parametrization $\exp$, the   Lie bracket on $\mathbb V(\Phi)$ is still given by \eqref{xra}.
\end{proposition}

\begin{proof} 
  Formula \eqref{cxzraw4t} implies that, for each $V\in\mathbb V(\Phi)$ and $n\in\mathbb N_0$, we have
\begin{equation}\label{vfyte6i}
V^{n+1}\big(\mathcal P^{(n)}(\Phi')\big)=\{0\}, \quad n\in\mathbb N_0.
\end{equation}
Hence, for each $t\in\mathbb R$ and $V\in\mathbb V(\Phi)$, the linear operator $\sum_{n=1}^\infty\frac{t^n}{n!}\, V^n$ is well defined and is an element of $\mathbb V(\Phi)$. In view of \eqref{vfyte6i}, we easily see that the map
$\R\ni t\mapsto\exp(tV):=\mathbf 1+ \sum_{n=1}^\infty\frac{t^n}{n!}\, V^n\in\mathbb M(\Phi)$
is smooth. Thus, $(\exp(tV))_{t\in\R}$ is a one-parameter subgroup of $\mathbb M(\Phi)$.
Since $\frac d{dt}\big|_{t=0}\exp(tV)=V$, the exponential map is defined for all $V\in\mathbb V(\Phi)$ and is equal to $\exp(V)$.

Next, we note that, for $V\in\mathbb V(\Phi)$, we have $\log(\mathbf 1+V)=\sum_{n=1}^\infty\frac{(-1)^{n+1}}{n}\, V^n$. Hence, analogously to the above, we have $\log(\mathbf 1+V)\in\mathbb V(\Phi)$. Similarly to the proof of Proposition~\ref{cxtstew5}, we now conclude that that exponential map is bijective and the logarithm map is its inverse.

  By Lemmas~\ref{tfy7e6e}, \ref{gctessswe5} (i), Remark~\ref{gcyd6ue6}, formula~\eqref{rtse5wu656}, and Remark~\ref{gcdtsw5qy}, the maps in formulas \eqref{vctsaq45678} and \eqref{vctre5y67} are analytic. 
 The fact that the exponential map $\exp$ provides an equivalent global parametrization of $\mathbb M(\Phi)$ follows from Lemma~\ref{analytic_reparametrization}. 
\end{proof}



\begin{proposition}\label{bufdsa}
$\mathbb M(\Phi)$ is a regular Lie group. \end{proposition}

\begin{proof} Let $[0, 1]\ni t\mapsto V(t)\in \mathbb V(\Phi)$ be a smooth curve. Similarly to the proof of Proposition~\ref{vytde6e6}, we first need to find a smooth curve $[0, 1]\ni t\mapsto P(t)\in \mathbb M(\Phi)$ that satisfies
\begin{equation}\label{csrwqy45q5}
P'(t)=P(t)V(t), \quad t\in[0, 1], 
\end{equation}
where $P(t)V(t)$ is understood as the product of operators from $\mathbb U(\Phi)$.
We state that
\begin{equation}\label{7t8r88}
P(t)=\mathbf 1+\sum_{l=1}^\infty\int_0^{t}dr_1\bigg(\int_0^{r_1}dr_2\bigg(\dotsm\bigg(\int_0^{r_{l-1}}dr_l\, V(r_{l})\bigg)\dotsm \bigg)V(r_2)\bigg) V(r_1).\end{equation}
Indeed, at least formally, \eqref{7t8r88} implies
\begin{equation}\label{vctrsy55}
P'(t)=V(t)+\sum_{l=2}^\infty \bigg(\int_0^{t}dr_2\bigg(\dotsm\bigg(\int_0^{r_{l-1}}dr_l\, V(r_{l})\bigg)\dotsm \bigg)V(r_2)\bigg) V(t), \end{equation}
which in turn implies \eqref{csrwqy45q5}. In fact, formula \eqref{vfyte6i} and Lemma~\ref{gctessswe5} (see also Remark~\ref{gcyd6ue6}) justify 
formulas \eqref{7t8r88}, \eqref{vctrsy55}.

That the map $C^\infty([0, 1];\mathbb V(\Phi))\ni V(\cdot)\mapsto P(1)\in\mathbb M(\Phi) $ is analytic can be shown similarly to the proof of Propositions~\ref{vytde6e6} and \ref{vcte6u}.
\end{proof}

\section{The Sheffer Lie group}\label{bvfyr64u}

\subsection{The Sheffer group}

Let $P=[P_{ik}]_{i, k\in\mathbb N_0}\in\mathbb M(\Phi)$ and let $(P^{(k)}(\omega, f^{(k)}))_{f^{(k)}\in\Phi^{\hat\odot k}, \, k\in\mathbb N_0}$ be the corresponding monic polynomial sequence. By Remark~\ref{cfta4a4zz}, for a fixed $\omega\in\Phi'$ and each $k\in\mathbb N_0$, 
we have $P^{(k)}(\omega, \cdot)=\sum_{i=0}^k\langle\omega^{\otimes i}, P_{ik}\, \cdot\rangle\in(\Phi^{\hat\odot k})'$. For each fixed $\omega\in\Phi'$, we define
\begin{equation}\label{vcrses}
G(\omega, \xi):=\sum_{k=0}^\infty \frac1{k!}\, P^{(k)}(\omega, \xi^{\otimes k})\in \mathcal F_0(\Phi).
\end{equation}
 The $G(\omega, \xi)$ is called the {\it (exponential) generating function of the monic polynomial sequence corresponding to the operator~$P$}. In view of Remark~\ref{ghcydydsszza}, the generating function $G(\omega, \xi)$ uniquely characterizes $P$. Hence, we may also think of $G(\omega, \xi)$ as the {\it generating function of the operator~$P$}.

An operator $P\in\mathbb M(\Phi)$ is called a {\it Sheffer operator} and the corresponding monic polynomial sequence is called a {\it Sheffer sequence} if their generating function is of the form
\begin{equation}\label{fgxtrsr}
G(\omega, \xi)=\exp\big[\langle \omega, B(\xi)\rangle\big]A(\xi), \end{equation}
where $(A(\xi), B(\xi))\in\mathcal S(\Phi)$. We denote by $\mathbb S(\Phi)$ the set of all Sheffer operators.
Note that, for $B(\xi)=\sum_{k=1}^\infty B_k\xi^{\otimes k}\in \mathcal F_1(\Phi)$ and $\omega\in\Phi'$, we have $\langle\omega, B(\xi)\rangle=\sum_{k=1}^\infty\langle\omega, B_k\xi^{\otimes k}\rangle\in\mathcal W_0(\Phi)$. Hence, for each $\omega\in\Phi'$, we indeed have $\exp\big[\langle \omega, B(\xi)\rangle\big]\in\mathcal F_0(\Phi)$, and so $G(\omega, \xi)\in\mathcal F_0(\Phi)$.

An operator $P\in\mathbb M(\Phi)$ is called an {\it Appell operator} and the corresponding monic polynomial sequence is called an {\it Appell sequence} if their generating function is of the form
$ G(\omega, \xi)=\exp\big[\langle \omega, \xi\rangle\big]A(\xi)$, where $A(\xi)\in\mathcal F_0(\Phi)$. Thus, $P$ is an Appell operator if it is a Sheffer operator for which $B(\xi)=\xi$. We denote by $\mathbb A(\Phi)$ the set of all Appell operators. 

An operator $P\in\mathbb M(\Phi)$ is called an {\it umbral operator} and the corresponding monic polynomial sequence is called a {\it sequence of binomial type}, or just a {\it binomial sequence} if their generating function is of the form
$ G(\omega, \xi)=\exp\big[\langle \omega, B(\xi)\rangle\big]$, where $B(\xi)\in\mathcal F_1(\Phi)$. Thus, $P$ is an umbral operator if it is a Sheffer operator for which $A(\xi)=1$. We denote by $\mathbb B(\Phi)$ the set of all umbral operators. 

\begin{remark}
Similarly to \cite[Theorem~4.1]{FKLO}, one can show that, for $P\in\mathbb M(\Phi)$, the corresponding monic polynomial sequence is of binomial type if and only if, for each $n\in\mathbb N$, $\omega, \omega'\in\Phi'$ and $\xi\in\Phi$, 
$$P^{(n)}(\omega+\omega', \xi^{\otimes n})=\sum_{k=0}^n\binom nk P^{(k)}(\omega, \xi^{\otimes k}) P^{(n-k)}(\omega', \xi^{\otimes (n-k)}).$$
\end{remark}

\begin{proposition}\label{bhude7i8}
Let $A(\xi)=\sum_{k=0}^\infty A_k\xi^{\otimes k}\in\mathcal F_0(\Phi)$ and $B(\xi)=\sum_{k=1}^\infty B_k\xi^{\otimes k}\in\mathcal F_1(\Phi)$, and let $P=[P_{ik}]_{i, k\in\mathbb N_0}$ have the generating function $G(\omega, \xi)$ given by~\eqref{fgxtrsr}. 

(i) We have 
\begin{equation}\label{vgd6u4i}
P_{0n}=n!\, A_n, \quad n\in\mathbb N, \end{equation}
 and for $1\le k<n$, 
\begin{equation}\label{vcftstezz}
P_{kn}=\frac{n!}{k!}\bigg(\sum_{{l_1, \dots, l_k \in\mathbb N}\atop l_1+\dots+l_k=n}(B_{l_1}\odot \cdots \odot B_{l_k})+\sum_{m=1}^{n-k}\sum_{{l_1, \dots, l_k \in\mathbb N}\atop l_1+\dots+l_k=n-m}(A_m\odot B_{l_1}\odot \cdots \odot B_{l_k})\bigg). 
\end{equation}

(ii) We have 
\begin{equation}
B_n=\frac1{n!}\, P_{1n}-\sum_{m=1}^{n-1}\frac1{(n-m)!}\, P_{0\, n-m}\odot B_m\quad\text{for } n\ge2.\label{ctre6uu}
\end{equation}
In particular, a Sheffer operator $P\in\mathbb S(\Phi)$ uniquely identifies $A(\xi)\in\mathcal F_0(\Phi)$ and $B(\xi)\in\mathcal F_1(\Phi)$ through the generating function of $P$.
\end{proposition}

\begin{proof}
Part (i) follows by straightforward calculations from the definition of the generating function of a Sheffer operator. Formula~\eqref{ctre6uu} follows from part (i). Formula \eqref{ctre6uu} allows us to recurrently calculate $B_{n}$ ($n\ge2$) through $P_{01}, P_{02}, \dots, P_{0\, n-1}$ and $P_{12}, P_{13}, \dots, P_{1n}$. 
\end{proof}

\begin{corollary}\label{vgy7kus}
(i) There exist continuous (in fact analytic) maps 
$$g_{kn}: \prod_{j=1}^{n-k}(\Phi^{\hat\odot j})'\times\prod_{l=2}^{n-k+1}\mathcal L(\Phi^{\hat\odot l}, \Phi)\to\mathcal L(\Phi^{\hat\odot n}, \Phi^{\hat\odot k}), \quad 2\le k<n, $$
for which the following statement holds: $P=[P_{kn}]_{k, n\in\mathbb N_0}\in\mathbb M(\Phi)$ is a Sheffer operator 
 if and only if 
\begin{equation}\label{dasrrh}
P_{kn}=g_{kn}(P_{01}, \dots, P_{0\, n-k}, P_{12}, \dots, P_{1\, n-k+1}), \quad 2\le k<n.\end{equation}

(ii) $P=[P_{kn}]_{k, n\in\mathbb N_0}\in\mathbb M(\Phi)$ is an Appell operator if and only if $P_{1n}=n!\, P_{0\, n-1}\odot\mathbf 1_1$ for $n\ge2$ and condition \eqref{dasrrh} is satisfied.

(iii) $P=[P_{kn}]_{k, n\in\mathbb N_0}\in\mathbb M(\Phi)$ is an umbral operator if and only if $P_{0n}= 0$ for $n\ge1$ and condition \eqref{dasrrh} is satisfied.
\end{corollary}

\begin{proof} (i) Fix arbitrary $P_{0j}\in (\Phi^{\hat\odot j})'$ ($j\ge1$) and $P_{1l}\in \mathcal L(\Phi^{\hat\odot l}, \Phi)$ ($l\ge2$). We state that there exists a unique $P\in\mathbb S(\Phi)$ for which $P_{0j}$ ($j\ge1$) and $P_{1l}$ ($l\ge2$) are the entries of the block-matrix $P=[P_{kn}]_{k, n\in\mathbb N_0}$. Indeed, we define $A_n$ ($n\ge1$) using formula \eqref{vgd6u4i} and $B_n$ ($n\ge2$) recurrently through formula \eqref{ctre6uu}. We then define $P_{kn}$ for $2\le k<n$ by formula \eqref{vcftstezz}, which yields the required $P=[P_{kn}]_{k, n\in\mathbb N_0}\in \mathbb S(\Phi)$. 

A more careful analysis of the recurrence formula \eqref{ctre6uu} shows that $B_n$ is an analytic function of $P_{01}, P_{02}, \dots, P_{0\, n-1}, P_{12}, P_{13}, \dots, P_{1n}$. Trivially, by formula~\eqref{vgd6u4i}, we see that $A_n$ ($n\ge1$) is an analytic function of $P_{0n}$. 
Substituting these functions for $A_n$ and $B_n$ into~\eqref{vcftstezz}, we conclude that $P_{kn}$ ($2\le k<n$) is an analytic function of $P_{01}, \dots, P_{0\, n-k}, P_{12}, \dots, P_{1\, n-k+1}$.

(ii) and (iii). We only need to describe the respective conditions $B_n=0$ for $n\ge2$ and $A_n=0$ for $n\ge1$ in terms of $P_{0j}$ ($j\ge1$) and $P_{1j}$ ($j\ge2$). But this easily follows from formulas \eqref{vgd6u4i} and \eqref{ctre6uu}.
\end{proof}

\begin{proposition} \label{vhydye} Let $P=[P_{ik}]_{i, k\in\mathbb N_0}\in\mathbb M(\Phi)$. Then $P$ has generating function~\eqref{fgxtrsr} if and only if, for each $i\in\mathbb N_0$, we have
\begin{equation}\label{cydtyds}
i!\sum_{k=i}^\infty\frac{1}{k!} P_{ik}\xi^{\otimes k}= B(\xi)^{\odot i}A(\xi), \end{equation}
formula \eqref{cydtyds} being understood as an equality of formal tensor power series from $\mathcal F(\Phi;\Phi^{\hat\odot i})$.
\end{proposition}

\begin{proof} Let $P=[P_{ik}]_{i, k\in\mathbb N_0}$ have generating function~\eqref{fgxtrsr}. 
For $\omega\in\Phi'$ and $s\in\mathbb C$, we have, in view of Remarks~\ref{bydtrsw5w} and~\ref{cfta4a4zz}, 
\begin{align}
G(s\omega, \xi)&=\sum_{k=0}^\infty \frac1{k!}\sum_{i=0}^k s^i\langle\omega^{\otimes i}, P_{ik}\xi^{\otimes k}\rangle
=\sum_{i=0}^\infty s^i\sum_{k=i}^\infty\frac1{k!}\langle\omega^{\otimes i}, P_{ik}\xi^{\otimes k}\rangle\notag\\
&=\sum_{i=0}^\infty s^i\bigg\langle \omega^{\otimes i}, \sum_{k=i}^\infty\frac1{k!} P_{ik}\xi^{\otimes k}\bigg\rangle.\label{dxtrsrtestres}
\end{align}
On the other hand, 
\begin{equation}\label{ctsrw5}
\exp\big[\langle s\omega, B(\xi)\rangle\big]A(\xi)=\sum_{i=0}^\infty s^i\bigg\langle\omega^{\otimes i}, \frac1{i!}\, B(\xi)^{\odot i}A(\xi)\bigg\rangle.
\end{equation}
Formulas \eqref{dxtrsrtestres} and \eqref{ctsrw5} imply \eqref{cydtyds}. The converse statement is then obvious.
\end{proof}

We define a map 
$\mathcal I:\mathcal S(\Phi)\to \mathbb S(\Phi)$ by $\mathcal I(A(\xi), B(\xi))=P$, where $P\in\mathbb S(\Phi)$ has generating function \eqref{fgxtrsr}.

\begin{proposition}\label{vgfdstt} 

(i) The $\mathbb S(\Phi)$ is a subgroup of $\mathbb M(\Phi)$. The $\mathbb A(\Phi)$ is an abelian normal subgroup of $\mathbb S(\Phi)$, $\mathbb B(\Phi)$ is a subgroup of $\mathbb S(\Phi)$, and $\mathbb S(\Phi)$ is the semidirect product of $\mathbb A(\Phi)$ and $\mathbb B(\Phi)$, i.e., $\mathbb S(\Phi)=\mathbb A(\Phi)\rtimes\mathbb B(\Phi)$. 

(ii) The following maps are group isomorphisms:
\begin{equation}\label{ctrsw56u}
\mathcal I:\mathcal S(\Phi)\to \mathbb S(\Phi), \quad \mathcal I\restriction_{\mathcal F_0(\Phi)}:\mathcal F_0(\Phi)\to \mathbb A(\Phi), \quad
 \mathcal I\restriction_{\mathcal F_1(\Phi)}:\mathcal F_1(\Phi)\to \mathbb B(\Phi).\end{equation}
\end{proposition}

 We will call $\mathbb S(\Phi)$ the \emph{Sheffer group (on $\Phi'$)}, $\mathbb A(\Phi)$ the \emph{Appell group}, and $\mathbb B(\Phi)$ the \emph{umbral group}.

 \begin{proof}[Proof of Proposition~\ref{vgfdstt} ] 
By Proposition~\ref{bhude7i8} (ii), the three maps in formula \eqref{ctrsw56u} are bijective. Since $\mathcal S(\Phi)$ is a group and $\mathcal S(\Phi):=\mathcal F_0(\Phi) \rtimes\mathcal F_1(\Phi)$, we only need to prove the following

{\it Statement}. For $m=1, 2$, let $(A^{(m)}(\xi), B^{(m)}(\xi))\in\mathcal S(\Phi)$ and let 
 $$P^{(m)}=[P^{(m)}_{ik}]_{i, k\in\mathbb N_0}:=\mathcal I(A^{(m)}(\xi), B^{(m)}(\xi))\in\mathbb S(\Phi).$$ 
 Then
\begin{equation}\label{fxste66}
\mathcal I\big((A^{(1)}(\xi), B^{(1)}(\xi))\ast (A^{(2)}(\xi), B^{(2)}(\xi))\big)=P^{(1)}P^{(2)}.
\end{equation}

Indeed, let $P=[P_{ik}]_{i, k\in\mathbb N_0}:=P^{(1)}P^{(2)}$. We have, by Proposition~\ref{vhydye}, 
\begin{align}
 i!\sum_{k=i}^\infty\frac{1}{k!} P_{ik}\xi^{\otimes k}&= i!\sum_{k=i}^\infty\frac{1}{k!} \sum_{l=i}^k P_{il}^{(1)}P_{lk}^{(2)}\xi^{\otimes k}
 =i!\sum_{l=i}^\infty\frac1{l!}\, P_{il}^{(1)}l!\sum_{k=l}^\infty \frac1{k!}\, P_{lk}^{(2)}\xi^{\otimes k}\notag\\
 &=i!\sum_{l=i}^\infty\frac1{l!}\, P_{il}^{(1)}\big(B^{(2)}(\xi)^{\odot l}A^{(2)}(\xi)\big)=\bigg(i!\sum_{l=i}^\infty\frac1{l!}\, P_{il}^{(1)}\big(B^{(2)}(\xi)^{\odot l}\big)\bigg)A^{(2)}(\xi)\notag\\
 &=B^{(1)}(B^{(2)}(\xi))^{\odot i}A^{(1)}(B^{(2)}(\xi))A^{(2)}(\xi).\label{vgcxaw45678}
\end{align}
Formulas \eqref{vcyrtsw5u}, \eqref{tsay4w}, \eqref{vgcxaw45678} and Proposition~\ref{vhydye} imply \eqref{fxste66}.
 \end{proof}
 
 \begin{lemma}\label{vct6eux}
 The $\mathbb S(\Phi)$, $\mathbb A(\Phi)$, and $\mathbb B(\Phi)$ are closed subsets of $\mathbb M(\Phi)$.
 \end{lemma}
 
 \begin{proof}
 Let $(P^{\alpha})$ be a net in $\mathbb M(\Phi)$ such that each $P^\alpha$ belongs to $\mathbb S(\Phi)$ and $P^\alpha\to P$ in $\mathbb M(\Phi)$. Let $[P_{kn}^\alpha]_{k, n\in\mathbb N_0}$ and $[P_{kn}]_{k, n\in\mathbb N_0}$ be the block-matrices of $P^\alpha$ and $P$, respectively. By a property of the product topology, we then have, for each $0\le k<n$, $P_{kn}^\alpha\to P_{kn}$ in $\mathcal L(\Phi^{\hat\odot n}, \Phi^{\hat\odot k})$. By Corollary~\ref{vgy7kus} (i), 
\begin{equation}\label{vgtu6e7r890y8}
P_{kn}^\alpha=g_{kn}(P_{01}^\alpha, \dots, P_{0\, n-k}^\alpha, P_{12}^\alpha, \dots, P_{1\, n-k+1}^\alpha), \quad 2\le k<n.\end{equation}
By the continuity of the function $g_{kn}$, the right-hand side of formula \eqref{vgtu6e7r890y8} converges to $g_{kn}(P_{01}, \dots, P_{0\, n-k}, P_{12}, \dots, P_{1\, n-k+1})$, whereas the left-hand side converges to $P_{kn}$. Thus, $P\in\mathbb S(\Phi)$ by Corollary~\ref{vgy7kus} (i), and so $\mathbb S(\Phi)$ is closed in $\mathbb M(\Phi)$. 

One can similarly prove that $\mathbb A(\Phi)$ and $\mathbb B(\Phi)$ are closed subsets of $\mathbb M(\Phi)$ by using Corollary~\ref{vgy7kus} (ii) and (iii). 
 \end{proof}

\subsection{The $\mathbb S(\Phi)$ as a Lie group}

For the definition of an embedded Lie subgroup, see Subsection~\ref{app:EmbLieSubGroup} in the Appendix.

 
\begin{theorem}\label{cftstw6u}

(i) Let $\log:\mathbb M(\Phi)\to\mathbb V(\Phi)$ be the logarithm map given by \eqref{vctre5y67}. Define
\begin{equation*}
\mathfrak s(\Phi):=\log\big(\mathbb S(\Phi)\big), \quad \mathfrak a(\Phi):=\log\big(\mathbb A(\Phi)\big), \quad \mathfrak b(\Phi):=\log\big(\mathbb B(\Phi)\big).\end{equation*}
Then $\mathfrak s(\Phi)$, $\mathfrak a(\Phi)$ and $\mathfrak b(\Phi)$ are closed vector subspaces of $\mathbb V(\Phi)$. Hence, 
 $\mathbb S(\Phi)$, $\mathbb A(\Phi)$, and $\mathbb B(\Phi)$ are embedded Lie subgroups of $\mathbb M(\Phi)$, and $\mathbb A(\Phi)$ and $\mathbb B(\Phi)$ are also embedded Lie subgroups of $\mathbb S(\Phi)$. 
  In particular, $\mathbb S(\Phi)$ is a Lie group modeled on the complete l.c.s.\ $\mathfrak s(\Phi)$ through the global parametrization $\exp:\mathfrak s(\Phi)\to\mathbb S(\Phi)$, and the Lie algebra of $\mathbb S(\Phi)$ can be identified with $\mathfrak s(\Phi)$.

(ii)
The group isomorphism $\mathcal I:\mathcal S(\Phi)\to \mathbb S(\Phi)$ and its inverse 
$\mathcal I^{-1}:\mathbb S(\Phi)\to \mathcal S(\Phi)$ are analytic maps.

(iii) Define a map $\mathcal R:\mathcal W(\Phi)\to\mathfrak s(\Phi)$ by $\mathcal R:=\log(\mathcal I\operatorname{EXP})$. Then $\mathcal R$ is bijective. Furthermore, let $\alpha(\xi)=\sum_{k=1}^\infty\alpha_k\xi^{\otimes k}\in\mathcal W_0(\Phi)$, $\beta(\xi)=\sum_{k=2}^\infty\beta_k\xi^{\otimes k}\in\mathcal W_1(\Phi)$, and let 
$V=[V_{ik}]_{i, k\in\mathbb N_0}:=\mathcal R (\alpha(\xi), \beta(\xi))$.
 Then
 \begin{equation}\label{vcyrte6ue}
V_{ik}=(k)_{k-i+1}\, \beta_{k-i+1}\odot\mathbf 1_{i-1}+(k)_{k-i}\, \alpha_{k-i}\odot\mathbf1_i, \quad 0\le i<k, 
\end{equation}
where $(k)_l:=k(k-1)\dotsm(k-l+1)$ is the Pochhammer symbol.

(iv) The $\mathbb S(\Phi)$, $\mathbb A(\Phi)$, and $\mathbb B(\Phi)$ are regular Lie groups. 

\end{theorem}

\begin{remark}In view of statement (ii) of Theorem~\ref{cftstw6u}, one can alternatively think of $\mathbb S(\Phi)$ as a Lie group modeled on the l.c.s.\ $\mathcal W(\Phi)$ through the global parametrization $\mathcal I\operatorname{EXP}:\mathcal W(\Phi)\to\mathbb S(\Phi)$, or through the global parametrization $\mathcal I\mathcal T:\mathcal W(\Phi)\to\mathbb S(\Phi)$ (where $\mathcal T$ is given by \eqref{vcstsreaea}). 
\end{remark}

\begin{proof}[Proof of Theorem~\ref{cftstw6u}] (i) Similarly to the proofs in Section~\ref{vgydi67r}, we conclude from formulas~\eqref{vgd6u4i}, \eqref{vcftstezz} that $\mathcal I$, considered as a map from $\mathcal S(\Phi)$ to $\mathbb M(\Phi)$, is analytic.

Fix $P\in\mathbb S(\Phi)$. Define $(\alpha(\xi), \beta(\xi)):=\operatorname{LOG}(\mathcal I^{-1}P)\in\mathcal W(\Phi)$ and
\begin{equation}\label{caaoiou9iuy}
(A(t, \xi), B(t, \xi)):=\operatorname{EXP}\big(t(\alpha(\xi), \beta(\xi))\big), \quad t\in\R.
\end{equation}
Then $\big((A(t, \xi), B(t, \xi))\big)_{t\in\R}$ is a one-parameter subgroup of $\mathcal S(\Phi)$ satisfying 
$$(A(1, \xi), B(1, \xi))=\mathcal I^{-1}P.$$
For $t\in\mathbb R$, define $P(t)=[P_{ik}(t)]_{i, k\in\mathbb N_0}:=\mathcal I(A(t, \xi), B(t, \xi))$.
Then, by Proposition~\ref{vgfdstt}~(ii) and the analyticity of the map $\mathcal I$, $(P(t))_{t\in\R}$ is a one-parameter subgroup of $\mathbb M(\Phi)$ with $P(1)=P$. Hence, $\log(P)=\frac d{dt}\big|_{t=0}P(t)$, 
where the limit is taken in $\mathbb U(\Phi)$.

By the definition of $P(t)$ and Proposition~\ref{vhydye}, we have
\begin{equation}\label{cfgdtsds}
i!\sum_{k=i}^\infty\frac{1}{k!} P_{ik}(t)\xi^{\otimes k}= B(t, \xi)^{\odot i}A(t, \xi), \quad i\in\mathbb N_0.\end{equation}
Using formulas \eqref{cfrtsw5yw}--\eqref{vcftse5a5ywu}, one can easily show that
$$\frac d{dt}\big(B(t, \xi)^{\odot i}A(t, \xi)\big)=iB'(t, \xi)\odot B(t, \xi)^{\odot(i-1)}A(t, \xi)+B(t, \xi)^{\odot i}A'(t, \xi). $$
Hence, by \eqref{caaoiou9iuy}, 
\begin{align}
\frac d{dt}\Big|_{t=0}\big(B(t, \xi)^{\odot i}A(t, \xi)\big)&=i\beta(\xi)\odot\xi^{\otimes(i-1)}+\xi^{\otimes i}\alpha(\xi)\notag\\
&=\sum_{k=i+1}^\infty\big(i\beta_{k-i+1}\odot\mathbf 1_{i-1}+\alpha_{k-i}\odot\mathbf1_i\big)\xi^{\otimes k}, \quad i\in\mathbb N_0, \label{cesarwe44waq4}
\end{align}
where $\alpha(\xi)=\sum_{k=1}^\infty\alpha_k\xi^{\otimes k}$ and $\beta(\xi)=\sum_{k=2}^\infty\beta_k\xi^{\otimes k}$. 

Denote $V=[V_{ik}]_{i, k\in\mathbb N_0}:=\log(P)$. Since $V_{ik}=\frac d{dt}\big|_{t=0}P_{ik}(t)$, formulas \eqref{cfgdtsds} and \eqref{cesarwe44waq4} imply that each $V_{ik}$ is given by formula \eqref{vcyrte6ue}.

By Propositions~\ref{cydsawa} and~\ref{vgfdstt}~(ii), the map $\operatorname{LOG}(\mathcal I^{-1}):\mathbb S(\Phi)\to \mathcal W(\Phi)$. The $\mathfrak s(\Phi)$ consists of all $V=[V_{ik}]_{i, k\in\mathbb N_0}\in\mathbb V(\Phi)$ for which $V_{ik}$ ($0\le i<k$) is of the form \eqref{vcyrte6ue} for some $\alpha_k\in(\Phi^{\hat\odot k})'$ ($k\ge1$) and $\beta_k\in\mathcal L(\Phi^{\hat\odot k}, \Phi)$ ($k\ge2$). In particular, $\mathfrak s(\Phi)$ is a vector subspace of $\mathbb V(\Phi)$. Lemma~\ref{vct6eux} then implies that $\mathfrak s(\Phi)$ is closed in $\mathbb V(\Phi)$. 

We similarly prove that both $\mathfrak a(\Phi)$ and $\mathfrak b(\Phi)$ are closed vector subspaces of $\mathbb V(\Phi)$. This proves statement~(i) of the theorem.

(ii) Recall that $\mathcal I$, considered as a map from $\mathcal S(\Phi)$ to $\mathbb M(\Phi)$, is analytic. Since $\mathcal I$ takes values in $\mathbb S(\Phi)$ and $\mathbb S(\Phi)$ is an embedded Lie subgroup of $\mathbb M(\Phi)$,  the map $\mathcal I:\mathcal S(\Phi)\to \mathbb S(\Phi)$ is analytic. To prove that its inverse is also analytic, we proceed as follows. 

Recall that we have endowed $\mathbb S(\Phi)$ with the topology induced by $\mathbb M(\Phi)$. Define the map $\operatorname{Proj}:\mathbb S(\Phi)\to\mathcal W(\Phi)$ which sends $P=[P_{ik}]_{i, k\in\mathbb N_0}\in\mathbb S(\Phi)$ to 
$$\operatorname{Proj} P:=\big((P_{0k})_{k\ge1}, (P_{1k})_{k\ge2}\big)\in\mathcal W(\Phi).$$
 By Corollary~\ref{vgy7kus}~(i), the map $\operatorname{Proj}$ is bijective. 
Obviously, the map $\operatorname{Proj}$ is analytic.

Define a map $\mathcal J:\mathcal W(\Phi)\to\mathcal S(\Phi)$, $\mathcal J:=\mathcal I^{-1}\operatorname{Proj}^{-1}$. Thus, $\mathcal J(\operatorname{Proj}P)=\mathcal I^{-1}P$ for $P\in\mathbb S(\Phi)$. Then, to prove that the map $\mathcal I^{-1}:\mathbb S(\Phi)\to\mathcal S(\Phi)$ is analytic, it suffices to prove that the map $\mathcal J:\mathcal W(\Phi)\to\mathcal S(\Phi)$ is analytic. But the latter statement easily follows from~\eqref{vgd6u4i} and the recurrence formula~\eqref{ctre6uu}.


(iii) The fact that the map $\mathcal R$ is bijective follows immediately from Propositions~\ref{cydsawa}, \ref{gfstew5} and \ref{vgfdstt} (ii). Formula~\eqref{vcyrte6ue} follows from the proof of part~(i) of the theorem.

 Finally, part (iv) of the theorem follows from Proposition~\ref{vgfdstt}~(ii), part~(ii) of the present theorem, and the corresponding results for the Lie groups $\mathcal S(\Phi)$, $\mathcal F_0(\Phi)$, and $\mathcal F_1(\Phi)$. \end{proof}

The following corollary is now obvious.

\begin{corollary}
The maps
$\mathfrak s(\Phi)\ni V\mapsto\exp(V)\in\mathbb S(\Phi)$, $\mathfrak a(\Phi)\ni V\mapsto\exp(V)\in\mathbb A(\Phi)$ and $\mathfrak b(\Phi)\ni V\mapsto\exp(V)\in\mathbb B(\Phi) $
are bijective, analytic and their inverse maps are analytic. Furthermore, these are exponential maps for the respective Lie groups. In particular, for $V\in \mathfrak s(\Phi)$, $\big(\exp(tV)\big)_{t\in\R}$ is a one-parameter subgroup of $\mathbb S(\Phi)$ and $\frac{d}{dt}\big|_{t=0}\exp(tV)=V$. 
\end{corollary}

\subsection{The Lie algebra $\mathfrak s(\Phi)$}\label{xra5aq354q}

Our next aim is to better understand the Lie algebra $\mathfrak s(\Phi)$ of the Lie group $\mathbb S(\Phi)$.

Let $\sigma\in\Phi'$ and $f^{(n)}\in\Phi^{\hat\odot n}$. As easily seen, the derivative in direction $\sigma$ of the monomial $\langle\omega^{\otimes n}, f^{(n)}\rangle$ is given by
\begin{equation}\label{zera4wyq}
(D_\sigma\langle\cdot^{\otimes n}, f^{(n)}\rangle)(\omega)=n\langle\omega^{\otimes(n-1)}\odot\sigma, f^{(n)}\rangle= n\langle\omega^{\otimes(n-1)}, (\sigma\odot\mathbf 1_{n-1})f^{(n)}\rangle.
\end{equation}
 Note that $\sigma\odot\mathbf 1_{n-1}\in\mathcal L(\Phi^{\hat\odot n}, \Phi^{\hat\odot(n-1)})$. Thus, $D_\sigma\in\mathcal L(\mathcal P(\Phi'))$.
 
 For each $\xi\in\Phi$, we define $M(\xi)\in\mathcal L(\mathcal P(\Phi'))$ as the operator of multiplication by the monomial $\langle\omega, \xi\rangle$:
\begin{equation*}
(M(\xi)p)(\omega):= \langle\omega, \xi\rangle p(\omega), \quad p\in\mathcal P(\Phi').
\end{equation*}
 
 As easily seen, the operators $D_\sigma$ ($\sigma\in\Phi'$) and $M(\xi)$ ($\xi\in\Phi$) form a representation of a Weyl algebra:
 \begin{align}
 &[D_{\sigma_1}, D_{\sigma_2}]=[M(\xi_1), M(\xi_2)]= 0, \quad \sigma_1, \sigma_2\in\Phi', \ \xi_1, \xi_2\in\Phi, \notag\\
 &[D_\sigma, M(\xi)]=\langle\sigma, \xi\rangle\mathbf1, \quad \sigma\in\Phi', \ \xi\in\Phi.\label{vtdrsesw3}
 \end{align}
 
 We define the gradient $\nabla$ on $\mathcal P(\Phi')$ so that, for each polynomial $p\in \mathcal P(\Phi')$ and a fixed $\omega\in\Phi'$, $\nabla p(\omega)\in\Phi$ and
$(D_\sigma p)(\omega)=\langle\sigma, \nabla p(\omega)\rangle$ for all $\sigma\in\Phi'$.
 Formula \eqref{zera4wyq} implies
\begin{equation}\label{fqdtfr}
(D_\sigma\langle\cdot^{\otimes n}, f^{(n)}\rangle)(\omega)=\langle\sigma, n(\omega^{\otimes(n-1)} \odot\mathbf 1_1)f^{(n)}\rangle\end{equation}
 and so
 \begin{equation}\label{cxerw5yu}
( \nabla\langle\cdot^{\otimes n}, f^{(n)}\rangle)(\omega)=n(\omega^{\otimes(n-1)}\odot\mathbf 1_1)f^{(n)}.\end{equation}
 We similarly define the $k$th power of the gradient $\nabla^k$ ($k\ge2$): 
 for each polynomial $p\in \mathcal P(\Phi')$ and a fixed $\omega\in\Phi'$, $\nabla^k p(\omega)\in\Phi^{\hat\odot k}$ and
 $$D_{\sigma_1}\dotsm D_{\sigma_k}\, p(\omega)=\langle\sigma_1\odot\dots\odot\sigma_k, \nabla^kp(\omega)\rangle, \quad\forall\sigma_1, \dots, \sigma_k\in\Phi'. $$
Similarly to \eqref{fqdtfr}, \eqref{cxerw5yu}, we obtain:
\begin{equation}\label{dszeaewaeu}
( \nabla^k\langle\cdot^{\otimes n}, f^{(n)}\rangle)(\omega)=(n)_k(\omega^{\otimes(n-k)}\odot\mathbf 1_k)f^{(n)}.\end{equation} 
 
 For $k\in\mathbb N$ and $\alpha_k\in(\Phi^{\hat\odot k})'$, formula \eqref{dszeaewaeu} implies, for a fixed $\omega\in\Phi'$, 
 \begin{align}
\alpha_k\big((\nabla^k\langle\cdot^{\otimes n}, f^{(n)}\rangle)(\omega)\big)&=(n)_k(\omega^{\otimes(n-k)}\odot \alpha_k)f^{(n)}\notag\\
 &=(n)_k\langle\omega^{\otimes(n-k)}, (\mathbf 1_{n-k}\odot \alpha_k)f^{(n)}\rangle.\label{vcrtew654uw}
 \end{align}
 Since $\mathbf 1_{n-k}\odot \alpha_k\in\mathcal L(\Phi^{\hat\odot n}, \Phi^{\hat\odot (n-k)})$, formula \eqref{vcrtew654uw} allows us to interpret $\alpha_k\nabla^k$ as an element of $\mathbb V(\Phi)$ by setting $(\alpha_k\nabla^kp)(\omega):=\alpha_k\big((\nabla^k p)(\omega)\big)$ for $p\in\mathcal P(\Phi')$. Note that
 $$\alpha_k\nabla^k\big( \mathcal P^{(n)}(\Phi')\big)\subset \mathcal P^{(n-k)}(\Phi'), $$
 where $\mathcal P^{(m)}(\Phi'):=\{0\}$ for $m<0$.
 Hence, for $\alpha(\xi)=\sum_{k=1}^\infty \alpha_k\xi^{\otimes k}\in\mathcal W_0(\Phi)$, we have
 $\sum_{k=1}^\infty \alpha_k\nabla^k\in\mathbb V(\Phi)$. Below we will use the notation $\alpha(\nabla):=\sum_{k=1}^\infty \alpha_k\nabla^k$.


Similarly, for $k\ge2$ and $\beta_k\in\mathcal L(\Phi^{\hat\odot k}, \Phi)$, formula \eqref{dszeaewaeu} implies, for a fixed $\omega\in\Phi'$, 
 $$
 \beta_k\big((\nabla^k\langle\cdot^{\otimes n}, f^{(n)}\rangle)(\omega)\big)=(n)_k(\omega^{\otimes(n-k)}\odot \beta_k)f^{(n)}\in\Phi, $$
 and so
 \begin{align}
 \big\langle\omega, \beta_k\big((\nabla^k\langle\cdot^{\otimes n}, f^{(n)}\rangle)(\omega)\big)\big\rangle&=(n)_k\langle
 \omega, (\omega^{\otimes(n-k)}\odot \beta_k)f^{(n)} \rangle\notag\\
 &=(n)_k\langle\omega^{\otimes(n-k+1)}, (\mathbf 1_{n-k}\odot \beta_k)f^{(n)}\rangle.\label{cfsterewq}
 \end{align}
Since $\mathbf 1_{n-k}\odot \beta_k\in\mathcal L(\Phi^{\hat\odot n}, \Phi^{\hat\odot (n-k+1)})$, formula \eqref{cfsterewq} allows us to define
an operator $M(\beta_k\nabla^k)\in \mathbb V (\Phi)$ by setting
\begin{equation}\label{vcydtyrdst6}
\big(M(\beta_k\nabla^k) p\big)(\omega):=\big\langle \omega, \beta_k\big((\nabla^kp)(\omega)\big)\big\rangle, \quad p\in\mathcal P(\Phi').\end{equation}
In particular, for $\xi\in\Phi$, formulas \eqref{cfsterewq}, \eqref{vcydtyrdst6} yield
\begin{align}
\big(M(\beta_k\nabla^k) \langle\cdot^{\otimes n}, \xi^{\otimes n}\rangle\big)(\omega)&=(n)_k
\langle \omega^{\otimes(n-k+1)}, \xi^{\otimes(n-k)}\odot(\beta_k\xi^{\otimes k})\rangle\notag\\
&=(n)_k\langle\omega, \beta_k\xi^{\otimes k}\rangle\langle\omega^{\otimes(n-k)}, \xi^{\otimes(n-k)}\rangle \label{jiy8i}\\
&=(n)_k \big(M(\beta_k\xi^{\otimes k}) \langle\cdot^{\otimes(n-k)}, \xi^{\otimes(n-k)}\rangle\big)(\omega).\notag
\end{align}
Note that
$$M(\beta_k\nabla^k) \big(\mathcal P^{(n)}(\Phi')\big)\subset \mathcal P^{(n-k+1)}(\Phi').$$
Hence, for $\beta(\xi)=\sum_{k=2}^\infty \beta_k\xi^{\otimes k}\in\mathcal W_1(\Phi)$, we have $\sum_{k=2}^\infty M(\beta_k\nabla^k)\in\mathbb V(\Phi)$. Below we will use the notation
$M(\beta(\nabla)){: =}\sum_{k=2}^\infty M(\beta_k\nabla^k)$.

Recall the map $\mathcal R:\mathcal W(\Phi)\to\mathfrak s(\Phi)$ defined in Theorem~\ref{cftstw6u} (iii).

\begin{theorem}\label{dsw56w64w3} 
(i) For each $(\alpha(\xi), \beta(\xi))\in\mathcal W(\Phi)$, we have
\begin{equation}\label{xtsw5y4w35}
\mathcal R(\alpha(\xi), \beta(\xi))=\alpha(\nabla)+M(\beta(\nabla)).
\end{equation}

(ii) The Lie bracket $[\cdot, \cdot]$ on $\mathfrak s(\Phi)$ is the commutator of the continuous linear operators. Furthermore, for any $(\alpha^{(m)}(\xi), \beta^{(m)}(\xi))\in\mathcal W(\Phi)$ ($m=1, 2$), we have
\begin{align}
&\big[\alpha^{(1)}(\nabla)+M(\beta^{(1)}(\nabla)), \alpha^{(2)}(\nabla)+M(\beta^{(2)}(\nabla)) \big]\notag\\
&\quad =D_{\beta^{(2)}}\alpha^{(1)}(\nabla)- D_{\beta^{(1)}}\alpha^{(2)}(\nabla)+M\big(D_{\beta^{(2)}}\beta^{(1)}(\nabla)- D_{\beta^{(1)}}\beta^{(2)}(\nabla)\big).\label{bvgfyd}
\end{align}
\end{theorem}

\begin{remark}One can think of commutation relations \eqref{bvgfyd} as a non-trivial consequence of the commutation relations in the Weyl algebra, see \eqref{vtdrsesw3}. 
\end{remark}

\begin{proof} Let $\alpha(\xi)=\sum_{k=1}^\infty\alpha_k\xi^{\otimes k}\in\mathcal W_0(\Phi)$ and $\beta(\xi)=\sum_{k=2}^\infty\beta_k\xi^{\otimes k}\in\mathcal W_1(\Phi)$. Let $V=[V_{ik}]_{i, k\in\mathbb N_0}:=\mathcal R (\alpha(\xi), \beta(\xi))$. Recall that each $V_{ik}$ ($0\le i<k$) is given by~\eqref{vcyrte6ue}. Formulas \eqref{vcrtew654uw} and \eqref{jiy8i} imply, for $\xi\in\Phi$ and $k
\in\mathbb N$, 
\begin{align}
\big(V\langle \cdot^{\otimes k}, \xi^{\otimes k}\rangle\big)(\omega)&=\sum_{i=0}^{k-1}\langle\omega^{\otimes i}, V_{ik}\xi^{\otimes k}\rangle\notag\\
&=\quad\sum_{i=1}^{k-1} (k)_{k-i+1}\, \langle\omega, \beta_{k-i+1}\xi^{\otimes(k-i+1)}\rangle \langle\omega^{\otimes (i-1)}, \xi^{\otimes (i-1)}\rangle \notag\\
&\quad
+\sum_{i=0}^{k-1}(k)_{k-i}\langle\alpha_{k-i}, \xi^{\otimes(k-i)}\rangle\langle\omega^{\otimes i}, \xi^{\otimes i}\rangle\notag\\
&=\sum_{l=2}^{k} (k)_{l}\, \langle\omega, \beta_{l}\xi^{\otimes l}\rangle \langle\omega^{\otimes (k-l)}, \xi^{\otimes (k-l)}\rangle
+\sum_{l=1}^{k}(k)_{l}\langle\alpha_{l}, \xi^{\otimes l}\rangle\langle\omega^{\otimes (k-l)}, \xi^{\otimes (k-l)}\rangle\notag\\
&=\big((M(\beta(\nabla))+\alpha(\nabla))\langle \cdot^{\otimes k}, \xi^{\otimes k}\rangle\big)(\omega).\notag
\end{align}
This proves formula \eqref{xtsw5y4w35}.

That the Lie bracket on $\mathfrak s(\Phi)$ is the commutator of the operators follows from Proposition~\ref{gfstew5}. In view of Proposition~\ref{vgfdstt} (ii) and Theorem~\ref{cftstw6u} (ii), 
we have, for any $(\alpha^{(m)}(\xi), \beta^{(m)}(\xi))\in\mathcal W(\Phi)$ ($m=1, 2$), 
\begin{align}
&\big[\mathcal R(\alpha^{(1)}(\xi), \beta^{(1)}(\xi)), \mathcal R(\alpha^{(2)}(\xi), \beta^{(2)}(\xi))\big]\notag\\
&\quad=\mathcal R\big([(\alpha^{(1)}(\xi), \beta^{(1)}(\xi)), (\alpha^{(2)}(\xi), \beta^{(2)}(\xi)) ]\big).
\label{cftsreara}\end{align}
Now formula~\eqref{bvgfyd} follows from Proposition~\ref{cydsawa} and formulas~\eqref{xtsw5y4w35} and \eqref{cftsreara}.
 \end{proof}
 
 The following corollary is immediate.
 
\begin{corollary}
A $P\in\mathbb M(\Phi)$ is a Sheffer operator if and only if there exist $\alpha(\xi)\in\mathcal W_0(\Phi)$ and $\beta(\xi)\in\mathcal W_1(\Phi)$ such that
$$P=\exp\big[\alpha(\nabla)+M(\beta(\nabla))\big]. $$
\end{corollary}

 \begin{corollary}The following Campbell--Baker--Hausdorff formula holds for all $V^{(1)}, V^{(2)}\in\mathbb V(\Phi)$, in particular, for all 
 $V^{(1)}, V^{(2)}\in\mathfrak s(\Phi)$:
 \begin{align}
 \exp(V^{(1)})\exp(V^{(2)})&=\exp\bigg(V^{(1)}+V^{(2)}+\frac12\, [V^{(1)}, V^{(2)}]\notag\\
 &\quad+\frac1{12}\big([V^{(1)}, [V^{(1)}, V^{(2)}]]-[V^{(2)}, [V^{(1)}, V^{(2)}]]\big)+\cdots\bigg).\label{vcxeET}
 \end{align}
 \end{corollary}
 
 \begin{proof}
 The statement follows from \cite[p.~1013]{Milnor}. Note that in view of formula \eqref{cxzraw4t}, the infinite series in formula \eqref{vcxeET} obviously converges in $\mathbb V(\Phi)$. 
 \end{proof}
 
 \begin{corollary} Let $\Phi=C_\mathrm c(X;\mathbb C)$ be as in Subsection~\ref{vcts5ywiyu}. Then an operator $Q \in \mathcal L(\mathcal P(\Phi')$ belongs to $\mathfrak s(\Phi)$ if and only if $Q$ is as in formula~\eqref{vftyre6i4bhft}. 
 \end{corollary}
 
 \begin{proof} By formula \eqref{dszeaewaeu}, we have, for $f^{(n)}\in\Phi^{\hat\odot n}$, 
\begin{align}
\big( \nabla^k\langle\cdot^{\otimes n}, f^{(n)}\rangle\big)(\omega, x_1, \dots, x_k)&=(n)_k\int_{X^{n-k}}f^{(n)}(x_1, \dots, x_k, \cdot)\, d\omega^{\otimes(n-k)}\notag\\
&=\big(D_{x_1}\dotsm D_{x_k}\langle\cdot^{\otimes n}, f^{(n)}\rangle\big)(\omega).\label{vcyrtds5y}
\end{align} 

Note also that, for $k\ge2$, $(\Phi^{\hat\odot k})'$, can be identified with the subset $M_{\mathrm{sym}}(X^k;\mathbb C)$ of $M(X^k;\mathbb C)$ that consists of $\mathbb C$-valued Radon measures on $X^k$ that remain invariant under the natural action of the symmetric group $S_k$ onto $X^k$. Hence, by \eqref{vcyrtds5y}, for $k\in\mathbb N$, $\alpha_k\in(\Phi^{\hat\odot k})'$ and $p\in\mathcal P(\Phi')$, we have
\begin{equation}\label{vgdsta}
(\alpha_k\nabla p)(\omega)=\int_{X^k}\big(D_{x_1}\dotsm D_{x_k}p\big)(\omega)\, \alpha_k(dx_1\dotsm dx_k).
\end{equation}
Next, for each $k\ge2$, $\beta_k\in\mathcal L(\Phi^{\hat\odot k}, \Phi)$ and $x\in X$, the map $\Phi^{\hat\odot k}\ni f^{(k)}\mapsto(\beta_k f^{(k)})(x)$ belongs to $(\Phi^{\hat\odot k})'=M_{\mathrm{sym}}(X^k;\mathbb C)$. Therefore, by \eqref{vcydtyrdst6} and \eqref{vcyrtds5y}, we have, for $p\in\mathcal P(\Phi')$, 
\begin{equation}\label{cxtsw6u}
\big(M(\beta_k\nabla^k) p\big)(\omega)=\int_X\int_{X^k} (D_{x_1}\dotsm D_{x_k}\, p)(\omega)\, \beta_k(x, dx_1\dotsm dx_k)\, \omega(dx).
\end{equation}
Now the statement follows from Theorem~\ref{dsw56w64w3} (i) and formulas \eqref{vgdsta} and \eqref{cxtsw6u}.
 \end{proof}

\subsection{Remarks on the Riordan group}\label{vgdfr6e64u}

Consider the following bijective (analytic) transformation: 
\begin{equation*}
\mathbb M(\Phi)\ni P=[P_{ik}]_{i, k\in\mathbb N_0}\mapsto \mathfrak IP=\big[(k!/i!)P_{ik}\big]_{i, k\in\mathbb N_0}\in\mathbb M(\Phi).\end{equation*}
Denote by $\mathfrak R(\Phi) $ the image of the set of Sheffer operators, $\mathbb S(\Phi)$, under this transformation.
It easily follows from \eqref{vcrses} and \eqref{fgxtrsr} that a $P=[P_{ik}]_{i, k\in\mathbb N_0}\in\mathbb M(\Phi)$ belongs to $\mathfrak R(\Phi)$ if and only if its usual generating function
is of the form
\begin{equation}\label{vctrdst5sw}
\sum_{k=0}^\infty P^{(k)}(\omega, \xi^{\otimes k})=\frac{A(\xi)}{1-\langle\omega, B(\xi)\rangle}=A(\xi)\sum_{k=0}^\infty \langle\omega, B(\xi)\rangle^k, 
\end{equation}
where $(A(\xi), B(\xi))\in\mathcal S(\Phi)$, compare with \cite{ansh}. By Proposition~\ref{vhydye}, a $P=[P_{ik}]_{i, k\in\mathbb N_0}\in\mathbb M(\Phi)$ belongs to $\mathfrak R(\Phi)$ if and only if, for each $i\in\mathbb N_0$, we have
\begin{equation}\label{cfxsdtserse}
\sum_{k=i}^\infty P_{ik}\xi^{\otimes k}= B(\xi)^{\odot i}A(\xi).\end{equation}

By a trivial modification of the proof of Proposition~\ref{vgfdstt}, one can show that $\mathfrak R(\Phi)$ is a subgroup of $\mathbb M(\Phi)$, and furthermore, the restriction of the map $\mathfrak I$ to $\mathbb S(\Phi)$ (for which we will preserve the notation $\mathfrak I$) establishes a group isomorphism between $\mathbb S(\Phi)$ and $\mathfrak R(\Phi)$.
Let us call $\mathfrak R(\Phi)$ the {\it Riordan group on $\Phi'$}. Indeed, in the case $\Phi=\mathbb C$, condition \eqref{cfxsdtserse} ensures that $\mathfrak R(\mathbb C)$ is the Riordan group\footnote{In fact, the Riordan group, as defined in \cite{Shapiro}, consists of lower-diagonal matrices, hence it coincides with the image of $\mathfrak R(\mathbb C)$ under the transformation $P\mapsto P^T$, where $P^T$ is the transpose of $P$.} defined by Shapiro et al.\ \cite{Shapiro}.  See also \cite[Section~7.3]{ShapiroBook} for the case of a finite-dimensional space $\Phi$.
Furthermore, by analogy with the one-dimensional case, we will call $\mathfrak T(\Phi):=\mathfrak I\big(\mathbb A(\Phi)\big)$ the {\it Toeplitz group} and $\mathfrak L(\Phi):=\mathfrak I\big(\mathbb B(\Phi)\big)$ the {\it Lagrange group on $\Phi'$}. By Proposition~\ref{vgfdstt}~(i), we have $\mathfrak R(\Phi)=\mathfrak T(\Phi)\rtimes\mathfrak L(\Phi)$.

One can easily modify the proof of Theorem~\ref{cftstw6u} to show that a counterpart of this 
theorem holds for the groups $\mathfrak R(\Phi)$, $\mathfrak T(\Phi)$ and $\mathfrak L(\Phi)$. In particular, 
$$\mathfrak r(\Phi):=\log\big(\mathfrak R(\Phi)\big), \quad \mathfrak t(\Phi):=\log\big(\mathfrak T(\Phi)\big), \quad \mathfrak l(\Phi):=\log\big(\mathfrak L(\Phi)\big)$$
are closed vector subspaces of $\mathbb V(\Phi)$, and a counterpart of formula \eqref{vcyrte6ue} takes the form 
\begin{equation}\label{vgxzfxf}
V_{ik}=i(\beta_{k-i+1}\odot\mathbf 1_{i-1})+(\alpha_{k-i}\odot\mathbf1_i), \quad 0\le i<k.
\end{equation}

Next, similarly to \eqref{cxerw5yu}, we define the {\it zero-gradient\,\footnote{This gradient is associated to the $q$-differentiation with $q=0$, see e.g.\ \cite[Section~1]{Kac}.}} $\nabla_{0}$ that satisfies $\nabla_{0}1=0$ and
 $$\big(\nabla_{0}\langle\cdot^{\otimes n}, f^{(n)}\rangle\big)(\omega)=(\omega^{\otimes(n-1)}\odot\mathbf 1_1)f^{(n)}, \quad n\in\mathbb N.$$
 Hence, similarly to \eqref{dszeaewaeu}, we obtain
 \begin{equation*}
 \big(\nabla_{0}^k\langle\cdot^{\otimes n}, f^{(n)}\rangle\big)(\omega)=\chi_{n\ge k}(\omega^{\otimes(n-k)}\odot\mathbf 1_k)f^{(n)}, 
 \end{equation*}
 where $\chi_{n\ge k}$ is equal to 1 if $n\ge k$, and to 0 otherwise.
Then, similarly to Subsection~\ref{xra5aq354q}, we define, for $\alpha_k\in(\Phi^{\hat\odot k})'$ and $\beta_k\in\mathcal L(\Phi^{\hat\odot k}, \Phi)$, linear operators $\alpha_k\nabla_{0}^k$ and $M(\beta_k\nabla_{0}^k)$ from $\mathbb V(\Phi)$:
\begin{align}
\big(\alpha_k\nabla_0^k\langle\cdot^{\otimes n}, \xi^{\otimes n}\rangle\big)(\omega)&=\chi_{n\ge k}\langle\omega^{\otimes(n-k)}, \xi^{\otimes(n-k)}\odot(\alpha_k\xi^{\otimes k})\rangle, \label{hstd}\\
\big(M(\beta_k\nabla_0^k) \langle\cdot^{\otimes n}, \xi^{\otimes n}\rangle\big)(\omega)&=\chi_{n\ge k}
\langle \omega^{\otimes(n-k+1)}, \xi^{\otimes(n-k)}\odot(\beta_k\xi^{\otimes k})\rangle.
\end{align}
 We similarly define operators $\alpha(\nabla_{0})$ and $M(\beta(\nabla_{0}))$ for $\alpha(\xi)=\sum_{k=1}^\infty \alpha_k\xi^{\otimes k}\in\mathcal W_0(\Phi)$ and $\beta(\xi)=\sum_{k=2}^\infty \beta_k\xi^{\otimes k}\in\mathcal W_1(\Phi)$.

Next, using \eqref{jiy8i} with $k=1$ and $\beta_1=\mathbf 1_1$, we obtain, for the usual gradient $\nabla$, the {\it number operator} $N:=M(\nabla)$:
\begin{equation}\label{cxds5y6e5s}
\big(N\langle\cdot^{\otimes n}, f^{(n)}\rangle
\big)(\omega)=n \langle\omega^{\otimes n}, f^{(n)}\rangle, \quad f^{(n)}\in\Phi^{\hat\odot n}, \ n\in\mathbb N_0.
\end{equation}

By using formulas \eqref{vgxzfxf}--\eqref{cxds5y6e5s}, one can now easily derive the following counterpart of Theorem~\ref{dsw56w64w3}.

\begin{theorem}\label{xretsrea45}
Define a map 
$\widetilde{\mathcal I}:\mathcal S(\Phi)\to \mathfrak R(\Phi)$ by $\widetilde{\mathcal I}(A(\xi), B(\xi))=P$, where $P\in\mathfrak R(\Phi)$ has generating function \eqref{vctrdst5sw}.
Define a map $\widetilde{\mathcal R}:\mathcal W(\Phi)\to\mathfrak r(\Phi)$ by $\widetilde{\mathcal R}:=\log(\widetilde{\mathcal I}\operatorname{EXP})$. Then, for $(\alpha(\xi), \beta(\xi))\in\mathcal W(\Phi)$, 
\begin{equation*}
\widetilde{\mathcal R}(\alpha(\xi), \beta(\xi))=\alpha(\nabla_{0})+NM(\beta(\nabla_{0})).
\end{equation*}
Furthermore, for any $(\alpha^{(m)}(\xi), \beta^{(m)}(\xi))\in\mathcal W(\Phi)$ ($m=1, 2$), we have
\begin{align}
&\big[\alpha^{(1)}(\nabla_0)+NM(\beta^{(1)}(\nabla_0)), \alpha^{(2)}(\nabla_0)+NM(\beta^{(2)}(\nabla_0)) \big]\notag\\
&\quad =D_{\beta^{(2)}}\alpha^{(1)}(\nabla_0)- D_{\beta^{(1)}}\alpha^{(2)}(\nabla_0)+NM\big(D_{\beta^{(2)}}\beta^{(1)}(\nabla_0) -D_{\beta^{(1)}}\beta^{(2)}(\nabla_0)\big).\notag
\end{align}
\end{theorem}

\begin{remark}\label{rem:Riordan1dim}
In the case $\Phi=\mathbb C$, formula \eqref{vgxzfxf} provides (after taking the transposition) a full and faithful representation of the Lie algebra $\mathfrak r(\mathbb C)$, reproducing the result of \cite[Theorem~16]{Cheonatal} in the case of monic polynomial sequences. At the same time, Theorem~\ref{xretsrea45} provides improved understanding of  the Lie algebra $\mathfrak r(\mathbb C)$ and the Lie bracket on it. 
\end{remark}

\begin{remark}
It would be interesting to study a generalization of the Riordan group $\mathfrak R(\mathbb C)$ for which the corresponding polynomial sequences of several (or infinitely many) non-commuting variables are as discussed in \cite{ansh}. 
\end{remark}

\subsection*{Acknowledgments}

MJO was supported by the Portuguese national funds through FCT--Funda{\c c}\~ao para a Ci\^encia e a Tecnologia, I.P., within the project UID/04561/2025\\
 \href{https://doi.org/10.54499/UID/04561/2025}{https://doi.org/10.54499/UID/04561/2025}.

\setcounter{equation}{0}
\appendix
\renewcommand{\thesection}{A}
\section{Appendix: Some results on analyticity}

\setcounter{theorem}{0}

We recall the definition of an analytic function acting in locally convex topological vector spaces, \cite{Bochnak} and \cite[Section~3.1]{Dineen}, see also \cite[Section~3]{Milnor} and \cite[Section~2]{Gloeckner}.

Let $V$ and $W$ be l.c.s.'s over $\mathbb C$, and assume that $W$ is complete. A function $f:V\to W$ is called \emph{holomorphic} if it is continuous and, for any $z, v\in V$, the G\^ateaux derivative of $f$ at point $z$ in direction $v$ exists:
\begin{equation*}
f'(z;v)=D_{v}f(z)=\lim_{s\to0}\frac1s\big(f(z+sv)-f(z)\big). 
\end{equation*}
If $f:V\to W$ is holomorphic, then for any $z\in V$, there exists a unique sequence $(p^{(n)})_{n=1}^\infty$, with $p^{(n)}:V\to W$ being a homogeneous polynomial of degree $n$, such that 
$f(z+v)=f(z)+\sum_{n=1}^\infty \frac1{n!}\, p^{(n)}(v)$ for $v\in V$.
If additionally each homogeneous polynomial $p^{(n)}:V\to W$ is continuous, then the function $f$ is called {\it analytic}. 

Assume that a function $f:V\to W$ is holomorphic and the G\^ateaux derivative $V^2\ni(z, v)\mapsto f'(z;v)\in W$ is a continuous function. Then, by \cite[Assertion~3.10]{Milnor}, all higher G\^ateaux derivatives $V^{n+1}\ni(z;v_1, \dots, v_n)\mapsto f^{(n)}(z;v_1, \dots, v_n)\in W$ exist and are continuous functions. Hence, in this case, the homogeneous polynomials $p_n(v)=f^{(n)}(z;v, \dots, v)$ are continuous functions of $v$, and so $f$ is analytic. 

\begin{lemma}\label{tfy7e6e}
Let $(V_l)_{l=1}^\infty$ and $(W_k)_{k=1}^\infty$ be two sequences of complete complex l.c.s.'s. Let $V:=\prod_{l=1}^\infty V_l$ and $W:=\prod_{k=1}^\infty W_k$ be the topological products of the spaces $V_l$ and $W_k$, respectively. For $r\in\mathbb N$ 
and $1\le l_1<l_2<\dots<l_r$, denote by $\operatorname{Pr}^V_{l_1, \dots, l_r}$ the projection of $V$ onto $V_{l_1}\times V_{l_2}\times \dots\times V_{l_r}$, and for $k\in\mathbb N$, denote by $\operatorname{Pr}^W_k$ the projection of $W$ onto $W_k$. Let a map $\psi:V\to W$ be such that, for each $k\in\mathbb N$, there exist $r\in\mathbb N$, $1\le l_1<l_2<\dots<l_r$, and a map $\psi^{(k)}_{l_1, \dots, l_r}:V_{l_1}\times \dots\times V_{l_r}\to W_k$ such that
\begin{equation*}
\operatorname{Pr}^W_k\psi= \psi^{(k)}_{l_1, \dots, l_r}\operatorname{Pr}^V_{l_1, \dots, l_r}.
\end{equation*}
 Then the map $\psi$ is analytic if and only if, for each $k\in\mathbb N$, the map 
$\psi^{(k)}_{l_1, \dots, l_r}$
is analytic.
\end{lemma}

\begin{proof}
The lemma easily follows from the fact that, in a product space, say $W=\prod_{k=1}^\infty W_k$, any open set is a union of open sets of the form $\prod_{k=1}^\infty O_k$, where each $O_k$ is open in $W_k$ and only finitely many sets $O_k$ are not equal to $W_k$.
\end{proof}

Everywhere below $\Phi=\indlim_{n\to\infty}\Phi_n$ is an (LB)-space.

\begin{lemma}\label{gctessswe5} Let $i, j\in\mathbb N$ and $k\in\mathbb N_0$.

(i) The map 
\begin{equation}\label{cxdraqy}
\mathcal L(\Phi^{\hat\odot j}, \Phi^{\hat\odot k})\times \mathcal L(\Phi^{\hat\odot i}, \Phi^{\hat\odot j})\ni(A_1, A_2)\mapsto A_1A_2\in 
 \mathcal L(\Phi^{\hat\odot i}, \Phi^{\hat\odot k})
 \end{equation}
is analytic. 

(ii) The map 
\begin{align*}
&C([0, 1]; \mathcal L(\Phi^{\hat\odot j}, \Phi^{\hat\odot k}))\times C([0, 1];\mathcal L(\Phi^{\hat\odot i}, \Phi^{\hat\odot j}))\ni (A_1(\cdot), A_2(\cdot))\\
&\qquad \mapsto A_1(\cdot)A_2(\cdot)\in C([0, 1]; \mathcal L(\Phi^{\hat\odot i}, \Phi^{\hat\odot k}))
\end{align*}
is analytic. 
\end{lemma}

\begin{remark}\label{gcyd6ue6} Since the composition of analytic maps is analytic, the results of Lemma~\ref{gctessswe5} admit an immediate extension to the case of the product of more than two operators or continuous functions, respectively. 
\end{remark}

\begin{proof}[Proof of Lemma~\ref{gctessswe5}] (i) Denote the map in formula \eqref{cxdraqy} by $f$. First, let us prove that $f$ is continuous. Since 
$\mathcal L(\Phi^{\hat\odot i}, \Phi^{\hat\odot k})=\prlim_{n\to\infty}\mathcal L(\Phi_n^{\hat\odot i}, \Phi^{\hat\odot k})$, 
by e.g.\ \cite[Ch.~II, \S~5.2]{Schaefer}, it is sufficient to prove that, for each $n\in\mathbb N$, the map 
\begin{equation}\label{vctsea}
\mathcal L(\Phi^{\hat\odot j}, \Phi^{\hat\odot k})\times \mathcal L(\Phi^{\hat\odot i}, \Phi^{\hat\odot j})\ni(A_1, A_2)\mapsto (A_1A_2)\restriction_{\Phi_n^{\hat\odot i}}=A_1\big(A_2\restriction_{\Phi_n^{\hat\odot i}}\big) \in 
 \mathcal L(\Phi_n^{\hat\odot i}, \Phi^{\hat\odot k})
 \end{equation}
 is continuous. The map $ \mathcal L(\Phi^{\hat\odot i}, \Phi^{\hat\odot j})\ni A_2\mapsto A_2\restriction_{\Phi_n^{\hat\odot i}}\in \mathcal L(\Phi_n^{\hat\odot i}, \Phi^{\hat\odot j})$ 
 is continuous by the definition of the projective topology. Hence, to prove the continuity of the map in~\eqref{vctsea}, it is sufficient to show that the following map is continuous: 
\begin{equation}\label{rdresdses}
\mathcal L(\Phi^{\hat\odot j}, \Phi^{\hat\odot k})\times \mathcal L(\Phi_n^{\hat\odot i}, \Phi^{\hat\odot j})\ni(A_1, A_2)\mapsto f_n(A_1, A_2)=A_1A_2
\in \mathcal L(\Phi_n^{\hat\odot i}, \Phi^{\hat\odot k}).
 \end{equation}
 Since the map $f_n$ is bilinear, by e.g.\ \cite[Ch.~III, \S~5]{Schaefer}, it is sufficient to prove that $f_n$ is continuous at $(0, 0)$. Since
 $ \mathcal L(\Phi_n^{\hat\odot i}, \Phi^{\hat\odot k})= \indlim_{m\to\infty}\mathcal L(\Phi_n^{\hat\odot i}, \Phi_m^{\hat\odot k})$, 
 for each $V_n$, a neighborhood of zero in $ \mathcal L(\Phi_n^{\hat\odot i}, \Phi^{\hat\odot k})$, one can find a sequence $(r_m)_{m=n}^\infty$ with $r_m>0$ and $r_m\ge r_{m+1}$ for all $m\ge n$, and such that 
\begin{equation}\label{cfxa4S}
O_n=\operatorname{conv} \bigcup_{m=n}^\infty O_{n, m}\subset V_n.
\end{equation}
 Here, 
 \begin{equation}\label{vcgftdst}
 O_{n, m}= \big\{A\in \mathcal L(\Phi_n^{\hat\odot i}, \Phi_m^{\hat\odot k})\mid \|A\|_{ \mathcal L(\Phi_n^{\hat\odot i}, \Phi_m^{\hat\odot k})}<r_m\big\}
 \end{equation}
 and for a set $\mathcal A$, $\operatorname{conv} \mathcal A$ denotes the convex hull of $\mathcal A$. We obviously have 
\begin{equation}\label{vcxzrgew44}
\mathcal L(\Phi^{\hat\odot j}, \Phi^{\hat\odot k})\times \mathcal L(\Phi_n^{\hat\odot i}, \Phi^{\hat\odot j})= \indlim_{l\to\infty}
 \big(\mathcal L(\Phi^{\hat\odot j}, \Phi^{\hat\odot k})\times \mathcal L(\Phi_n^{\hat\odot i}, \Phi_l^{\hat\odot j})\big).\end{equation}
 For any $m\ge l\ge n$, we define sets
 \begin{align*}
U_{l, m}:=&\big\{A_1\in \mathcal L(\Phi_{l}^{\hat\odot j}, \Phi_{m}^{\hat\odot k})\mid \|A_1\|_{ \mathcal L(\Phi_{l}^{\hat\odot j}, \Phi_{m}^{\hat\odot k})}<r_{m}\big\}, \quad U_l:=\operatorname{conv} \bigcup_{m=l}^\infty U_{l, m}, \\
\mathcal V_{n, l}:=&\big\{A_2\in \mathcal L(\Phi_n^{\hat\odot i}, \Phi_{l}^{\hat\odot j})\mid \|A_2\|_{\mathcal L(\Phi_n^{\hat\odot i}, \Phi_{l}^{\hat\odot j})}\le 1\big\} 
 \end{align*} 
 Note that the set
$U_l$ is open in $\mathcal L(\Phi_l^{ \hat\odot j}, \Phi^{\hat\odot k})$. Further define
\begin{align*}
 {\mathcal U}_{l, m}:=&\big\{A_1\in \mathcal L(\Phi^{\hat\odot j}, \Phi^{\hat\odot k})\mid A_1\restriction_{\Phi_l^{\hat\odot j}}\in U_{l, m}\big\}, \\
{\mathcal U}_l:=&\operatorname{conv} \bigcup_{m=l}^\infty {\mathcal U}_{l, m}=\big\{A_1\in \mathcal L(\Phi^{\hat\odot j}, \Phi^{\hat\odot k})\mid A_1\restriction_{\Phi_l^{\hat\odot j}}\in U_l\big\}.
\end{align*}
By the definition of the projective limit topology, the set ${\mathcal U}_l$ is open in $\mathcal L(\Phi^{\hat\odot j}, \Phi^{\hat\odot k})$. But then, by \eqref{vcxzrgew44}, the set
\begin{equation}\label{vctw6u}
\mathcal O_n:=\operatorname{conv}\bigcup_{l=n}^\infty\, {\mathcal U}_l\times \mathcal V_{n, l}=\operatorname{conv}\bigcup_{l=n}^\infty\bigg(\operatorname{conv} \bigcup_{m=l}^\infty {\mathcal U}_{l, m}\bigg) \times \mathcal V_{n, l} 
\end{equation}
is open in $\mathcal L(\Phi^{\hat\odot j}, \Phi^{\hat\odot k})\times \mathcal L(\Phi_n^{\hat\odot i}, \Phi^{\hat\odot j})$. For any
$(A_1, A_2)\in \mathcal U_{l, m}\times \mathcal V_{n, l}$, we have $A_1A_2=(A_1\restriction_{\Phi_l^{\hat\odot j}})A_2\in \mathcal L(\Phi_n^{\hat\odot i}, \Phi_{m}^{\hat\odot k})$ and $\|A_1A_2\|_{\mathcal L(\Phi_n^{\hat\odot i}, \Phi_{m}^{\hat\odot k})}<r_m$. Hence, 
\begin{equation}\label{bufr7u}
f_n( \mathcal U_{l, m}\times \mathcal V_{n, l})\subset O_{n, m}.
\end{equation}
Since the map $f_n$ is bilinear, formulas \eqref{cfxa4S}, \eqref{vctw6u} and \eqref{bufr7u}
imply that $f_n(\mathcal O_n)\subset O_n$. 
Thus, the continuity of $f_n$ is proven, and so $f$ is continuous. 
 
Next, for any $A=(A_1, A_2), B=(B_1, B_2)\in \mathcal L(\Phi^{\hat\odot j}, \Phi^{\hat\odot k})\times \mathcal L(\Phi^{\hat\odot i}, \Phi^{\hat\odot j})$, we trivially have
$$f'(A;B)= \lim_{s\to0}\big(A_1B_2+B_1A_2+sB_1B_2\big)=A_1B_2+B_1A_2.$$
Note that $f'(A;\cdot)$ is linear. The continuity of the function $f'(A;B)$ in $(A, B)$ follows immediately from the continuity of~$f$.

(ii) By using Lemma~\ref{vcrswu5}, one can prove statement (ii) completely analogously to the proof of statement (i). 
\end{proof}

\begin{lemma}\label{cgtsrta} Let $i_1, i_2\in\mathbb N$, $j_1, j_2\in\mathbb N_0$ and $l_1, l_2\in\mathbb N$. We have:

(i) The map
\begin{equation}\label{vcfresar5y67}
\mathcal L(\Phi^{\hat\odot i_1}, \Phi^{\hat\odot j_1})\times \mathcal L(\Phi^{\hat\odot i_2}, \Phi^{\hat\odot j_2})\ni (A_1, A_2)\mapsto
A_1^{\odot l_1}\odot A_2^{\odot l_2}\in\mathcal L(\Phi^{\hat\odot(i_1l_1+i_2l_2)}, \Phi^{\hat\odot(j_1l_1+j_2l_2)})\end{equation}
is analytic.

(ii) The map
\begin{align*}
&C([0, 1];\mathcal L(\Phi^{\hat\odot i_1}, \Phi^{\hat\odot j_1}))\times C([0, 1];\mathcal L(\Phi^{\hat\odot i_2}, \Phi^{\hat\odot j_2}))\ni (A_1(\cdot), A_2(\cdot))\\
&\qquad \mapsto
A_1(\cdot)^{\odot l_1}\odot A_2(\cdot)^{\odot l_2}\in C([0, 1];\mathcal L(\Phi^{\hat\odot(i_1l_1+i_2l_2)}, \Phi^{\hat\odot(j_1l_1+j_2l_2)}))
\end{align*}
is analytic.
\end{lemma}

\begin{remark}\label{cxfet678i9}
Similarly to Remark~\ref{gcyd6ue6}, the result of Lemma~\ref{cgtsrta} admits an extension to the symmetric tensor product of more than two symmetric powers of operators or functions, respectively.
\end{remark}

\begin{proof}[Proof of Lemma~\ref{cgtsrta}] We will only discuss the proof of part (i). 
If $l_1=l_2=1$, then, we can prove the statement by slightly modifying the proof of Lemma~\ref{gctessswe5} (i). Indeed, denote by $g$ the map in formula \eqref{vcfresar5y67} 
with $l_1=l_2=1$. Similarly to \eqref{rdresdses}, we conclude that, in order to show the continuity of $g$, it is sufficient to prove the continuity of the bilinear map
$$\mathcal L(\Phi_n^{\hat\odot i_1}, \Phi^{\hat\odot j_1})\times \mathcal L(\Phi_n^{\hat\odot i_2}, \Phi^{\hat\odot j_2})\ni (A_1, A_2)\mapsto
g_n(A_1, A_2)=A_1\odot A_2\in\mathcal L(\Phi_n^{\hat\odot(i_1+i_2)}, \Phi^{\hat\odot(j_1+j_2)})$$
for each $n\in\mathbb N$. Similarly to \eqref{cfxa4S}, \eqref{vcgftdst}, we consider a set of the form
\begin{align*}
O_n&=\operatorname{conv} \bigcup_{m=n}^\infty O_{n, m}, \\
O_{n, m}&=\big\{A\in \mathcal L(\Phi_n^{\hat\odot (i_1+i_2)}, \Phi_m^{\hat\odot (j_1+j_2)})\mid \|A\|_{\mathcal L(\Phi_n^{\hat\odot (i_1+i_2)}, \Phi_m^{\hat\odot (j_1+j_2)})}<r_m\big\}.
\end{align*}
Similarly to \eqref{vcxzrgew44}, we have
$$\mathcal L(\Phi_n^{\hat\odot i_1}, \Phi^{\hat\odot j_1})\times \mathcal L(\Phi_n^{\hat\odot i_2}, \Phi^{\hat\odot j_2})=\indlim_{m\to\infty}
\big(\mathcal L(\Phi_n^{\hat\odot i_1}, \Phi_m^{\hat\odot j_1})\times \mathcal L(\Phi_n^{\hat\odot i_2}, \Phi_m^{\hat\odot j_2})\big).$$
Define
\begin{align*}
\mathcal U_{n, m}:=&\big\{A_1\in \mathcal L(\Phi_n^{\hat\odot i_1}, \Phi_m^{\hat\odot j_1})\mid \big\|A_1\|_{\mathcal L(\Phi_n^{\hat\odot i_1}, \Phi_m^{\hat\odot j_1})}<r_m\}, \\ 
\mathcal V_{n, m}:=&\big\{A_2\in \mathcal L(\Phi_n^{\hat\odot i_2}, \Phi_m^{\hat\odot j_2})\mid \big\|A_2\|_{\mathcal L(\Phi_n^{\hat\odot i_2}, \Phi_m^{\hat\odot j_2})}<1\}, \\
\mathcal O_n:=&\operatorname{conv}\bigcup_{m=n}^\infty \mathcal U_{n, m}\times \mathcal V_{n, m}.
\end{align*}
Then, $g_n(\mathcal U_{n, m}\times \mathcal V_{n, m})\subset O_{n, m}$, and so $g_n(\mathcal O_n)\subset O_n$. Hence, the map $g_n$ is continuous. The analyticity of the map $g$ can now be easily shown.

To prove the statement of part (i) for general $l_1$ and $l_2$, it is now sufficient to show that, for any $i\in\mathbb N$, $j\in\mathbb N_0$ and $l\ge2$, the map
$$\mathcal L(\Phi^{\hat\odot i}, \Phi^{\hat\odot j})\ni A\mapsto A^{\odot l}\in \mathcal L(\Phi^{\hat\odot il}, \Phi^{\hat\odot jl})$$
is analytic. To this end, we represent this map as a composition of two analytic maps. First, we note that the `diagonal' map 
$$\mathcal L(\Phi^{\hat\odot i}, \Phi^{\hat\odot j})\ni A\mapsto (\underbrace{A, \dots, A}_{\text{$l$ times}})\in \mathcal L(\Phi^{\hat\odot i}, \Phi^{\hat\odot j})^l$$
is obviously analytic.
Next, since the map $g$ is analytic, the following map is also analytic:
$$ \mathcal L(\Phi^{\hat\odot i}, \Phi^{\hat\odot j})^l\ni(A_1, A_2, \dots, A_l)\mapsto A_1\odot A_2\odot\dots\odot A_l\in \mathcal L(\Phi^{\hat\odot il}, \Phi^{\hat\odot jl}).$$
From this the lemma follows.
\end{proof}

\begin{lemma}\label{ctrsw6u4}
For any $i\in\mathbb N$ and $j\in\mathbb N_0$, the following map is analytic:
\begin{equation}\label{dstew53w5}
C([0, 1];\mathcal L(\Phi^{\hat\odot i}, \Phi^{\hat\odot j}))\ni A(\cdot)\mapsto \int_0^{\bullet} A(r)\, dr\in C([0, 1];\mathcal L(\Phi^{\hat\odot i}, \Phi^{\hat\odot j})), \end{equation}
which in turn implies analyticity of the map
$$
C([0, 1];\mathcal L(\Phi^{\hat\odot i}, \Phi^{\hat\odot j}))\ni A(\cdot)\mapsto \int_0^1 A(r)\, dr\in \mathcal L(\Phi^{\hat\odot i}, \Phi^{\hat\odot j}). $$
\end{lemma}

\begin{proof} Similarly to the proofs of Lemmas~\ref{gctessswe5} and \ref{cgtsrta}, we easily see that the statement will follow if we prove the continuity of the map in \eqref{dstew53w5}. In view of formula \eqref{fstsw5u}, and similarly to \eqref{rdresdses}, it is sufficient to prove that, for each $n\in\mathbb N$, the following map is continuous:
\begin{equation}\label{vctrs6ue}
C([0, 1];\mathcal L(\Phi_n^{\hat\odot i}, \Phi^{\hat\odot j}))\ni A(\cdot)\mapsto \int_0^\bullet A(r)\, dr\in C([0, 1];\mathcal L(\Phi_n^{\hat\odot i}, \Phi^{\hat\odot j})).\end{equation}
Since the map in \eqref{vctrs6ue} is linear, by formula \eqref{fxdste5swu5} and e.g.\ \ \cite[Ch.~II, \S~6.1]{Schaefer}, it is sufficient to prove that, for each $k\in\mathbb N$, the following map is continuous:
$$C([0, 1];\mathcal L(\Phi_n^{\hat\odot i}, \Phi_k^{\hat\odot j}))\ni A(\cdot)\mapsto \int_0^\bullet A(r)\, dr\in C([0, 1];\mathcal L(\Phi_n^{\hat\odot i}, \Phi_k^{\hat\odot j})).$$
But this continuity immediately follows from the estimate 
\[\sup_{t\in[0, 1]}\bigg\|\int_0^t A(r)\, dr\bigg\|_{\mathcal L(\Phi_n^{\hat\odot i}, \Phi_k^{\hat\odot j})}\le \int_0^1\|A(r)\|_{\mathcal L(\Phi_n^{\hat\odot i}, \Phi_k^{\hat\odot j})}\, dr\le \sup_{t\in[0, 1]}\|A(t)\|_{\mathcal L(\Phi_n^{\hat\odot i}, \Phi_k^{\hat\odot j})}.\qedhere\]
\end{proof}

\renewcommand{\thesection}{B}

\section{Appendix: (Infinite-dimensional) Lie structures with a global parametrization}

We will now briefly recall some constructions of Milnor \cite{Milnor} in the special case of a Lie group with a global parametrization over a complete l.c.s. 

\subsection{Lie group with a global parametrization}\label{app:LieGroup}

Let $V$ be a complete l.c.s.\ over $\mathbb C$, let $G$ be a set, and let $\varphi:V\to G$ be a bijective map. The map $\varphi$ defines a topology on $G$ which makes $\varphi$ into a homeomorphism. Then $G$ becomes a (single-chart smooth) {\it manifold modeled on $V$}, and $\varphi$ is called a {\it global parametrization of $G$}.

Let $\psi:V\to G$ be another bijective map such that both maps $\varphi^{-1}\circ \psi:V\to V$ and $\psi^{-1}\circ\varphi:V\to V$ are analytic. Then $\psi$ provides another global parametrization of $G$. 

We will say that a map $f:G\to G$ is analytic if the map $\varphi^{-1}\circ f\circ\varphi: V\to V$ is analytic. Similarly, for a complete l.c.s.\ $W$ over $\mathbb C$, one defines analytic maps from $G$ to $W$ and from $W$ to $G$.

Now assume additionally that $(G,\cdot)$ is a group such that the group multiplication $G^2\ni(g,h)\mapsto g\cdot h\in G$ and the inversion map $G\ni g \mapsto g^{-1}\in G$ are analytic maps. Then $G$ is called a {\it Lie group modeled on $V$}.  

Note that, if $G$ is a Lie group, then $G$ is, of course, a topological group.

\subsection{Tangent vectors and tangent spaces}\label{app:TangentVectors}

Let $I$ be an open subset of $\R$. We will say that  $\gamma:I\to G$ is a smooth ($G$-valued) curve if the function  $I\ni t\mapsto \varphi^{-1}(\gamma(t))\in V$ is smooth. We will use the standard notation $\gamma'(t):=\frac{d}{dt}\,\varphi^{-1}(\gamma(t))$.

Assume $0\in I$,  let $\gamma:I\to G$ be a smooth curve,  and denote $g:=\gamma(0)\in G$. The vector $\gamma'(0)\in V$ is then called the \emph{tangent vector} of $\gamma$ at $g$. 
For a fixed $g\in G$, two smooth curves $\gamma,\tilde{\gamma}:I\to G$ with $\gamma(0)=\tilde{\gamma}(0)=g\in G$ are called {\it equivalent} if they have equal tangent vectors at $g$.
 The set of all corresponding equivalence classes of curves is called the \emph{tangent space} at $g$, and is denoted by $T_{g}G$. For each class from $T_{g}G$, we can, hence, consider the corresponding tangent vector from $V$. Conversely, for each $v\in V$, we can define a smooth curve $\gamma:I\to G$ by setting 
$  \gamma(t):=\varphi(tv+\varphi^{-1}(g))$ ($t\in I$), whose tangent vector at $g$ is obviously $v$. Thus, we have a bijection between the tangent space $T_{g}G$ and the space $V$, which makes each $T_{g}G$   a l.c.s.\ isomorphic to $V$;  below we will often identify $T_gG$ with $V$, writing $T_gG \cong V$. Note, however, that this identification of $T_gG$ with $V$ depends on the parametrization $\varphi$.

\subsection{Lie algebra and Lie bracket}\label{app:LieAlgebra}
Let $e$ be the identity element of the Lie group $G$. Without loss of generality, we can assume that $\varphi(0)=e$. (Otherwise, we can consider $\tilde{\varphi}:V\to G$ defined by $\tilde{\varphi}(v):=\varphi(v+\varphi^{-1}(e))$,  which is evidently a  global parametrization of $G$.) For each $v, w\in T_e G\cong V$, 
consider the function 
\begin{equation}\label{awawfxdtrsre}
h_{v, w}(s_1, s_2):=\varphi^{-1}\big(\varphi(s_1v)\cdot\varphi(s_2w)\big)\in V,\quad s_1, s_2\in\mathbb C.
\end{equation}
 Since the group multiplication is analytic, the function $h_{v, w}:\mathbb C^2\to V$ is also analytic. Hence, we can consider its Taylor expansion at $(0, 0)$, which is given by (note that $h_{v, w}(0, 0)=0$):
\begin{equation}\label{qqqvcxfdxdz}
h_{v, w}(s_1, s_2)=s_1 v + s_2w +s_1s_2b(v, w)+\cdots
\end{equation}
with $b(v, w)\in V$. We now introduce a product operation on $T_eG\cong V$:
\begin{equation}\label{cftes56u7}
[v, w]:=b(v, w)-b(w, v)\in V, \quad v, w\in V .\end{equation}
 Then $\mathfrak{g}:=T_eG$ is called the \emph{Lie algebra of the Lie group $G$}, and $[\cdot, \cdot]$ is called the \emph{Lie bracket on $\mathfrak g$}. (In particular, if the group $G$ is abelian, the corresponding Lie bracket is identically zero.)

Let $\psi:V\to G$ be another parametrization of $G$, with $\psi(0)=e$. Assume that, for each $v\in V$,
\begin{equation}\label{cxts5sw5y}
\frac d{dt}\Big|_{t=0}\varphi^{-1}(\psi(tv))=v,\quad v\in V,
\end{equation}
i.e., the derivative of the map $\varphi^{-1}\circ\psi:V\to V$ at zero is the identity map in $V$. Then we will say that the map $\psi$ provides an {\it equivalent (global) parametrization of} $G$. 

Formula \eqref{cxts5sw5y} implies that, for each smooth curve $\gamma:I\to G$ with $\gamma(0)=e$, we have $\frac d{dt}\big|_{t=0}\varphi^{-1}(\gamma(t))=\frac d{dt}\big|_{t=0}\psi^{-1}(\gamma(t))$. Hence, the identification of $\mathfrak g$ with $V$ under the equivalent parametrization $\psi$ is the same as  the identification of $\mathfrak g$ with $V$ under $\varphi$. This implies that the Lie bracket, considered as a map $V^2\ni(v,w)\mapsto [v,w]\in V$, does not change if we use the equivalent parametrization $\psi$ instead of $\varphi$. 

\subsection{Exponential map}\label{app:ExpMap}

A {\it one-parameter subgroup of $G$} is a smooth curve $\gamma:\mathbb R\to G$ that satisfies $\gamma(0)=e$ and $\gamma(s+t)=\gamma(s)\cdot\gamma(t)$ for all $s,t\in\mathbb R$. 

Assume that, for each $v\in \mathfrak g$, there exists a one-parameter subgroup $\gamma_v(\cdot)$ of $G$ that satisfies $\gamma_v'(0)=v$. Then we can define the {\it exponential map} $\operatorname{Exp}:\mathfrak{g}\to G$ by $\operatorname{Exp}(v):=\gamma_v(1)$. Moreover, it can be checked that $\operatorname{Exp}(tv)=\gamma_{tv}(1)=\gamma_v(t)$, $t\in \mathbb R$.

In the lemma below, we use the identification $\mathfrak g\cong V$ under parametrization $\varphi$.

\begin{lemma}\label{analytic_reparametrization}
Assume that the exponential map $\operatorname{Exp}:V\to G$  exists and is a bijective map. Denote by $\operatorname{Log}:G\to V$ the inverse map of $\operatorname{Exp}$, and assume that both maps $\operatorname{Exp}$ and $\operatorname{Log}$ are analytic. Then the exponential map $\operatorname{Exp}:V\to G$  provides an equivalent (global) parametrization of $G$. Hence, the Lie bracket, considered as a map $V^2\ni(v,w)\mapsto [v,w]\in V$, does not change if we use the parametrization $\operatorname{Exp}$ instead of $\varphi$.
\end{lemma}

\begin{proof} It follows from the conditions of the lemma that the exponential map $\operatorname{Exp}$ provides a parametrization of $G$. Hence, we only need to check that formula \eqref{cxts5sw5y} holds for $\psi=\operatorname{Exp}$. 
Let $v\in V$, and let $\gamma_v(\cdot)$ be the corresponding one-parameter subgroup of $G$.  Then $\operatorname{Exp}(tv)=\gamma_v(t)$ and $\frac{d}{dt}\big|_{t=0}\varphi^{-1}(\gamma_v(t))=v$ by the definition of $\gamma_v(\cdot)$.
\end{proof}

\subsection{Regular Lie group}\label{app:RegLieGroup}
For each $g\in G$ and $v\in\mathfrak g\cong V$, we define $g\cdot v \in T_gG\cong V$ by\footnote{For $g\in G$ and $v\in\mathfrak g$, $g\cdot v\in T_gG$ is a special case of the product in the tangent bundle $TG$ considered as a Lie group, see \cite[Section~5]{Milnor}. 
} $g\cdot v :=\frac d{ds}\big|_{s=0}\gamma(s)$, where $\gamma(s):=g\cdot \varphi(sv)$. 

Assume that, for each smooth curve $\gamma:[0,1]\to \mathfrak g\cong V$, there exists a  smooth curve $\eta:[0,1]\to G$ such that $\eta'(t)=\eta(t)\cdot\gamma(t)$ for each $t\in[0,1]$, and the map
$C^\infty([0,1];\mathfrak g)\ni\gamma\mapsto \eta(1)\in G$ is analytic. Then one says that the Lie group $G$ is {\it regular}  (see \cite[Section~7]{Milnor} or \cite[Section~5]{Gloeckner} for further details).

Note that the condition of regularity is stronger than the condition of existence of the analytic exponential map $\operatorname{Exp}:V\to G$.   

\subsection{Embedded Lie subgroup}\label{app:EmbLieSubGroup}

Assume that the conditions of Lemma~\ref{analytic_reparametrization} are satisfied. Let $W$ be a  closed subspace of $V$, and assume that $H:=\operatorname{Exp}(W)$ is a subgroup of $G$. Then one says that $H$ is an {\it embedded Lie subgroup of $G$}.  See \cite[Section~8]{Gloeckner} for further details.

Note that $H$ is itself a Lie group modeled on $W$ through the global parametrization $\operatorname{Exp}\restriction_W:W\to H$.

\end{document}